\documentclass[10pt, DIV18]{scrartcl}
\usepackage{latexsym,amsmath,amssymb,graphics,graphicx,amscd,amsfonts}
\usepackage[thmmarks,amsmath]{ntheorem}
\usepackage[arrow, matrix, curve]{xy}
\usepackage[authoryear,round]{natbib}
\usepackage{eucal}
\usepackage[authoryear,round]{natbib}
\newtheorem{theorem}{Theorem}[section]
\newtheorem{definition}[theorem]{Definition}
\newtheorem{lemma}[theorem]{Lemma}
\newtheorem{proposition}[theorem]{Proposition}
\newtheorem{corollary}[theorem]{Corollary}
\theorembodyfont{\normalfont}
\theoremsymbol{\ensuremath{\lozenge}}
\newtheorem{example}[theorem]{Example}
\theoremsymbol{\ensuremath{\triangle}}
\newtheorem{remark}[theorem]{Remark}
\theoremstyle{nonumberplain}
\theoremsymbol{\ensuremath{\square}}
\theoremheaderfont{\scshape}
\theoremseparator{:}
\theorembodyfont{\normalfont}
\newtheorem{proof}{Proof}
\usepackage{dsfont}

\newcommand{\V}{{\mathcal{V}}}
\newcommand{\K}{{\mathcal{K}}}

\newcommand{\R}{\mathbb{R}}

\newcommand{\smoo}{C^\infty(G)}
\newcommand{\mx}{\mathfrak{X}}
\newcommand{\lie}[1]{\mathfrak{#1}}
\newcommand{\dr}{\mathbf{d}}

\newcommand{\ldr}[1]{{{\pounds}}_{#1}}
\newcommand{\ip}[1]{{\mathbf{i}}_{#1}}
\newcommand{\an}[1]{\arrowvert_{#1}}

\newcommand{\poi}[1]{\{#1\}}
\newcommand{\pois}{\poi{\cdot\,,\cdot}}

\newcommand{\inv}{^{-1}}
\newcommand{\rr}{\rightrightarrows}

\DeclareMathOperator{\dom}{Dom}
\DeclareMathOperator{\erz}{span}
\DeclareMathOperator{\Ad}{Ad}
\DeclareMathOperator{\ad}{ad}
\DeclareMathOperator{\graph}{Graph}
\DeclareMathOperator{\rad}{rad}

\DeclareFontFamily{U}{matha}{\hyphenchar\font45}
\DeclareFontShape{U}{matha}{m}{n}{
      <5> <6> <7> <8> <9> <10> gen * matha
      <10.95> matha10 <12> <14.4> <17.28> <20.74> <24.88> matha12
      }{}
\DeclareSymbolFont{matha}{U}{matha}{m}{n}
\DeclareFontSubstitution{U}{matha}{m}{n}
\DeclareMathSymbol{\operp}         {2}{matha}{"6B}

\allowdisplaybreaks[1]

\begin{document}
\title{Dirac Lie groups, Dirac homogeneous spaces and the theorem of
Drinfel$'$d.}
\author{Madeleine Jotz\footnote{\textbf{Madeleine Jotz}, Section de Math{\'e}matiques, 
 Ecole Polytechnique
   F{\'e}d{\'e}rale de Lausanne,
 1015 Lausanne, Switzerland
 (madeleine.jotz@epfl.ch). Supported by a swiss NSF font.}}
\date{}

\maketitle

 \begin{abstract}
\centerline{\textbf{Abstract}}
The notions of \emph{Poisson Lie group} and \emph{Poisson homogeneous space}
are extended to the Dirac category. The theorem of Drinfel$'$d (\cite{Drinfeld93})
on the one-to-one correspondence between Poisson homogeneous spaces
of a Poisson Lie group and a special class of
Lagrangian subalgebras of the Lie bialgebra associated
to the Poisson Lie group is proved to hold in this more general setting.
 \end{abstract}


\tableofcontents

\section{Introduction}
A \emph{Poisson Lie group}  is a Lie group  endowed with a
 Poisson structure that is compatible with the Lie group structure.
 Poisson Lie groups were introduced 
 by \cite{Drinfeld83a} and studied by \cite{Semenov85}. Their aim was
to understand the Hamiltonian structure of the group of dressing
transformations of a completely integrable system. 
The study of the  geometry of Poisson Lie
groups was started with the works of  Lu and Weinstein 
(see \cite{Lu90}, \cite{LuWe90}, \cite{LuWe89} among others).
The notion of Poisson Lie group was generalized to the notion
of Poisson Lie groupoids by \cite{Weinstein88b}.

A \emph{Poisson homogeneous space}  of a Poisson Lie group
 is a homogeneous space of the Lie group that is endowed with a Poisson structure such that
the left action of the Lie group on the homogeneous space 
is a Poisson map.
 Poisson homogeneous spaces of  Poisson Lie groups
are in correspondence with suitable subspaces of the direct sum of the Lie
algebra with its dual. We show that this correspondence result
fits in a  more general and natural context: the one of Dirac manifolds,
which are objects generalizing in a sense the Poisson manifolds.

\medskip

Let $G$ be a  Lie group endowed with a  bivector field $\pi_G\in\Gamma(\bigwedge^2TM)$.
The bivector field $\pi_G$ is 
\emph{multiplicative}, and $(G,\pi_G)$ is a \emph{Poisson Lie group},
if $\pi_G$ satisfies 
\[\pi_G(gh)=T_gR_h\pi_G(g)+T_hL_g\pi_G(h)
\] 
for all $g,h\in G$, where $R_h$ is the right multiplication by $h$ on $G$ and
$L_g$ is the left multiplication by $g$. 
In other words, the group multiplication has to be
compatible with the Poisson structure in the sense that the multiplication map
$$m: G\times G \to G$$
is a Poisson map if $G\times G$ is endowed with
 the product Poisson
structure defined by $\pi_G$. 
Equivalently, the graph $\graph(\pi_G^\sharp)\subseteq TG\operp T^*G$ of the 
vector bundle homomorphism 
$\pi_G^\sharp:T^*G\to TG$, $\pi_G^\sharp(\dr f)=\pi_G(\cdot,\dr f)$ 
associated to $\pi_G$ is a subgroupoid of the Pontryagin groupoid 
$TG\operp T^*G\rr\lie g^*$ defined by $G$. 
More generally, a Poisson Lie groupoid is a Lie groupoid 
$G\rr P$ endowed with a Poisson structure $\pi_G$ such that 
$\graph(\pi_G^\sharp)$ is a subgroupoid of the Pontryagin groupoid $TG\operp T^*G
\rr TP\operp A^*G$, where $A^*G$ is the dual of the Lie algebroid $AG\to P$ 
associated to  the Lie groupoid $G\rr P$
(see \cite{CoDaWe87}, \cite{Pradines88}, \cite{Mackenzie05} for the induced Lie
groupoid
structures $TG\rr TP$ and $T^*G\rr A^*G$).

A Poisson Lie group $(G,\pi_G)$ induces a Lie algebra structure on the direct
sum of the Lie algebra $\lie g$ of $G$ with its dual $\lie g^*$.
The adjoint action of $\lie g$ on $\lie g\operp\lie g^*$ integrates to a
natural action of $G$ on $\lie g\operp\lie g^*$, see for
example
\cite{Lu90}.
The generalization of this to Poisson Lie groupoids was done in \cite{MaXu94}: the direct sum
$AG\operp A^*G$ of the Lie algebroid and its dual inherits the structure 
of a \emph{Lie bialgebroid}, that generalizes the Lie bialgebra of a Poisson Lie group.

A \emph{Poisson homogeneous space} $(P,\pi_P)$ of a Poisson Lie group
$(G,\pi_G)$ is a homogeneous space $P$ of $G$ endowed with a Poisson structure
$\pi_P$ such that
the transitive left action of $G$ on $P$ 
$$\sigma:G\times P\to P
$$ 
is a Poisson map, where $G\times P$ is endowed with the product Poisson
structure defined  by $\pi_G$ and $\pi_P$ (see for example \cite{Lu08}).

Consider the pairing on $\lie g\operp\lie g^*$
defined by $\langle (x,\xi),(y,\eta)\rangle=\xi(y)+\eta(x)$
for all $ (x,\xi),(y,\eta)\in\lie g\operp\lie g^*$.
A theorem of Drinfel$'$d (see \cite{Drinfeld93} and \cite{DiMe99}, \cite{Lu08}
for more details about
the proof, see also \cite{LiWeXu98}) states that there is a one-to-one 
correspondence between
Poisson structures  $\pi_{G/H}$ on $G/H$ such that
$(G/H,\pi_{G/H})$ is a Poisson homogeneous space of $(G,\pi_G)$, and 
Lagrangian subalgebras $\lie D$ of $\lie g\operp\lie g^*$
satisfying $\lie D\cap(\lie
g\operp\{0\})
=\lie h\operp\{0\}$ ($\lie h$ being the Lie algebra of the Lie subgroup $H$)
which are invariant under the restriction to
$H$ of the action of $G$ on $\lie g\operp\lie g^*$.   

\medskip

Because of this theorem, it appears natural to try to pass to the category of
Dirac manifolds. A Dirac structure on a manifold $M$ is a subbundle $\mathsf
D$ of its 
\emph{Pontryagin} bundle $\mathsf P_M:=TM\operp T^*M$ 
that is Lagrangian relative to the
natural fiberwise pairing $\langle\cdot\,,\cdot\rangle$ defined on $\mathsf P_M$ by
$$\left\langle(v_m,\alpha_m),(w_m,\beta_m)\right\rangle
=\beta_m(v_m)+\alpha_m(w_m)$$
for all $m\in M$ and $ (v_m,\alpha_m),(w_m,\beta_m)\in T_mM\times T_m^*M$. 
The Dirac manifold $(M,\mathsf D)$ is \emph{integrable} if certain integrability
conditions are satisfied. Dirac manifolds generalize Poisson manifolds in the
sense that the graph of the homomorphism of vector bundles $\pi^\sharp:T^*M\to
TM$ associated to a Poisson bivector field $\pi$ on $M$ defines an integrable Dirac
structure on the manifold $M$.

In this paper, we study \emph{Dirac homogeneous spaces} of \emph{Dirac Lie
  groups} and give a generalization of the
theorem of Drinfel$'$d (our main theorem \ref{drinfeld}) in this more natural setting.

Dirac Lie groups have been defined independently by
\cite{Ortiz08}. The definition is made there in the context of groupoids and
is easily shown to be equivalent to the definition made here. An important
feature of a Dirac Lie group $(G,\mathsf D_G)$ is that the characteristic
distribution $\mathsf{G_0}\subseteq TG$ defined by
$\mathsf{G_0}\operp\{0\}=\mathsf D_G\cap (TM\operp\{0\})$ and the 
characteristic
codistribution
$\mathsf{P_1}=\operatorname{Proj}_{T^*M}(\mathsf D_G)$ are always left and right 
invariant
 and have thus constant dimensional fibers on $G$. Hence, integrable
multiplicative Dirac structures are only a
slight generalization of the graphs of multiplicative Poisson bivector fields. The 
approach
in \cite{Ortiz08} uses this fact for the definition of the Lie bialgebra of a
Dirac Lie group. Here, we formulate everything in the Dirac setting and get the
known results such as the definition of the Lie bialgebra of a Poisson Lie group 
as corollaries in the class of examples given by the Poisson Lie groups.

The reason why we chose this approach is because the situation seems to be 
 quite different in the case of a Dirac Lie groupoid. A Dirac Lie groupoid
is a  groupoid endowed with 
a Dirac structure that is a subgroupoid of the Pontryagin groupoid $TG\operp
T^*G\rr TP\operp A^*G$ (see \cite{Ortiz08t}).
The characteristic distribution $\mathsf{G_0}$ can be more complicated in
this case,
and the geometry involved is not necessarily
induced by an underlying Poisson Lie groupoid (it seems that it is even not
necessarily the case if $\mathsf{G_0}$ has
constant rank, see \cite{Jotz10c}).   

The theorem of Drinfel$'$d has 
been extended in \cite{LiWeXu98}
to a correspondence between a certain class of Dirac subspaces 
of the Lie bialgebroid of a Poisson Lie groupoid and its Poisson homogeneous
spaces. Our next aim is to generalize this result to 
Dirac Lie groupoids (\cite{Jotz10b}). For this, 
we will need to construct the object that will play the
role of the Lie bialgebroid in this setting. Here, the results known for 
Poisson Lie groupoids will be the guidelines, but 
it will not be possible to use them
as it is done in \cite{Ortiz08} in the particular case of Dirac Lie
groups.

\paragraph{Outline of the paper}
Backgrounds about Dirac manifolds and actions of Lie groups are recalled in
section
\ref{section_generalities}. The definition of a Dirac Lie group is given in
section \ref{section_DiracLie} and compared with the definition in
\cite{Ortiz08}. Geometric properties of Dirac Lie groups are proved and the
construction of the Lie bialgebra of a Dirac Lie group is made, as well as
the definition of the induced action of $G$ on it. 

Dirac homogeneous spaces of Dirac Lie groups are defined in 
section \ref{section_Dirachomogeneous} and their first geometric properties are
proved. Our main theorem about the correspondence between (integrable) Dirac homogeneous
spaces of a (integrable) Dirac lie group  and Lagrangian subspaces (subalgebras)
of $\lie g\operp\lie g^*$ is proved there, too.

In section \ref{special_case}, we study the special class of Dirac Lie groups
where the characteristic subgroup $N$ is closed in the Lie group $G$,
and the corresponding Dirac homogeneous spaces. 

\paragraph{Acknowledgments}
The author would like to thank Professor Jiang Hua Lu for many interesting
questions, discussions and advises,
especially for the discussion on cocycles at the end of Subsection \ref{geom_prop},
 Professor Karl-Hermann Neeb for the examples in Section
\ref{special_case}, and Professor 
Tudor Ratiu for many interesting discussions and advice.
Many thanks go also to the referee for his useful comments.

\paragraph{Notations and conventions}
Let $M$ be a smooth manifold. We will denote by $\mx(M)$ and $\Omega^1(M)$ the
sheaves of smooth local  sections of the tangent and the cotangent bundle,
respectively. For an arbitrary vector bundle $\mathsf E\to M$, the sheaf of
local sections of $\mathsf E$ will be written $\Gamma(\mathsf E)$. 
We will write $\dom(\sigma)$ for the open subset of the smooth manifold $M$
where the local section $\sigma\in\Gamma(\mathsf E)$ is defined.

\medskip

We will write $\mathcal D\operp\Delta$ for the direct sum of a subbundle
$\mathcal D$ of $TM$ and a subbundle $\Delta$ of $T^*M$. We choose this
notation to distinguish such a direct sum from a direct sum of subbundles of
the same vector bundle. 
Following \cite{MaYo06}, we will call
the direct sum $\mathsf P_M:=TM\operp T^*M$ the 
\emph{Pontryagin bundle} on the manifold $M$.

\section{Generalities on Dirac structures}\label{section_generalities}
\subsection{Dirac structures}\label{subsection_dirac_structures}
The Pontryagin bundle $\mathsf{P}_M=TM
\operp T^* M$ of a smooth manifold $M$ is
endowed with a non-degenerate symmetric fiberwise bilinear form of signature $(\dim M, \dim M)$ given by
\begin{equation}\label{pairing}
\left\langle (u_m, \alpha_m), ( v_m, \beta_m ) \right\rangle 
: = \left\langle\beta_m , u_ m \right\rangle + \left\langle\alpha_m, v _m \right\rangle
\end{equation}
for all $u _m, v _m \in T _mM$ and $\alpha_m, \beta_m \in T^\ast_mM$. A
\emph{Dirac structure} (see \cite{Courant90a}) on $M $ is a Lagrangian vector 
subbundle $\mathsf{D} \subset \mathsf{P}_M $. That is, $ \mathsf{D}$ coincides with its
orthogonal relative to \eqref{pairing} and so its fibers are necessarily $\dim M $-dimensional.
The pair $(M,\mathsf D)$ is then called a \emph{Dirac manifold}.

\medskip

Let $(M,\mathsf D)$ be a Dirac manifold. 
For each $m\in M$, the Dirac structure $\mathsf{D}$ 
defines two subspaces 
$\mathsf{G_0}(m), \mathsf{G_1}(m) \subset T_mM $ by 
\begin{align*}
\mathsf{G_0}(m)&:= \{v_m \in T_mM \mid (v_m, 0)\in\mathsf D(m) \} \quad\text{
  and} \quad  
\mathsf{G_1}(m):= \left\{v_m \in T_mM \mid \exists\,
\alpha_m \in T_m^*M\; : 
(v_m, \alpha_m) \in \mathsf{D}(m)\right\},
\end{align*}
 and two subspaces 
$\mathsf{P_0}(m), \mathsf{P_1}(m) \subset T^*M$ defined analogously. 
The distributions $\mathsf{G_0}=\cup_{m\in M}\mathsf{G_0}(m)$
and $\mathsf{P_0}=\cup_{m\in M}\mathsf{P_0}(m)$
are not necessarily smooth. 
The distributions
$\mathsf{G_1}=\cup_{m\in M}\mathsf{G_1}(m)$
(respectively $\mathsf{P_1}=\cup_{m\in M}\mathsf{P_1}(m)$)
are smooth since they are the projections on 
$TM$ (respectively $T^*M$) of $\mathsf D$.

We have the equalities 
\[\mathsf{P_0}(m)=\mathsf{G_1}(m)^\circ, \quad \mathsf{G_0}(m)=\mathsf{P_1}(m)^\circ,
\quad \mathsf{P_1}(m)=\mathsf{G_0}(m)^\circ, 
\quad \mathsf{G_1}(m)\subseteq\mathsf{P_0}(m)^\circ.
\]

\medskip
The space $\Gamma(\mathsf{P}_M) $ of local sections of the Pontryagin
bundle is endowed with a  skew-symmetric bracket given by
\begin{align}\label{wrong_bracket}
[(X, \alpha), (Y, \beta) ] : 
&= \left( [X, Y],  \boldsymbol{\pounds}_{X} \beta 
- \boldsymbol{\pounds}_{Y} \alpha + \frac{1}{2} \mathbf{d}\left(\alpha(Y) 
- \beta(X) \right) \right) \nonumber 
= \left([X, Y],  \boldsymbol{\pounds}_{X} \beta - \mathbf{i}_Y
  \mathbf{d}\alpha - \frac{1}{2} \mathbf{d} \left\langle (X, \alpha), (Y, \beta) \right\rangle
\right)
\end{align}
(see \cite{Courant90a}). 
This bracket is $\R$-bilinear (in the sense that 
$[a_1(X_1, \alpha_1)+a_2(X_2, \alpha_2), (Y, \beta) ]$\linebreak
$=a_1[(X_1, \alpha_1), (Y, \beta) ]+a_2[(X_2, \alpha_2), (Y, \beta)
]$ for all $a_1,a_2\in\R$ and
$(X_1, \alpha_1),(X_2, \alpha_2), (Y, \beta)\in\Gamma(\mathsf P_M)$  
on the common domain of
definition of the three sections) and does not in general satisfy the Jacobi identity.

The Dirac structure $\mathsf D$ is 
\emph{integrable} if 
$[ \Gamma(\mathsf{D}), \Gamma(\mathsf{D}) ] \subset \Gamma(\mathsf{D}) $. 
Since $\left\langle (X, \alpha), (Y, \beta) \right\rangle = 0$ if 
$(X, \alpha), (Y, \beta) \in \Gamma(\mathsf{D})$, integrability of the Dirac structure 
is  expressed  relative to a non-skew-symmetric bracket that 
differs from 
\eqref{wrong_bracket} by eliminating in the second line the third term of
the second component. This truncated expression 
is called the 
\emph{Courant-Dorfman} or \emph{Dorfman bracket} (\cite{Dorfman93}):
\begin{equation}\label{Courant_bracket}
[(X, \alpha), (Y, \beta) ] : 
= \left( [X, Y],  \boldsymbol{\pounds}_{X} \beta - \ip{Y} \dr\alpha \right).
\end{equation} 
The restriction of the Courant bracket to the sections of an integrable Dirac structure
is skew-symmetric and satisfies the Jacobi identity.
It satisfies also the Leibnitz-rule:
\begin{equation}\label{leibnitz} 
[(X,\alpha), f(Y,\beta)]=f[(X,\alpha), (Y,\beta)]+ X(f)\cdot (Y,\beta)
\end{equation}
for all $(X,\alpha), (Y,\beta)\in\Gamma(\mathsf D)$ and $f\in C^\infty(M)$.

The Dirac manifold $(M,\mathsf D)$ is integrable if and only if the \emph{tensor} 
$T_{\mathsf{D}}$ defined on sections $(X,\alpha),(Y,\beta),(Z,\gamma)$ 
of $\mathsf D$ by
 $T_{\mathsf{D}}\Bigl( (X,\alpha),(Y,\beta),(Z,\gamma)\Bigr)
=\Bigl\langle [(X, \alpha), (Y, \beta)] ,(Z,\gamma)\Bigr\rangle
$
vanishes identically on $M$ (see \cite{Courant90a}).
\medskip

\medskip

The class of Dirac structures presented in the next example 
will be very important in the following.
\begin{example}\label{exPoisson}
Let $M$ be a smooth manifold endowed with a globally defined bivector field
$\pi\in\Gamma\left(\bigwedge^2 TM\right)$.  Then the subdistribution $\mathsf
D_\pi\subseteq \mathsf P_M$ defined by
\[\mathsf{D}_\pi(m)=\left\{(\pi^\sharp(\alpha),\alpha)(m)\mid
  \alpha\in\Omega^1(M), \dom(\alpha)\ni m\right\}\quad \text{ for all }\quad m\in M,
\]
where $\pi^\sharp:T^*M\to TM$ is defined by 
$\pi^\sharp(\alpha)=\pi(\alpha,\cdot)\in\mx(M)$ for all
$\alpha\in\Omega^1(M)$,
is a Dirac structure on $M$. It is integrable if and only if the bivector field
satisfies
$[\pi,\pi]=0$, that is, if and only if $(M,\pi)$ is a Poisson manifold.

Note that for this class of Dirac manifolds, $T_{\mathsf D_\pi}$ is a $3$-tensor on
$T^*M$ since each section of $\mathsf{P_1}=T^*M$ 
corresponds  to exactly one section of $\mathsf D_\pi$. 
The equality $2\cdot T_{\mathsf D_\pi}
=[\pi,\pi]$ (compare (1.82) in \cite{DuZu05} with Proposition 2.5.3 in
 \cite{Courant90a}), 
where $[\cdot\,,\cdot]$ is the Schouten
bracket, shows that $\pi$ is a Poisson bivector if and only if 
$\mathsf D_\pi$ is integrable.
\end{example}

\paragraph{The product of two Dirac manifolds.}
Let $(M,\mathsf{D_M})$ and $(N,\mathsf D_N)$ be Dirac manifolds. 
Consider the product $M\times N$. We identify in the following always (without
mentioning it) the tangent space $T(M\times N)$ with $TM\oplus TN$, and write
$(v_p,w_q)$ for the elements of $T_{(p,q)}(M\times N)=T_pM\oplus
T_qN$. That is, an element of $\mx(M\times N)$ is written $(X,Y)$
with $X\in\mx(M)$ and $Y\in\mx(N)$. 
We identify in the same manner $T^*(M\times N)$ with
$T^*M\oplus T^*N$.

The \emph{product Dirac structure} $\mathsf{D}_M\oplus
\mathsf{D}_N$ on $M\times N$ is the direct sum of $\mathsf{D}_M$ 
and $\mathsf D_N$: the pair
$\left((X,Y),(\alpha,\beta)\right)$ is a section of $\mathsf{D}_M\oplus \mathsf{D}_N$ 
if and only if
$(X,\alpha)\in\Gamma(\mathsf{D}_M)$ and $(Y,\beta)\in\Gamma(\mathsf{D}_N)$.

The Dirac manifold $(M\times N,\mathsf D_M\oplus\mathsf D_N)$ is integrable 
if and only if $(M,\mathsf D_M)$ and $(N,\mathsf D_N)$ are integrable.

\paragraph{Maps in the Dirac category.}
Let $(M,\mathsf D_M)$ and $(N,\mathsf D_N)$ be smooth Dirac manifolds
and $\varphi:M\to N$ a smooth map.
The map $\varphi$ is said to be \emph{backward Dirac} if
for all $(X,\alpha)\in\Gamma(\mathsf D_M)$ there exists
$(Y,\beta)\in\Gamma(\mathsf D_N)$ such that
\[ X\sim_\varphi Y \quad \text{ and }\quad \alpha=\varphi^*\beta.\]
The map $\varphi$ is said to be \emph{forward Dirac} if
for all $(Y,\beta)\in\Gamma(\mathsf D_N)$ there exists
$(X,\alpha)\in\Gamma(\mathsf D_M)$ such that
\[ X\sim_\varphi Y \quad \text{ and } \quad\alpha=\varphi^*\beta.\]

\paragraph{Symmetries of a Dirac manifold $(M,\mathsf{D})$.}
Let $G$ be a Lie group and
$\Phi: G\times M \rightarrow M$ a smooth left action. Then $G$ is called a
\emph{symmetry Lie group of} $(M,\mathsf{D})$ if for every $g\in G$ the condition
$(X,\alpha) \in \Gamma(\mathsf{D})$ implies that  $\left( \Phi_g^\ast X, \Phi_g^\ast
  \alpha \right) \in \Gamma(\mathsf{D})$. We say then that the Lie group $G$ acts
\textit{canonically} or \textit{by Dirac actions} on $M$. 

Let $\mathfrak{g}$ be a Lie algebra and $x \in \mathfrak{g} \mapsto x_M \in
\mathfrak{X}(M)$ be a smooth left Lie algebra action, that is, the map $(m, x) \in
M \times \mathfrak{g} \mapsto x_M(m) \in TM $ is smooth and $x \in\mathfrak{g}
\mapsto x_M  \in \mathfrak{X}(M)$ is a Lie algebra anti-homomorphism.  The Lie
algebra $\mathfrak{g}$ is said to be a 
\emph{symmetry Lie algebra of} $(M,\mathsf{D})$ if for every $x \in \mathfrak{g}$ the condition
$(X,\alpha) \in \Gamma(\mathsf{D})$ implies that  
$\left(\boldsymbol{\pounds}_{x_M}X,\boldsymbol{\pounds}_{x_M}\alpha \right) \in
\Gamma(\mathsf{D})$.  Of course, if $\mathfrak{g}$ is the Lie algebra of
$G $ and $x\mapsto x_M$ the  infinitesimal action map induced by the $G
$-action on $M $, then if $G $
is a symmetry Lie group  of $\mathsf{D}$ it follows that $\mathfrak{g}$ is a symmetry Lie
algebra of $\mathsf{D}$.

\paragraph{Regular reduction of Dirac structures by Lie group actions}\label{reduction}
Assume that we have a canonical free and proper $G$-action on the Dirac manifold 
$(M, \mathsf
D)$. 
Let $\mathcal V$ be the vertical space of the action, that is, the vector
subbundle of $TM$ spanned by the fundamental vector fields $x_M$, $x\in \lie
g$, where $\lie g $ is the Lie  algebra of the Lie group $G$. Set $\mathcal
K:=\mathcal V\operp\{0\}\subseteq \mathsf P_M$ and consider its smooth
orthogonal $\K^\perp\subseteq \mathsf P_M$ relative to the bracket
$<\cdot\,,\cdot>$ on $\mathsf P_M$.
Then both vector bundles $\mathsf D$ and $\K^\perp$ are
$G$-invariant and it is shown in \cite{JoRa08} following \cite{BuCaGu07}
that, under the assumption that $\mathsf D\cap\K^\perp$ 
is a vector bundle on $M$,  the ``quotient'' 
\begin{equation}\label{bucagured}
\mathsf D_{M/G}=\left.\frac{(D\cap \mathcal{K}^\perp)+\mathcal{K}}{\mathcal{K}}\right/G
\end{equation}
defines a Dirac structure on $M/G$, called the \emph{reduced} Dirac structure.
The formulation of this in terms of  smooth
sections is the following: the
 reduced Dirac structure on
$M/G$ is given by
\begin{align}\label{bard}
\mathsf {D}_{M/G}=\operatorname{span}\left\{(\bar X,\bar \alpha)\in
  \Gamma(\mathsf P_{M/G})
\left| \exists
X\in \mathfrak{X}(M) \text{ such that } X\sim_q\bar{X}\text{ and
}(X,q^*\bar{\alpha})\in \Gamma(\mathsf D)\right.\right\},
\end{align}
where $q:M\to M/G$ is the quotient map (see
\cite{Blankenstein00}, \cite{BlvdS01}).

\medskip
The Dirac structure $\mathsf D_{M/G}$ is then the \emph{forward Dirac
image $q(\mathsf D)$ of $\mathsf D$ under $q$}.
Note that the
\emph{pullback Dirac structure} or \emph{backward Dirac image
$q^*\mathsf D_{M/G}$ of a Dirac bundle $\mathsf D_{M/G}\subseteq \mathsf{P}_{M/G}$ under $q$},
defined on $M$ 
by
\[\Gamma\left(q^*\mathsf D_{M/G}\right)
=\left\{
(X,\alpha)\in \Gamma(\mathsf{P}_M)
\left|
\exists (\bar X,\bar\alpha)\in\Gamma(\mathsf
    D_{M/G})
\text{ such that } X\sim_q \bar X,\,\alpha=q^*\bar\alpha
\right.\right\},\]
is a Dirac structure on $M$.
It is easy to check that 
 $\mathsf D$ and $q^*(q(\mathsf D))$ are equal
if and only if we have the inclusion $\mathcal V\subseteq \mathsf{G_0}$.
The equality $q(q^*\mathsf D_{M/G}))=\mathsf D_{M/G}$ holds for any Dirac
structure
$\mathsf D_{M/G}$ on $M/G$.

\subsection{Invariant Dirac structures on a Lie group}
\begin{definition}
A Dirac structure $\mathsf{D}\subseteq TG\operp TG^*$ on a Lie group $G$ is
called \emph{left invariant} if it is invariant under the action of $G$ on
$TG\operp TG^*$ induced from the left action of $G$ on itself. 
That is, if for all $g, h\in G$ and all 
$(v_g,\alpha_g)\in \mathsf{D}(g)$ we have
$(T_gL_hv_g,(T_{hg}L_{h^{-1}})^*\alpha_g)\in \mathsf{D}(hg)$.

In the same manner, a Dirac structure $\mathsf{D}$ on a Lie group $G$ is called 
\emph{right invariant} if it is invariant under the action on $TG\operp TG^*$
induced from the right action of $G$ on itself. 
\end{definition}

Let $G$ be a Lie group with Lie algebra $\lie g$ and let $\lie D$ be a
Dirac subspace of $\lie g\oplus\lie g^*$, that is, $\lie D$ is a vector
subspace of  $\lie g\oplus\lie g^*$ that is orthogonal to itself relative to
the pairing $\langle\cdot\,,\cdot\rangle_{\lie g}$ defined on $\lie g\oplus\lie g^*$ by
$\langle(x,\xi),(y,\eta)\rangle_{\lie g}=\eta(x)+\xi(y)$
for all $x,y\in\lie
g$ and $\xi,\eta\in\lie g^*$.
We set 
\[\lie g_0:=\{x\in \lie g\mid (x,0)\in \lie D\},
\quad \lie g_1:=\{x\in \lie g\mid \exists \xi\in \lie g^* :
(x,\xi)\in \lie D\},\]
\[\lie p_0:=\{\xi\in \lie g^*\mid (0,\xi)\in \lie D\},
\quad \text{ and }\quad \lie p_1:=\{\xi\in \lie g^*\mid \exists x
\in \lie g :
(x,\xi)\in \lie D\}.\]
Then we have $
\lie g_0^\circ=\lie p_1,
\lie p_1^\circ=\lie g_0,
\lie g_1^\circ=\lie p_0$, 
 and  $\lie p_0^\circ=\lie g_1.
$

\medskip

Let $\lie D$ be a Dirac subspace of $\lie g\oplus\lie g^*$, and define
$\lie D^L$ on $G$ by 
\[\lie D^L(g)=\{(T_eL_gx,(T_gL_{g^{-1}})^*\xi)\mid (x,\xi)\in
\lie D\}\]
for all $g\in G$. Then $\lie D^L$ is a left invariant Dirac structure on
$G$. Conversely, if $\mathsf{D}$ is a left-invariant Dirac structure on a Lie group
$G$, then $\mathsf{D}=\lie D^L$, 
where $\lie D:=\mathsf{D}(e)\subseteq \lie g\oplus \lie g^*$.

The next  proposition shows that 
the integrability of $\lie D^L$ depends  only on
$\lie D$ (see also  \cite{Milburn07}).
\begin{proposition}\label{invariant_Dirac_closed}
The Dirac structure $\lie D^L$ is integrable if and only if  
$\zeta([x,y])+\xi([y,z])+\eta([z,x])=0
$
for all pairs $(x,\xi), (y,\eta)$ and $ (z,\zeta)\in \lie D$.
\end{proposition}

\begin{proof}
Recall that $\lie D^L$ is integrable if for all sections $(X,\alpha),
(Y,\beta)\in\Gamma(\lie D^L)$, we have 
$[(X,\alpha),(Y,\beta)]=([X,Y],\ldr{X}\beta-\ip{Y}\dr\alpha)\in\Gamma(\lie D^L)
$.

By \eqref{leibnitz}, it suffices to show this for a set of spanning
sections of $\lie D^L$.
For $(x,\xi)\in \lie D$, the left invariant pair $(x^L,\xi^L)$, defined by
$(x^L(g),\xi^L(g))=(T_eL_gx,(T_gL_{g^{-1}})^*\xi)$ for all $g\in G$,
is a section of $\lie D^L$. Choose $(x,\xi)$,  $(y,\eta)$ and $(z,\zeta)\in
\lie D$. Then we have 
$[x^L,y^L]=[x,y]^L$ by definition of the Lie bracket in $\lie g$,
$\ldr{x^L}\eta^L=(\ad_x^*\eta)^L$,
$\ip{y^L}\dr\xi^L=(\ad_y^*\xi)^L$, where
for $\xi\in\lie g^*$ and $x\in\lie g$, the element $\ad_x^*\xi\in\lie g^*$ is
defined
by $\ad_x^*\xi(y)=\xi([y,x])$ for all $y\in \lie g$.
We get 
\begin{align*}
\left\langle \left([x,y]^L, \ldr{x^L}\eta^L-\ip{y^L}\dr\xi^L\right),
\left(z^L,\zeta^L\right)\right\rangle
&=\zeta([x,y])+\eta([z,x])+\xi([y,z]).
\end{align*}
Hence, since the sections $\left(z^L,\zeta^L\right)$, for all
$(z,\zeta)\in \lie D$, are spanning sections for $\lie D^L$, we conclude
that \linebreak
$[(x^L,\xi^L),(y^L,\eta^L)]= ([x,y]^L,
\ip{x^L}\dr\eta^L-\ip{y^L}\dr\xi^L)$ is a section of $\lie D^L$ if and only if 
$\zeta([x,y])+\eta([z,x])+\xi([y,z])=0$ for all
$(z,\zeta)\in \lie D$.
\end{proof}
Note that we have shown simultaneously that if $\lie D^L$ is integrable, then $\lie
g_0$ and $\lie g_1$ are Lie subalgebras of $\lie g$. Since $\mathsf{G_0}$ and
$\mathsf{G_1}$ are here obviously equal to $\lie g_0^L$ and $\lie g_1^L$,
respectively, we recover the fact that both distributions are then integrable.

\section{Dirac Lie groups}\label{section_DiracLie}
\subsection{Definitions}
\begin{definition}
A \emph{Dirac Lie group} is a Lie group $G$ endowed with a Dirac structure $\mathsf{D}_G\subseteq
TG\operp T^*G$  such that the
group multiplication map 
\[m:\left(G\times G,\mathsf D_G\oplus\mathsf D_G\right)\to \left(G,\mathsf
  D_G\right)\] 
is a \emph{forward Dirac map}. The Dirac structure $\mathsf D_G$ on $G$ is
called \emph{multiplicative} if it satisfies this condition.
\end{definition}

More explicitly,  there
exist for all  $g, h\in G$ and pairs $(v_{gh},\alpha_{gh})$ in 
$\mathsf{D}_G(gh)$, two pairs $(w_g,\beta_g)\in \mathsf{D}_G(g)$ and
$(u_h,\gamma_h)\in \mathsf{D}_G(h)$ such that 
\[ T_{(g,h)}m(w_g,u_h)=v_{gh}   \quad      \text{ and } \quad  (\beta_g,\gamma_h)=(T_{(g,h)}m)^*\alpha_{gh}.  \]
That is,  we have 
\[T_hL_gu_h+T_gR_hw_g=v_{gh}\in T_{gh}G\quad \text{ and }\quad 
(\beta_g,\gamma_h)=\alpha_{gh}\circ T_{(g,h)}m ,\]
which is equivalent to the following:
for all $(x_g,y_h)\in T_{(g,h)}(G\times G)$, we have 
\[\beta_g(x_g)+\gamma_h(y_h)=(\beta_g,\gamma_h)(x_g,y_h)=\alpha_{gh}(T_hL_gy_h+T_gR_hx_g).\]
Hence, we have in particular
$\beta_g=(T_gR_h)^*\alpha_{gh} \text{ and } \gamma_h=(T_hL_g)^*\alpha_{gh}$.

The next example shows that Dirac Lie groups generalize Poisson Lie
groups to the category of Dirac manifolds.
\begin{example}\label{eqDirac-Poisson}  
Let $(G,\pi)$ be a Lie group endowed with a bivector field
$\pi\in\Gamma(\bigwedge^2 TM)$, and let
$(G,\mathsf{D}_{\pi})$ be the associated Dirac structure as in Example \ref{exPoisson}.
We show that $\pi$ is multiplicative (see \cite{LuWe90}) if and only if 
$(G,\mathsf{D}_{\pi})$ is a Dirac Lie group.
The bivector field $\pi$ is multiplicative  if and only if 
$\pi(gh)=T_hL_g(\pi(h))+T_gR_h(\pi(g))$
for all $g,h\in G$.
For $f\in \smoo$, we define $X_f:=\pi(\cdot,\dr f)=-\pi^\sharp(\dr f)\in \mx(G)$.
Then the pair $(-X_f,\dr f)$ is a section of $\mathsf{D}_{\pi}$.
Choose $(-X_f(gh),\dr f _{gh})\in \mathsf{D}_{\pi}(gh)$ and consider  
$\dr (R_h^*f)_g\in T_gG^*$
and $\dr (L_g^*f)_h\in T_hG^*$. Then we have
$(-X_{R_h^*f}(g),\dr (R_h^*f)_g)\in
\mathsf{D}_{\pi}(g)$ and  $(-X_{L_g^*f}(h), \dr (L_g^*f)_h)\in \mathsf{D}_{\pi}(h)$.
This yields for an arbitrary $\alpha_{gh}$ in $T_{gh}^*G$:
\begin{align*}
\alpha_{gh}\left(T_gR_hX_{R_h^*f}(g)+T_hL_gX_{L_g^*f}(h)\right)
=& \pi(g)(\alpha_{gh}\circ T_gR_h,(R_h^*\dr f)_g)+\pi(h)(\alpha_{gh}\circ T_hL_g,(L_g^*\dr
f)_h)\\
=&(T_gR_h\pi(g))(\alpha_{gh},\dr f_{gh})+(T_hL_g\pi(h))(\alpha_{gh},\dr f_{gh}).
\end{align*}
The last sum is equal to 
$
\pi(gh)(\alpha_{gh},\dr f_{gh})=\alpha_{gh}\left( X_f(gh)      \right)
$
for all $f\in \smoo$ 
if and only if $\pi$ is multiplicative.
Hence, the equality 
$T_gR_h(-X_{R_h^*f}(g))+T_hL_g(-X_{L_g^*f}(h))=-X_f(gh)$ holds for all $f\in \smoo$ 
if and only if $\pi$ is multiplicative.
Since $(-X_f(gh),\dr f_{gh}) $ was an arbitrary element of
$\mathsf{D}_{\pi}(gh)$, we
have shown that 
$(G,\mathsf{D}_{\pi})$ is a Dirac Lie group if and only if 
$(G,\pi)$ is multiplicative. Thus, $(G,\mathsf{D}_{\pi})$ is an integrable
Dirac Lie group if and only if 
$(G,\pi)$ is a Poisson Lie group.
\end{example}

Note that the Dirac structure associated to the trivial Poisson structure on
$G$ is given by $\mathsf{D}_0=\{0\}\operp TG^*$. Since a Lie group endowed with the
trivial Poisson  structure $\pi=0$ on $G$ is always a Poisson Lie group, this shows that
$(G,\{0\}\operp TG^*)$ is a (integrable) Dirac Lie group. This can also be checked
directly from the definition.
\paragraph{The Dirac Lie group as a subgroupoid of the Pontryagin
  bundle}
An other approach to Dirac Lie groups can be found in \cite{Ortiz08}. For the
sake of completeness, we
show that both definitions are equivalent. For this, we have to introduce the
groupoid structure on the Pontryagin bundle of a Lie group.

Let $G$ be a Lie group. Then its Pontryagin bundle $\mathsf{P}_G$ has the structure
of a Lie groupoid over $\lie g^*$ as
follows. 
The \emph{target} and \emph{source} maps $\mathsf t$
and $\mathsf s$ are defined by
\[\begin{array}{cccc}
\mathsf t:&TG\operp TG^*&\to&\lie g^*\\
&(v_{g},\alpha_g)\in T_gG\times T_gG^*&\mapsto& (T_eR_g)^*\alpha_g
\end{array}\quad
\text{ and } 
\quad \begin{array}{cccc}
\mathsf s:&TG\operp TG^*&\to&\lie g^*\\
&(v_{g},\alpha_g)\in T_gG\times T_gG^*&\mapsto&  (T_eL_g)^*\alpha_g
\end{array}.
\]
If $\mathsf s(v_g,\alpha_g)=\mathsf t(w_h,\beta_h)$, then the product
$(v_g,\alpha_g)\star(w_h,\beta_h)$ makes sense and is equal
to
\begin{align*}
(v_g,\alpha_g)\star(w_h,\beta_h)
&=\bigl(T_gR_hv_g+T_hL_gw_h,(T_{gh}R_{h\inv})^*\alpha_g
\bigr)=\bigl(T_gR_hv_g+T_hL_gw_h,(T_{gh}L_{g\inv})^*\beta_h
\bigr).
\end{align*}
The identity map $\mathsf u: \lie g^*\to \mathsf{P}_G$ is given by 
$\mathsf u(\xi)=(0,\xi)\in \lie g\operp \lie g^*$ and 
the inverse map $\mathsf i: \mathsf{P}_G\to \mathsf{P}_G$ is defined
by $$\mathsf i: (v_g,\alpha_g)
\mapsto \bigl(-T_g(L_{g\inv}R_{g\inv})v_g,T_g(L_gR_{g\inv})^*\alpha_g\bigl).$$

Given this definition, it is easy to verify that
the graph $\mathsf D_{\pi}\subseteq \mathsf{P}_G$ of a Poisson structure on the Lie  group
$G$ is multiplicative if and only if
$\mathsf D_{\pi}$ is a \emph{subgroupoid} of the Pontryagin 
groupoid.

In \cite{Ortiz08}, a Dirac Lie group is hence defined as follows:
a Dirac structure $\mathsf D_G$ on a Lie group is called \emph{multiplicative} 
if $\mathsf D_G$ is a subgroupoid of the Pontryagin groupoid. 
The pair $(G,\mathsf D_G)$ is then called a Dirac Lie group.

It is easy to prove that the two definitions of a multiplicative
 Dirac structure on $G$ are equivalent.

\begin{remark}
The behavior of the \emph{Cartan Dirac structure} under group multiplication
is studied in \cite{AlBuMe07}. This interesting example of a Dirac structure
on a Lie group is defined on a Lie group
$G$ with Lie algebra $\lie g$ endowed with a bilinear, symmetric
$\Ad$-invariant bilinear form (see \cite{AlBuMe07} and references therein).
\end{remark}

\subsection{Geometric properties of Dirac Lie groups}\label{geom_prop}
In this section and the following, $(G,\mathsf D_G)$ will always be a Dirac
Lie group. We denote by  
$\lie g_1:=\mathsf{G_1}(e)$, $\lie
g_0:=\mathsf{G_0}(e)$,  $\lie p_1:=\mathsf{P_1}(e)$ and 
$\lie p_0:=\mathsf{P_0}(e)$
the smooth characteristic distributions evaluated at the neutral
element $e$ of $G$.

The following result has been shown independently by \cite{Ortiz08}.
\begin{proposition}\label{oncardis}
Let $(G,\mathsf{D}_G)$ be a Dirac Lie group. The associated codistribution (respectively
distribution) $\mathsf{P_1}$ (respectively $\mathsf{G_0}$) has constant rank
on $G$, and is given by $\mathsf{P_1}=\lie p_1^L=\lie p_1^R$ (respectively
$\mathsf{G_0}=\lie g_0^L=\lie g_0^R$).
\end{proposition}
\begin{proof}
We  use the definition of \cite{Ortiz08t}. If $v_g$ is an element of
$\mathsf{G_0}(g)$, we have $(v_g,0_g)\in\mathsf D_G(g)$ and $\mathsf
t(v_g,0_g)=\mathsf s(v_g,0_g)=0\in\lie g^*$. Thus, since 
$(0_{g\inv}, 0_{g\inv})\in\mathsf D_G(g\inv)$, we have 
$(0_{g\inv}, 0_{g\inv})\star (v_g,0_g)\in\mathsf D_G(e)$ and 
$(v_g,0_g)\star (0_{g\inv}, 0_{g\inv})\in\mathsf D_G(e)$.
But it is easy to see that $0_{g\inv} \star
v_g=T_{(g\inv,g)}m(0_{g\inv},v_g)=T_gL_{g\inv}v_g$ and 
$v_g \star 0_{g\inv}=T_gR_{g\inv}v_g$, and we get
$(T_gL_{g\inv}v_g,0_e)\in\mathsf D_G(e)$ and 
$(T_gR_{g\inv}v_g,0_e)\in\mathsf D_G(e)$.
We have thus shown that $T_gL_{g\inv}\mathsf{G_0}(g)\subseteq 
\lie g_0$ and $T_gR_{g\inv}\mathsf{G_0}(g)\subseteq 
\lie g_0$. Conversely, if $x\in\lie g_0$, then 
$(x,0)\in\mathsf D_G(e)$ and $(T_eL_gx, 0_g)=(0_g,0_g)\star (x,0)\in\mathsf
D_G(g)$ and $(T_eR_gx, 0_g)=(x,0)\star(0_g,0_g) \in\mathsf
D_G(g)$. Thus, we have shown the equalities $\mathsf{G_0}(g)=T_eL_g\lie
g_0=T_eR_g\lie g_0$ and $\mathsf{G_0}$ has constant rank on $G$.

As a consequence, 
$\mathsf{P_1}$, which is the annihilator of $\mathsf{G_0}$, has also
constant rank on $G$. The following equality follows easily: 
\[\mathsf{P_1}=\mathsf{G_0}^\circ=\left(\lie g_0^L\right)^\circ
=\left(\lie g_0^\circ\right)^L=\lie p_1^L,\]
and we get in the same manner $\mathsf{P_1}=\lie p_1^R$.
\end{proof}
 We have the
immediate corollaries:
\begin{corollary}\label{adG_invariance}
The subspaces $\lie g_0\subseteq \lie g$ and $\lie p_1\subseteq \lie g^*$
satisfy  $\Ad_g^*\lie p_1=\lie p_1$, $\Ad_g\lie g_0=\lie g_0$
for all $g\in G$. Consequently, we have $\ad_x^*\lie p_1\subseteq \lie
p_1$ for all $x\in\lie g$ and $\lie g_0$ is an ideal in $\lie g$.
\end{corollary}
\begin{proof}
We have $\mathsf{P_1}=\lie p_1^R=\lie p_1^L$  and 
$\mathsf{G_0}=\lie g_0^R=\lie g_0^L$ by Proposition \ref{oncardis}. 
Then, for all $g\in G$ and $\xi\in\lie p_1$, the 
covector $(T_gL_{g\inv})^*\xi$ is an element of $\mathsf{P_1}(g)$ and  there exists 
$\eta\in\lie p_1$ such that
$(T_gL_{g\inv})^*\xi=(T_gR_{g\inv})^*\eta$. This yields 
$\Ad_{g\inv}^*\xi=\eta\in\lie p_1$ and 
$\lie p_1$ is consequently $\Ad_{g\inv}^*$-invariant for all $g\in G$. 
In the same manner, we show that $\lie g_0$ is $\Ad_g$-invariant for all $g\in G$.

This yields by derivation $\ad_x^*\xi\in\lie p_1$ for all $\xi\in\lie p_1$ and 
$\ad_xz\in\lie g_0$ for all $z\in\lie g_0$ and $x\in\lie g$. The inclusion 
$[\lie g,\lie g_0]\subseteq \lie g_0$ holds then and shows that
$\lie g_0$ is an ideal of $\lie g$.
\end{proof}

If $G$ is a simple Lie group, the ideal $\lie g_0$ is either trivial or equal
to $\lie g$ and we get the following corollary.
\begin{corollary}
If $(G,\mathsf{D}_G)$ is a simple Dirac Lie group, the Dirac structure 
 $\mathsf D_G$ is either the graph of the
vector bundle homomorphism $T^*G\to TG$ induced by a
multiplicative bivector field on $G$, or the trivial tangent Dirac structure 
$\mathsf D_G=TG\operp\{0\}$.
\end{corollary}

We have also the following proposition.
\begin{proposition}\label{De}
Let $(G,\mathsf{D}_G)$ be a Dirac Lie group. Then we have $\mathsf{D}_G(e)=\lie
g_0\operp\lie p_1$. Consequently, the equality
$\alpha_e(Y(e))=0=\beta_e(X(e))$ 
holds for
all sections $(X,\alpha)$ and $(Y,\beta)$ of $\mathsf{D}_G$ defined on a
neighborhood of the neutral element $e$.
\end{proposition}

\begin{proof}
Choose $(x,\xi)\in\mathsf D_G(e)$. Then we have $\mathsf
s(x,\xi)=\xi\in\lie g^*$  and hence $\mathsf
u(\xi)=(0,\xi)=(x,\xi)\inv\star(x,\xi)\in\mathsf D_G(e)$. Hence, $(x,0)=(x,\xi)-(0,\xi)$ is
also an element of $\mathsf D_G(e)$ and $x\in \lie g_0$. This shows that
$\mathsf D_G(e)\subseteq \lie g_0\times\lie p_1$ and also $\lie p_1=\lie
p_0$. Because of this last equality, the inclusion
$\lie g_0\times\lie p_1\subseteq \mathsf D_G(e)$ is obvious.
\end{proof}

\begin{lemma}\label{P1exact}
The subbundle
$\mathsf{G_0}\subseteq TG$ is involutive and hence integrable,  and $\mathsf{P_1}$
is spanned by exact one-forms. 
\end{lemma}

\begin{proof}
Recall first that $\lie p_1^L=\mathsf{P_1}$ and $\lie g_0^L=\mathsf{G_0}$ by
Proposition \ref{oncardis}, and that  both distributions have hence constant
rank on $G$. Since $\lie g_0$ is an
ideal of $\lie g$, we have  in particular
$[\lie g_0,\lie g_0]\subseteq \lie g_0$, and 
$\mathsf{G_0}$ is  thus integrable in the sense of Frobenius.

Any $ g\in G $ lies in a foliated chart domain $ U $ described by
coordinates $(x^1,\dots,x^n)$ such that the first $ k $ among them define the 
local integral manifold of $\mathsf{G_0}$ containing $g$. Thus, for any $ g'\in  U$ 
the basis vector fields $\partial_{x^1},\dots,\partial_{x^k} $ evaluated at $g' $ 
span $\mathsf{G_0}(g') $. Since $\mathsf{P_1}=\mathsf{G_0}^\circ$, the codistribution
$\mathsf{P_1}$ is spanned by $\dr x^{k+1},\ldots,\dr x^n$ on the neighborhood $U$
of $g$.
\end{proof}

\begin{remark}\label{quotient}
Since $\lie g_0$ is an ideal in $\lie g$, the integral leaf $N$ of $\mathsf{G_0}$
through $e\in G$ is a normal subgroup of $G$. If $N$ is in addition
closed in $G$, its (left- or right-) action  on $G$ is proper. 

We will see later that in certain cases (for example when the Dirac Lie group
is integrable), the induced action of $N$ on $(G,\mathsf D_G)$ is canonical. 
Also, since  $\V_N:=\lie g_0^L=\lie g_0^R$ is the vertical space of the action of $N$
on $G$, it is easy to see that $\mathsf D_G\cap\K_N^\perp=\mathsf D_G$ and
this intersection has consequently constant rank on $G$ (recall the paragraph about
regular reduction of symmetric Dirac structures in Section \ref{reduction}). If
$N$ is closed in $G$, we can hence build the quotient $q_N:G\to G/N$ and 
$(G/N,q_N(\mathsf D_G))$ will be shown later to be a Dirac manifold with
$q_N(\mathsf D_G)$ the graph of a multiplicative bivector field on $G/N$. In
particular,
if $(G,\mathsf D_G)$ is integrable, the quotient $(G/N,q_N(\mathsf D_G))$ will be a
Poisson Lie group.
\end{remark}

\medskip 

Consider a Lie group $G$ and $\tilde p:\tilde G\to G$ its universal covering. 
Then there exists
a discrete normal subgroup $\Gamma$ of $\tilde G$ such that $G=\tilde
G/\Gamma$
(see \cite{Knapp02}). The following 
proposition is easy to prove.

\begin{proposition}\label{prop_univ_cover}
Let $\mathsf D_G$ be a multiplicative (integrable) Dirac structure on $G$.
Then the pullback Dirac structure $\tilde{\mathsf{D}_G}:=\tilde p^*\mathsf D_G$ is a (integrable)
multiplicative Dirac structure on $\tilde G$.
\end{proposition}
\begin{remark}\label{rem_univ_cover}
The integral leaf $\tilde N$ through $e\in\tilde G$ of the characteristic
distribution $\tilde{\mathsf{G_0}}$ defined by $\tilde{\mathsf{D}_G}$ on
$\tilde G$ is normal in $\tilde G$ and hence closed since $\tilde G$ is simply
connected (see \cite{HiNe91}). Hence, the quotient $\tilde G/\tilde N$ is here always well-defined.
\end{remark}
\begin{example}
Let $G$ be a connected Lie group.
 The Lie algebra $\lie g$ of $G$ can be Levi-decomposed as
the semi-direct product $\lie g=\lie s\oplus_\phi\rad \lie g$
with $\lie s$ semi-simple and $\phi:\lie s\to \operatorname{Der}(\rad\lie g)$ a Lie algebra
homomorphism (see for instance \cite{Knapp02}).

The ideal $\rad \lie g$ of $\lie g$ is a solvable ideal of $\lie g$ and its
integral leaf $R$ is closed in $G$ (see \cite{HiNe91}). The quotient $G/R$ is
then a semi-simple Lie group. Let $q_R:G\to G/R$ be the projection and $\pi$ be 
the standard multiplicative
Poisson structure on the semi-simple Lie group $G/R$ (see \cite{EtSc02} 
and \cite{Lu90}).  The pullback $q_R^*\mathsf
 D_\pi$ is an integrable Dirac structure on $G$. 
Its
characteristic distribution is the left or right invariant image of the ideal $\lie
g_0=\rad \lie g$ of $\lie g$ and the action of the integral leaf  $R$ of $\mathsf{G_0}$  on
$(G,q_R^*\mathsf D_\pi)$ is canonical, the Poisson Lie group associated 
to this Dirac Lie group as in Remark \ref{quotient} is obviously $(G/R,\pi)$.
\end{example}
\medskip

The following lemma will be useful for many proofs in this paper. We will
always use the following notation. If $\xi$ is an element of the subspace
$\lie p_1\subseteq \lie g^*$, then the one-form $\xi^L$ is a section of
$\mathsf{P_1}$ by Proposition \ref{oncardis}. We denote by  $X_\xi\in\mx(G)$ a vector field satisfying
$(X_\xi,\xi^L)\in\Gamma(\mathsf D_G)$. The vector field $X_\xi$
 is not
necessarily unique:  all $Y\in
X_\xi+\Gamma(\mathsf{G_0})$ satisfy the condition
$(Y,\xi^L)\in\Gamma(\mathsf D_G)$.

\begin{lemma}\label{too_useful}
Choose $\xi\in\lie p_1$ and corresponding vector fields $X_\xi$ and
$X_{\Ad_{h\inv}^*\xi}$
for $h\in G$.
 Then the inclusion
\begin{equation}\label{eq4}
X_\xi(gh)\quad \in\quad T_hL_gX_\xi(h)+T_gR_hX_{\Ad_{h\inv}^*\xi}(g)+\mathsf{G_0}(gh)
\end{equation}
holds for all  $g\in G$.
\end{lemma}
\begin{remark}\label{right_inv_ham}
If $Y_\xi$ and $Y_{\Ad_{h}^*\xi}\in\mx(G)$ are such that 
$(Y_\xi,\xi^R), (Y_{\Ad_{h}^*\xi},(\Ad_{h}^*\xi)^R)\in\Gamma(\mathsf D_G)$,
then we can show in the same manner
\begin{equation*}
Y_\xi(hg)\quad \in\quad T_gL_hY_{\Ad_{h}^*\xi}(g)+T_hR_gY_\xi(h)+\mathsf{G_0}(hg)
\end{equation*} for all  $g\in G$.
\end{remark}
\begin{proof}
Since $(G,\mathsf{D}_G)$ is a Dirac Lie group and
$(X_\xi(gh),\xi^L(gh))\in\mathsf D_G(gh)$, there exist 
$w_g\in T_gG$ and $u_{h}\in T_{h}G$ such that
\[\left(w_g,( T_{g}R_{h})^*\xi^L(gh)\right)
\in\mathsf{D}_G(g),\quad \left(u_{h},(T_{h}L_g)^*\xi^L(gh)\right)
\in\mathsf{D}_G(h)\quad 
\text{ and  }
\quad T_gR_{h}w_g+T_{h}L_gu_{h}=X_\xi(gh).\]
But since 
\[(T_{g}R_{h})^*\xi^L(gh)=(\Ad_{h\inv}^*\xi)^L(g)
\quad
\text{ and }
\quad (T_{h}L_g)^*\xi^L(gh)=\xi^L(h),\]
 we have
\[w_g-X_{ \Ad_{h\inv}^*\xi}(g)\in\mathsf{G_0}(g)
\quad \text{ and }\quad
u_{h}-X_\xi(h)\in\mathsf{G_0}(h).
\] 
With the equalities $\mathsf{G_0}=\lie g_0^L=\lie g_0^R$, this proves the claim.
\end{proof}
\begin{proposition}\label{lemma}
Let $\xi$ and $\eta$ be elements of $\lie p_1$ and
$X_\xi,X_\eta\in\mx(G)$ corresponding   vector 
fields.
The one-form $\ldr{X_\xi}\eta^L-\ip{X_\eta}\dr\xi^L$ is 
left invariant and equal to $\left(\dr_e\!\!\left(\eta^L(X_\xi)\right)\right)^L$.
\end{proposition}

\begin{proof}
Choose $x\in\lie g$
and, using the preceding Lemma, compute
\begin{align*}
&\left(\ldr{X_{\xi}}\eta^L-\ip{X_\eta}\dr\xi^L\right)(x^L)(g)\\
=&\,X_\xi(\eta^L(x^L))(g)
+\eta^L(\ldr{x^L}X_\xi)(g)
-X_\eta(\xi^L(x^L))(g)+x^L(\xi^L(X_\eta))(g)-\xi^L(\ldr{x^L}X_\eta)(g)\\
=&\,\eta^L(g)\left(\left.\frac{d}{dt}\right\an{t=0}T_{g\exp(tx)}
R_{\exp(-tx)}X_\xi(g\exp(tx))\right)
+(\ldr{x^L}\xi^L)(X_\eta)(g)\\
\overset{\eqref{eq4}}=&\left.\frac{d}{dt}\right\an{t=0}\eta^L(g)(X_{\Ad_{\exp(-tx)}^*\xi}(g))
+\eta\left(\left.\frac{d}{dt}\right\an{t=0}T_{\exp(tx)}R_{\exp(-tx)}X_\xi(\exp(tx))
\right)+(\ad_x^*\xi)^L(X_\eta)(g)\\
=&-\left.\frac{d}{dt}\right\an{t=0}(\Ad_{\exp(-tx)}^*\xi)^L(g)(X_\eta(g))
+\eta(\ldr{x^L}X_\xi(e))+(\ad_x^*\xi)^L(X_\eta)(g)\\
=&-(\ad_x^*\xi)^L(X_\eta)(g)+\ldr{x^L}\left(\eta^L(X_\xi)\right)(e)
-(\ldr{x^L}\eta^L)(X_\xi)(e)+(\ad_x^*\xi)^L(X_\eta)(g)=\,\dr_e\left(\eta^L(X_\xi)\right)(x),
\end{align*}
where we have used  Proposition \ref{De} and
$(\ldr{x^L}\eta^L)=(\ad_x^*\eta)^L\in\Gamma(\mathsf{P_1})$ by
Corollary \ref{adG_invariance}.
\end{proof}

\begin{definition}\label{defbracket}
Let $(G,\mathsf D_G)$ be a Dirac Lie group. Define
the bilinear, antisymmetric  bracket 
\[
[\cdot\,,\cdot]:\lie p_1\times\lie p_1\to\lie g^*
\quad 
\text{ by }
\quad [\xi,\eta]=\dr_e\!\!\left(\eta^L(X_\xi)\right),
\]
 where $X_\xi\in\mx(G)$ is such that  $(X_\xi,\xi^L)\in\Gamma(\mathsf D_G)$. That is, we set the notation
$\ldr{X_{\xi}}\eta^L-\ip{X_\eta}\dr\xi^L=:[\xi,\eta]^L$ and hence
$[(X_\xi,\xi^L),(X_\eta,\eta^L)]=([X_\xi,X_\eta],[\xi,\eta]^L)$ for all
$\xi,\eta\in\lie p_1$.
\end{definition}
Note that $[\xi,\eta]$ does not depend on the choice of the vector field $X_\xi$.
Indeed, if $Y_\xi\in\mx(G)$ is an other  vector field such that $(Y_\xi,\xi^L)\in\Gamma(\mathsf D_G)$, we have
$Y_\xi-X_\xi\in\Gamma(\mathsf{G_0})$ and hence
$\eta^L(X_\xi)=\eta^L((X_\xi-Y_\xi)+Y_\xi))=\eta^L(Y_\xi)$ since
$\eta^L\in\Gamma(\mathsf{P_1})$.

The bilinearity of the bracket is obvious. For the antisymmetry,
choose $\xi,\eta\in\lie p_1$. Then we have $\xi^L(X_\eta)+\eta^L(X_\xi)=0$  since $(X_\xi,\xi^L)$ and 
$(X_\eta,\eta^L)$ are sections of $\mathsf D_G$, and this leads to
\begin{align*}
[\xi,\eta]=\dr_e\left(\eta^L(X_\xi)\right)=\dr_e\left(-\xi^L(X_\eta)\right)
=-[\eta,\xi].
\end{align*}

\medskip

As a direct corollary of
Proposition \ref{lemma}
we recover the fact that every multiplicative Dirac
structure on a torus is trivial.

\begin{corollary}\label{abelianDiracLie}
Consider an \emph{Abelian} Dirac Lie group $(G,\mathsf D_G)$ and 
choose $x$ in the Lie algebra $\lie g$. Then the equality 
\[\left(\eta^L(X_\xi)\right)(g\cdot \exp(tx))
=[\xi,\eta](tx)+\left(\eta^L(X_\xi)\right)(g)
\]
holds for all $g\in G$ and $t\in \R$. 

As a consequence, if $\mathsf
D_{\mathbb{T}^n}$ is a multiplicative Dirac structure on the $n$-torus
$\R^n/\mathbb{Z}^n$, then $\mathsf D_{\mathbb{T}^n}$ is the direct sum 
\[\mathsf D_{\mathbb{T}^n}=\mathsf{G_0}\operp\mathsf{P_1}.
\]
\end{corollary} 

\begin{proof}
We have shown in Proposition \ref{lemma} 
that $\ldr{X_\xi}\eta^L-\ip{X_\eta}\xi^L=[\xi,\eta]^L$ is
a left-invariant one-form on $G$. 
We have for  all $\xi,\eta\in\lie p_1$ and $x\in\lie g$:
\begin{align}
[\xi,\eta]^L(x^L)=&(\ldr{X_\xi}\eta^L-\ip{X_\eta}\dr\xi^L)(x^L)
=\,\eta^L(\ldr{x^L}X_\xi)+(\ldr{x^L}\xi^L)(X_\eta) \quad \text{ (see the proof
  of Proposition \ref{lemma}) } \nonumber\\
=&\, x^L(\eta^L(X_\xi))-(\ldr{x^L}\eta^L)(X_\xi)+(\ldr{x^L}\xi^L)(X_\eta)
\nonumber\\
=&\, x^L(\eta^L(X_\xi))-(\ad_x^*\eta)^L(X_\xi)+(\ad_x^*\xi)^L(X_\eta)
=\, x^L(\eta^L(X_\xi))   \label{eq36}
\end{align} 
since $\ad_x^*\xi=\ad_x^*\eta=0$ 
because
$\lie g$ is Abelian.
We get $\dr(\eta^L(X_\xi))=[\xi,\eta]^L$ and the equality
$\frac{d}{dt}R_{\exp(tx)}^*f
=R_{\exp(tx)}^*\left(\ldr{x^L}f\right)
$ for all $f\in\smoo$
yields
\begin{align*}
\frac{d}{dt}\left(\eta^L(X_\xi)\right)(g\exp(tx))
=R_{\exp(tx)}^*\left(x^L(\eta^L(X_\xi))\right)(g)
\overset{\eqref{eq36}}=R_{\exp(tx)}^*\left([\xi,\eta](x)\right)=[\xi,\eta](x)
\end{align*}
for all $g\in G$ and $t\in \R$. We get
\begin{align*}
\left(\eta^L(X_\xi)\right)(g\exp(tx))=[\xi,\eta](x)\cdot
t+\left(\eta^L(X_\xi)\right)(g)
=[\xi,\eta](tx)+\left(\eta^L(X_\xi)\right)(g).
\end{align*}

On the $n$-dimensional torus $\mathbb T^n$, we have 
$\exp(tx)=tx+\mathbb Z^n$ for all $x\in \lie g=\R^n$ and all $t\in\R$. This
yields
\[\left(\eta^L(X_\xi)\right)(\exp(tx))=
[\xi,\eta](tx)+\left(\eta^L(X_\xi)\right)(0)=[\xi,\eta](tx)
\] 
for all $x\in\R^n$ and all $t\in\R$. But since the function 
$\left(\eta^L(X_\xi)\right)$ is well-defined on $\mathbb T^n=\R^n/\mathbb
Z^n$, the equality
$[\xi,\eta](tx)=[\xi,\eta](tx+z)$ has to hold for all
$x\in\R^n$, $t\in \R$ and $z\in\mathbb Z^n$. This leads to $[\xi,\eta]=0$ 
and hence $\eta^L(X_\xi)$ is constant and equal to its value at the neutral
element; $\eta^L(X_\xi)(0)=\eta(X_\xi(0))=0$ for all $\xi,\eta\in\lie p_1$
by Proposition \ref{De}.
Thus, each spanning vector field $X_\xi$, $\xi\in\lie p_1$, of $\mathsf{G_1}$ 
 is annihilated by
$\mathsf{P_1}$ and is consequently a section of $\mathsf{G_0}$.   
\end{proof}

The next proposition shows that the value of $[\xi,\eta]$, for
$\xi,\eta\in\lie p_1$, can be computed with any two one-forms in
$\Gamma(\mathsf{P_1})$ taking value $\xi,\eta$ in $e$.
\begin{proposition}\label{other_formula_for_the_bracket}
Let $\alpha,\beta\in\Gamma(\mathsf{P_1})$ be  such that
$\alpha(e)=\xi$ and $\beta(e)=\eta\in\lie p_1$. Then we have
$$[\xi,\eta]=\dr_e(\beta(X_\alpha)),
$$
where $X_\alpha\in\mx(G)$ is such that $(X_\alpha,\alpha)\in\Gamma(\mathsf D_G)$.
\end{proposition}

\begin{proof}
First, we show that $\ldr{x^R}\alpha\in\Gamma(\mathsf{P_1})$ for 
all $\alpha\in\Gamma(\mathsf{P_1})$. For all $g,h\in G$ we have
$(L_g^*\alpha)(h)=\alpha_{gh}\circ T_{h}L_g=(T_hL_g)^*\alpha_{gh}$. 
This is an element of 
$\mathsf{P_1}(h)$ since $\alpha_{gh}\in\mathsf{P_1}(gh)$ and
$\mathsf{P_1}$ is left-invariant by Proposition 
\ref{oncardis}. Thus, we get
$$(\ldr{x^R}\alpha)(h)=\left.\frac{d}{dt}\right\an{t=0}(L_{\exp(tx)}^*\alpha)(h)
\in\mathsf{P_1}(h)
$$
since $(L_{\exp(tx)}^*\alpha)(h)\in\mathsf{P_1}(h)$ for all $t$.
Choose $x\in\lie g$ and compute
\begin{align*}
[\xi,\eta](x)&=x^R(\eta^L(X_\xi))(e)
=\eta(\ldr{x^R}X_\xi(e))=\beta(e)(\ldr{x^R}X_\xi(e))
=x^R(\beta(X_\xi))-(\ldr{x^R}\beta)(X_\xi)(e)=-x^R(\xi^L(X_\beta))(e)\\
&=-\xi(\ldr{x^R}X_{\beta}(e))
=-\alpha(e)(\ldr{x^R}X_{\beta}(e))=-\ldr{x^R}(\alpha(X_\beta))(e)+(\ldr{x^R}\alpha)(X_{\beta})(e) 
=\dr_e(\beta(X_\alpha))(x)
\end{align*}
In the fifth and ninth
 equalities, we have used the fact
that $\ldr{x^R}\alpha,\ldr{x^R}\beta\in\Gamma(\mathsf{P_1})$ and Proposition \ref{De}.
\end{proof}

\bigskip

The next lemma holds for integrable Dirac Lie groups, and is in general not
true if the Dirac Lie group $(G,\mathsf D_G)$ is not integrable, as shows the
example following it. Recall that $N$ is the normal subgroup of $G$
defined by the integral leaf through $e$ of the integrable subbundle
$\mathsf{G_0}\subseteq TG$.
\begin{lemma}\label{lemma_on_action_of_G_0}
If $(G,\mathsf{D}_G)$ is integrable, then we have 
$(\ldr{x^L}X,\ldr{x^L}\alpha)$ and
$(\ldr{x^R}X,\ldr{x^R}\alpha)\in\Gamma(\mathsf{D}_G)$ 
for all $x\in\lie
g_0$ and $(X,\alpha)\in\Gamma(\mathsf{D}_G)$, and  the pairs
$(R_n^*X,R_n^*\alpha)$ and $(L_n^*X,L_n^*\alpha)$ are also elements of 
$\Gamma(\mathsf{D}_G)$ for all $n\in N$.
\end{lemma}
\begin{proof} 
The right and left invariant vector fields $x^R$ and $x^L$ defined on $G$ by
an element of $\lie g_0$ are sections of $\mathsf{G_0}$ since we have shown in
Proposition \ref{oncardis}
that $\mathsf{G_0}=\lie g_0^R=\lie g_0^L$.
If $(G,\mathsf D_G)$ is integrable, we have 
$\left[(x^L,0),(X,\alpha)\right]=\left(\ldr{x^L}X,\ldr{x^L}\alpha\right)$ and 
$\left[(x^R,0),(X,\alpha)\right]=\left(\ldr{x^R}X,\ldr{x^R}\alpha\right)\in\Gamma(\mathsf{D}_G)$
for all $(X,\alpha)\in\Gamma(\mathsf D_G)$.

In \cite{JoRa10b}, it is proved that 
 an integrable Dirac structure $\mathsf D$  is conserved along the
flow of the vector fields  $X\in\Gamma(\mathsf{G_0})$.

For each $x\in\lie g_0$, the flow of $x^L$ is $R_{\exp(tx)}$ and the flow 
of $x^R$ is $L_{\exp(tx)}$. Thus, we have \linebreak
$(R_{\exp(tx)}^*X,R_{\exp(tx)}^*\alpha)$ and
$(L_{\exp(tx)}^*X,L_{\exp(tx)}^*\alpha)\in
\Gamma(\mathsf D_G)$. This yields the claim since $N$ is generated as a
group by the elements $\exp(tx)$, $x\in\lie g_0$ and small $t$.
\end{proof}
\begin{example}\label{counter_ex_action_of_N}
Consider the  Dirac structure $\mathsf D_{\R^3}$ defined on the Lie group
$\R^3$ as the 
span of the sections
\[ \left(\partial_z,0\right),\quad
\left(z\partial_x,\dr y\right),\quad
\left(-z\partial_y,\dr x\right)
\]
of $\mathsf P_{\R^3}$. It is easy to show that    $(\R^3,\mathsf D_{\R^3})$ is
a Dirac Lie group (see also Corollary
\ref{abelianDiracLie} for a description of the multiplicative 
Dirac structures on $\R^n$). It is not integrable because, for instance, the bracket
of $(\partial_z,0)$ and $(z\partial_x,\dr y)$ is equal to 
$(\partial_x,0)$, which is not a section of $\mathsf D_{\R^3}$.
The Dirac structure is obviously not invariant under the action of 
$N=\{(0,0)\}\times \R$ on $\R^3$.
\end{example}

The following theorem shows how to decide if the action of $N$ on 
$(G,\mathsf D_G)$ is canonical.
\begin{theorem}\label{action_of_N}
The Dirac Lie group $(G,\mathsf D_G)$ is $N$-invariant if and only if the bracket
$[\cdot,\cdot]$ defined in Definition \ref{defbracket} has image in $\lie p_1$.
\end{theorem}

\begin{example}
Consider again Example \ref{counter_ex_action_of_N}. The bracket
on $\lie p_1=\erz\{\dr x(0),\dr y(0)\}$ is given by 
$[\dr y(0),\dr x(0)]=\dr_0\left((\dr x)(z\partial_x)\right)=\dr
z(0)\not\in\lie p_1$.
\end{example}

For the proof of this theorem, we need to introduce a new notation and show a lemma,
that will also be useful in the following.
\begin{definition}
Choose $\xi\in\lie p_1$ and $x\in\lie g$. Then the elements
\[\ad_x^*\xi\in\lie p_1\qquad\text{ and }\qquad\ad_\xi^*x\in\lie p_1^*
=\lie g/\lie g_0
\]
are defined by 
$(\ad_x^*\xi)(y)=\xi([y,x])$ for all $y\in\lie g$, and 
$(\ad_\xi^*x)(\eta)=[\eta,\xi](x)$ for all $\eta\in\lie p_1$.
\end{definition}
Note that $\ad_x^*\xi$ is an element of $\lie p_1$ by Corollary \ref{adG_invariance}.

\begin{lemma}\label{lie_der_of_Xmu} 
Choose $\xi\in\lie p_1$ and $X_\xi\in\mx(G)$ such that
$(X_\xi,\xi^L)\in\Gamma(\mathsf{D}_G)$. Then we have for all $x\in\lie g$:
\[(\ldr{x^L}X_\xi)(e)+\lie g_0
=-\ad_\xi^*x\,\in \lie g/\lie g_0
\] and consequently
 \begin{equation}\label{left_inv_der_of_Xmu}
\ldr{x^L}X_\xi\in (-\ad_\xi^*x)^L+X_{\ad_x^*\xi}+\Gamma(\mathsf{G_0})
\end{equation}
for all $g\in G$.
\end{lemma}
\begin{proof}
Choose $\eta\in\lie p_1$ and compute
\[\eta((\ldr{x^L}X_\xi)(e))
=\ldr{x^L}(\eta^L(X_\xi))(e)-(\ldr{x^L}\eta^L)(X_\xi)(e)
=[\xi,\eta](x)-0=\eta\,(-\ad_\xi^*x).
\]
This yields the first equality.
Using this and the proof of Proposition \ref{lemma}, we get:
\begin{align*}
\eta^L(\ldr{x^L}X_\xi)
=&-(\ad_{x}^*\xi)^L(X_\eta)+\eta\left(\ldr{x^L}X_\xi(e)\right)
=\,\eta^L\left(X_{\ad_x^*\xi}\right)+\eta^L\left((-\ad_\xi^*x)^L\right)
=\,\eta^L\left(X_{\ad_x^*\xi}+(-\ad_\xi^*x)^L\right).
\end{align*}
Since the left-invariant one-forms $\eta^L$, for all $\eta\in\lie p_1$, span 
$\Gamma(\mathsf{P_1})$ as a $\smoo$-module,\linebreak
 we have 
$\alpha\left(\ldr{x^L}X_\xi\right)=\alpha\left(X_{\ad_x^*\xi}-(\ad_\xi^*x)^L\right)
$
for all $\alpha\in\Gamma(\mathsf{P_1})$, and hence we are done because
$\mathsf{G_0}=\mathsf{P_1}^\circ$.
\end{proof}
\begin{remark}\label{remark_on_lieder_by_g0}
Note that if $x$ is an element of $\lie g_0$, we have 
$(\ldr{x^L}X_\xi,\ldr{x^L}\xi^L)=[(x^L,0),(X_\xi,\xi^L)]\in\Gamma(\mathsf{D}_G)$
if $\mathsf D_G$ is integrable. Since $x$ lies in the ideal $\lie g_0$ and $\xi\in\lie p_1=\lie g_0^\circ$, 
we have
$\ad_x^*\xi=0$, thus $\ldr{x^L}\xi^L=(\ad_x^*\xi)^L=0$ and we get
$\ldr{x^L}X_\xi\in\Gamma(\mathsf{G_0})$.

With Lemma \ref{lie_der_of_Xmu}, we can show that this is true without the
assumption that $\mathsf D_G$ is integrable; we need only the hypothesis that
the bracket on $\lie p_1$ has image in $\lie p_1$. 
We have then 
\[\ldr{x^L}X_\xi\in X_{\ad_x^*\xi}-(\ad_\xi^*x)^L+\Gamma(\mathsf{G_0})
=X_0-(\ad_\xi^*x)^L+\Gamma(\mathsf{G_0})=\Gamma(\mathsf{G_0}).
\] The vector field 
$X_0$ is indeed an element of  $\Gamma(\mathsf{G_0})$ by definition,  and for all
$\eta\in\lie p_1$, we have $(\ad_\xi^*x)(\eta)=[\eta,\xi](x)=0$ since
$[\eta,\xi]\in\lie p_1$ and $x\in\lie g_0$,
which shows that
$\ad_\xi^*x$ is trivial in $\lie g/\lie g_0$ and thus $(\ad_\xi^*x)^L\in\Gamma(\mathsf{G_0})$.
\end{remark}

\begin{proof}[of Theorem \ref{action_of_N}]
If the right action of $N$ on $(G,\mathsf D_G)$ is canonical, we have 
$(R_n^*X_\xi,R_n^*\xi^L)\in\Gamma(\mathsf D_G)$ for all $n\in N$ and
$\xi\in\lie p_1$. This yields
$(\ldr{x^L}X_\xi,\ldr{x^L}\xi^L)\in\Gamma(\mathsf D_G)$
for all $x\in\lie g_0$. Since $\mathsf D_G(e)=\lie g_0\times\lie p_1$, we get
$\ldr{x^L}X_\xi(e)\in\lie g_0$.
Hence, we have 
$[\xi,\eta](x)=x^L(\eta^L(X_\xi))(e)=(\ad_x^*\eta)(X_\xi(e))+\eta(\ldr{x^L}X_\xi(e))=0
$ for all $\xi,\eta\in\lie p_1$ and $x\in\lie g_0$ and consequently
 $[\xi,\eta]\in\lie p_1$.

Conversely, if $[\xi,\eta]\in\lie p_1$ for all $\xi,\eta\in\lie
p_1$,
we get $\ldr{x^L}X_\xi\in\Gamma(\mathsf{G_0})$ by
Remark \ref{remark_on_lieder_by_g0}.
Hence, recalling that $\ad_x^*\xi=0$ for $x\in\lie g_0$ and $\xi\in\lie p_1$,
we can compute
\begin{align*}
\frac{d}{dt}\left\langle
  (R_{\exp(tx)}^*X_\xi,R_{\exp(tx)}^*\xi^L),(X_\eta,\eta^L)\right\rangle(g)
&=\frac{d}{dt}\left(\eta^L(g)(R_{\exp(tx)}^*X_\xi)(g)+(R_{\exp(tx)}^*\xi^L)(g)(X_\eta(g))\right)\\
&=\eta^L\left(R_{\exp(tx)}^*(\ldr{x^L}X_\xi)\right)(g)+(R_{\exp(tx)}^*(\ad_x^*\xi)^L)(X_\eta)(g)\\
&=\eta\circ T_gL_{g\inv}\circ
T_{g\exp(tx)}R_{\exp(-tx)}\left(\ldr{x^L}X_\xi\right)(g\exp(tx))\\
&=(\Ad_{\exp(tx)}^*\eta)^L\left(\ldr{x^L}X_\xi\right)(g\exp(tx))=0
\end{align*}
since $\Ad_{\exp(tx)}^*\eta\in\lie p_1$ and
$\ldr{x^L}X_\xi\in\Gamma(\mathsf{G_0})$. But this yields 
\begin{align*}
\left\langle
  (R_{\exp(tx)}^*X_\xi,R_{\exp(tx)}^*\xi^L),(X_\eta,\eta^L)\right\rangle(g)&=
\left\langle
  (R_{\exp(0\cdot x)}^*X_\xi,R_{\exp(0\cdot
    x)}^*\xi^L),(X_\eta,\eta^L)\right\rangle(g)=\left\langle
  (X_\xi,\xi^L),(X_\eta,\eta^L)\right\rangle(g)=0
\end{align*}
for all $t\in\R$, which shows that $(R_{\exp(tx)}^*X_\xi,R_{\exp(tx)}^*\xi^L)$
is a section of $\mathsf{D}_G$ for all $t\in \R$. Hence, since $N$ is
generated as a group by the elements $\exp(tx)$, for $x\in\lie g_0$ and small
$t\in\R$, the proof is finished.
\end{proof}

The following theorem will be useful in the next section about integrable Dirac
Lie groups.
\begin{theorem}\label{cocycle}
The equality
\begin{align}
[\xi,\eta]([x,y])=&(\ad_y^*\xi)(\ad_{\eta}^*x)-(\ad_x^*\xi)(\ad_{\eta}^*y)
\label{bracket_of_bracket}+(\ad_x^*\eta)(\ad_{\xi}^*y)-(\ad_y^*\eta)(\ad_{\xi}^*x)
\end{align}
holds for all $\xi, \eta\in\lie p_1$ and $x,y\in\lie g$. 
\end{theorem}

\begin{proof}
By Definition \ref{defbracket}, we have
$[\xi,\eta]([x,y])=[x,y]^L\left( \eta^L(X_\xi)\right)(e)
$
for any $x,y\in\lie g$ and $\xi,\eta\in\lie p_1$.  
Hence we can compute
\begin{align*}
[\xi,\eta]([x,y])
=&[x,y]^L\left( \eta^L(X_\xi)\right)(e)
=\ldr{x^L}\ldr{y^L}\left(\eta^L(X_\xi)\right)(e)
-\ldr{y^L}\ldr{x^L}\left(\eta^L(X_\xi)\right)(e)\\
=&\ldr{x^L}\left(\ldr{y^L}\eta^L(X_\xi)+\eta^L(\ldr{y^L}X_\xi)
\right)(e)
-\ldr{y^L}\left(\ldr{x^L}\eta^L(X_\xi)+\eta^L(\ldr{x^L}X_\xi)
\right)(e)\\
=&\ldr{x^L}\left((\ad_y^*\eta)^L(X_\xi)+\eta^L(-\ad_\xi^*y)^L
+\eta^L\left(X_{\ad_y^*\xi}\right)
\right)(e)\\
&-\ldr{y^L}\left((\ad_x^*\eta)^L(X_\xi)+\eta^L(-\ad_\xi^*x)^L+\eta^L\left(X_{\ad_x^*\xi}\right)
\right)(e).
\end{align*}
Since $\eta^L(-\ad_\xi^*y)^L$ and $\eta^L(-\ad_\xi^*x)^L$ are constant
functions on $G$, we get hence:
\begin{align*}
[\xi,\eta]([x,y])
=&\ldr{x^L}\left((\ad_y^*\eta)^L(X_\xi)+\eta^L\left(X_{\ad_y^*\xi}\right)
\right)(e)-\ldr{y^L}\left((\ad_x\eta)^L(X_\xi)+\eta^L\left(X_{\ad_x^*\xi}\right)
\right)(e)\\
=&\Bigl(\ldr{x^L}(\ad_y^*\eta)^L(X_\xi)+(\ad_y^*\eta)^L(\ldr{x^L}X_\xi)+
(\ldr{x^L}\eta^L)\left(X_{\ad_y^*\xi}\right)+\eta^L\left(\ldr{x^L}X_{\ad_y^*\xi}
\right)\Bigr)(e)\\
&-\Bigl(\ldr{y^L}(\ad_x^*\eta)^L(X_\xi)+(\ad_x^*\eta)^L(\ldr{y^L}X_\xi)+
(\ldr{y^L}\eta^L)\left(X_{\ad_x^*\xi}\right)+\eta^L\left(\ldr{y^L}X_{\ad_x^*\xi}
\right)\Bigr)(e)\\
=&(\ad_x^*\ad_y^*\eta)(X_\xi(e))-(\ad_y^*\eta)(\ad_\xi^*x)+
(\ad_x^*\eta)\left(X_{\ad_y^*\xi}(e)\right)-\eta\left(\ad_ {\ad_y^*\xi}^*x
\right)\\
&-(\ad_y^*\ad_x^*\eta)(X_\xi(e))+(\ad_x^*\eta)(\ad_\xi^*y)
-(\ad_y^*\eta)\left(X_{\ad_x^*\xi}(e)\right)+\eta\left(\ad_ {\ad_x^*\xi}^*y
\right).
\end{align*}
Since  $\mathsf{D}_G(e)=\lie g_0\times\lie p_1$ and $\lie p_1$ is
$\ad_x^*$-invariant  for all $x\in \lie g$, 
the first,
third, fifth and seventh terms of this sum vanish. Thus, we get
\begin{align*}
[\xi,\eta]([x,y])&=-(\ad_y^*\eta)(\ad_\xi^*x)
-\eta\left(\ad_ {\ad_y^*\xi}^*x
\right)+(\ad_x^*\eta)(\ad_\xi^*y)+\eta\left(\ad_ {\ad_x^*\xi}^*y
\right)\\
&=-(\ad_y^*\eta)(\ad_\xi^*x)+[\ad_y^*\xi,\eta](x)
+(\ad_x^*\eta)(\ad_\xi^*y)-[\ad_x^*\xi,\eta](y)\\
&=-(\ad_y^*\eta)(\ad_\xi^*x)+(\ad_y^*\xi)(\ad_\eta^*x)
+(\ad_x^*\eta)(\ad_\xi^*y)-(\ad_x^*\xi)(\ad_\eta^*y).
\end{align*}
\end{proof}
\begin{remark}
Equation \eqref{bracket_of_bracket} is equivalent to either one of the
following equations for all $x,y\in\lie g$ and $\xi,\eta\in\lie p_1$:
\begin{equation}\label{ppart}
\ad_x^*([\xi,\eta])=[\ad_x^*\xi,\eta]-[\ad_x^*\eta,\xi]+\ad_{\ad_\eta^*x}^*\xi
-\ad_{\ad_\xi^*x}^*\eta,
\end{equation}
\begin{equation}\label{gpart}
\ad_\xi^*([x,y])
=[\ad_\xi^*x,y]-[\ad_\xi^*y,x]-\ad_{\ad_x^*\xi}^*y+\ad_{\ad_y^*\xi}^*x
\,\,\in\,\,\lie p_1^*=\lie g/\lie g_0
\end{equation}
\end{remark}

Let $(G,\mathsf D_G)$ be a Dirac Lie group. Then the space $\bigwedge^2\lie
g/\lie g_0$ is a $G$-module via $$g\cdot \left((x+\lie g_0)\wedge(y+\lie g_0)\right)
=(\Ad_gx+\lie g_0)\wedge(\Ad_gy+\lie g_0)$$ by Proposition
\ref{oncardis}, and by derivation, it is a $\lie g$-module 
via
$$z\cdot \left((x+\lie g_0)\wedge(y+\lie g_0)\right)
=([z,x]+\lie g_0)\wedge(y+\lie g_0)+(x+\lie g_0)\wedge([z,y]+\lie g_0).$$
Theorem \ref{cocycle} states then that the map $\delta:\lie g\to \bigwedge^2\lie
g/\lie g_0$ defined as the dual map of $[\cdot,\cdot]:\bigwedge^2 \lie
p_1\to\lie g^*$ is a Lie algebra $1$-cocycle,
 that is, we have
\[\delta([x,y])=x\cdot\delta(y)-y\cdot\delta(x)
\]
for all $x,y\in\lie g$. Hence, we can associate to each Dirac Lie group
$(G,\mathsf D_G)$ an ideal $\lie g_0$ and a Lie algebra $1$-cocycle $\delta:\lie g\to \bigwedge^2\lie
g/\lie g_0$.
If $(G,\mathsf D_G)$ is a Dirac Lie group, then the map 
$C:G\to \bigwedge^2\lie
g/\lie g_0$ defined by $C(g)(\xi,\eta)=\eta^R(Y_\xi)(g)=-\xi^R(Y_\eta)(g)$, where $Y_\xi\in\mx(G)$
is a vector field satisfying $(Y_\xi,\xi^R)\in\Gamma(\mathsf D_G)$, is a Lie group $1$-cocycle, i.e., it satisfies 
$C(gh)=C(g)+\Ad_gC(h)$ for all $g,h\in G$. The proof of this uses Remark
\ref{right_inv_ham}.
We have $C(e)=0\in\lie g/\lie g_0\wedge\lie g/\lie g_0$ by Proposition 
\ref{De}  and
$$(\dr_eC)(x)(\xi,\eta)=\left.\frac{d}{dt}\right\an{t=0}C(\exp(tx))(\xi,\eta)
=\left.\frac{d}{dt}\right\an{t=0}\eta^R(Y_\xi)(\exp(tx))
\overset{\text{Prop. }\ref{other_formula_for_the_bracket}}=[\xi,\eta](x)=\delta(x)(\xi,\eta)
$$ for all $\xi,\eta\in\lie p_1$.

Note that if $G$ is connected and $C:G\to\bigwedge^2\lie
g/\lie g_0$ is a Lie group $1$-cocycle integrating $\delta$, that is, with
$C(e)=0$ and
$\dr_eC=\delta$, then $C$ is unique (see \cite{Lu90}) and $\mathsf D_G$ is consequently given on $G$ by
\begin{equation}\label{cocycle_defines_Dirac}
\mathsf D_G(g)=\left\{(T_eR_g(C(g)^\sharp(\xi)+x),\xi^R(g))\mid \xi\in\lie p_1,
x\in\lie g_0\right\}
\end{equation} 
for all $g\in G$, where $C(g)^\sharp:\lie g/\lie g_0\to\lie g$ is defined as follows.
Choose a vector subspace $W\subseteq \lie g$ such that
$\lie g=\lie g_0\oplus W$, then we have an isomorphism $\phi_W:W\to\lie g/\lie g_0$,
$w\mapsto w+\lie g_0$. Set $C(g)^\sharp(\xi)=\phi_W\inv(C(g)(\xi,\cdot))$ for all
$\xi\in\lie p_1=\lie g_0^\circ$. Note that  by definition,
\eqref{cocycle_defines_Dirac} does not depend on the choice
of $W$.

Conversely, let $G$ be a connected and simply connected Lie group and $\lie
g_0$ an ideal in $\lie g$. Choose a Lie algebra $1$-cocycle $\delta:\lie g\to\bigwedge^2\lie
g/\lie g_0$. Then there exists a unique Lie group $1$-cocycle $C:G\to   \bigwedge^2\lie
g/\lie g_0$ integrating $\delta$ (see \cite{DuZu05}). Define 
$\mathsf D_G\subseteq \mathsf P_G$ by
\eqref{cocycle_defines_Dirac}. Then it is easy to check that
$\mathsf D_G$ is a multiplicative Dirac structure on $G$. We have shown the
following theorem.
\begin{theorem}
Let $G$ be  a connected and simply connected Lie group with Lie algebra $\lie g$.
Then we have a one-to-one correspondence
\[\left\{
\left(\lie g_0,\,\delta:\lie g\to \bigwedge{}\!^2\lie
g/\lie g_0\right)\left|\begin{array}{c}
\lie g_0\subseteq \lie g \text{ ideal },\\
\delta \text{ cocycle }
\end{array}\right.\right\}
\overset{1:1}\leftrightarrow \left\{
\begin{array}{c}
\text{ multiplicative Dirac structures }\\
\text{ on $G$
}\end{array}
\right\}.
\]
\end{theorem}
We will see in the next subsection that the \emph{integrable}
 multiplicative Dirac structures on $G$ correspond via this
bijection to the pairs $\left(\lie g_0,\,\delta:\lie g\to \bigwedge^2\lie
g/\lie g_0\right)$ such that the dual $[\cdot,\cdot]:=\delta^*:\bigwedge^2\lie
g_0^\circ\to
\lie g^*$ defines a Lie bracket on $\lie g_0^\circ=:\lie p_1$.

\begin{example}
\begin{enumerate}
\item Let $G=\R^n$. Then any vector subspace $V\subseteq \R^n\simeq T_0\R^n$ is an
ideal
in $\lie g=\R^n$ and any $\delta: \R^n\to \bigwedge ^2\R^n/V$ is a Lie algebra $1$-cocycle
since the cocycle condition is trivial in this particular case. The Lie group $1$-cocycle $C$ integrating $\delta$
is then the unique linear map $C:\R^n\to \bigwedge
^2\R^n/V$ with $\dr_0 C=\delta$, that is, $C$ is equal to $\delta$ if we
identify $G=\R^n$ with $\lie g=T_0\R^n$ via the exponential map. This shows
that each multiplicative Dirac structure on $\R^n$ is given by
\[\mathsf D_{\R^n}(r)=\left\{(\delta(r)^\sharp(\xi)+x),\xi)\mid \xi\in V^\circ,
x\in V\right\}\subseteq T_r\R^n\times T_r^*\R^n,
\]
with $V$ a vector subspace of $\R^n$, $\delta:\R^n\to \bigwedge^2 \R^n/V$ a linear 
map and $\delta(r)^\sharp$ defined as in \eqref{cocycle_defines_Dirac} with
a complement $W$ of $V$ in $\R^n$.
\item Let $G\subseteq \operatorname{GL}_n(\R)$ be the set of upper triangular
matrices with non-vanishing determinant. The Lie algebra $\lie g$ of $G$ is then the 
set of upper triangular matrices. Its commutator $\lie g_0:=[\lie g,\lie g]$ 
is the set of strictly 
upper triangular matrices, and integrates to the normal subgroup $N\subseteq G$ 
of upper triangular matrices with
all  entries  on the diagonal equal to $1$.
 Note that $G$ is not connected. The connected component
of the neutral element $e\in G$ is the set of upper triangular matrices with
strictly positive diagonal entries. 

The quotient $\lie g/\lie g_0$ is isomorphic to the set of diagonal matrices 
in $\lie g$. Since $\lie g_0=[\lie g,\lie g]$, it is easy to see that the 
cocycle condition is satisfied for a linear
map $\delta:\lie g\to\bigwedge^2\lie g/\lie g_0$ if and only if $\delta\an{\lie g_0}$
vanishes, that is, if and only if $\delta$ factors to
$\bar\delta:\lie g/\lie g_0\to\bigwedge^2\lie g/\lie g_0$. 
In other words, the dual map $[\cdot\,,\cdot]$ of a Lie algebra $1$-cocycle $\delta$
 has necessarily image in $\lie p_1:=\lie g_0^\circ$. 
Hence, if $C:G\to\bigwedge^2\lie g/\lie g_0$ is the Lie algebra $1$-cocycle associated to 
a multiplicative Dirac structure $\mathsf D_G$ on $G$, then the dual of its
derivative at $e$ has image in $\lie p_1$, that is, the bracket on $\lie p_1$
defined in Definition \ref{defbracket} has automatically image in $\lie p_1$.
Since $N$ is closed in $G$, this shows by Remark \ref{quotient} 
and Theorem \ref{action_of_N}
that any multiplicative Dirac structure on $G$ 
with $\lie g_0=[\lie g, \lie g]$ 
 is automatically the pullback
to $G$ under $q:G\to G/N$ of the graph of
a multiplicative bivector field on $G/N$.

More generally, 
this result holds for any Dirac Lie group $(G,\mathsf D_G)$ such that
$\lie g_0=[\lie g,\lie g]$. 
\end{enumerate}
\end{example}
\subsection{Integrable Dirac Lie groups: induced Lie bialgebra}
In continuation of the results in the preceding subsection, we can show that
 the integrability of $(G,\mathsf D_G)$  depends only on the properties of the
bracket defined in Definition \ref{defbracket}.
\begin{theorem}\label{liealgebra}
The Dirac Lie group $(G,\mathsf{D}_G)$ is integrable if and only if the bracket
 $[\cdot\,,\cdot]$ on $\lie p_1\times\lie p_1$ defined in Definition
\ref{defbracket} is a Lie bracket on $\lie p_1$.
\end{theorem}
  In this case, Theorem \ref{cocycle} implies that the pair 
$(\lie g/\lie g_0,\lie p_1)$ is a Lie bialgebra.

Of course, with Theorem \ref{action_of_N}, we could show this theorem by
considering the Lie bialgebra structure defined on $(\lie g/\lie g_0, \lie p_1)$ by
the multiplicative Poisson  structure on $\tilde G/\tilde N$ (recall Proposition
\ref{prop_univ_cover} and Remarks \ref{quotient} and \ref{rem_univ_cover}),  
but we prefer to do that in the setting of  Dirac manifolds.

For the proof of the theorem, we will need the following lemmas concerning the
tensor $T_{\mathsf D_G}$ (see Subsection \ref{subsection_dirac_structures} about Dirac manifolds).

\begin{lemma} Let $(G,\mathsf D_G)$ be a Dirac Lie group. The tensor $T_{\mathsf D_G}$ is given by
\begin{align}\label{formel}
2\,T_{\mathsf D_G}\Bigl((X_\xi,\xi^L),(X_\eta,\eta^L),(X_\zeta,\zeta^L)\Bigr)
=&X_\zeta(\eta^L(X_\xi))+X_\xi(\zeta^L(X_\eta))+X_\eta(\xi^L(X_\zeta))\nonumber\\
&\,-[\zeta,\eta]^L(X_\xi)-[\xi,\zeta]^L(X_\eta)-[\eta,\xi]^L(X_\zeta)
\end{align}
for all $\xi,\eta,\zeta\in\lie p_1$ and corresponding $X_\xi,X_\eta,X_\zeta\in\mx(G)$,
and in particular
\begin{align}
T_{\mathsf D_G}(e)((x,\xi),(y,\eta),(z,\zeta))
=&[\xi,\eta](z)+[\eta,\zeta](x)+[\zeta,\xi](y)\label{formel1}
\end{align}
for any $x,y,z\in\lie g_0$.
\end{lemma}
\begin{proof}
Choose
$\xi,\eta,\zeta\in\lie p_1$  and corresponding vector fields 
$X_\xi,X_\eta,X_\zeta\in\mx(G)$. Then the following equality is easy to show:
\begin{align*}
T_{\mathsf D_G}\Bigl((X_\xi,\xi^L),(X_\eta,\eta^L),(X_\zeta,\zeta^L)\Bigr)
=&\zeta^L([X_\xi,X_\eta])+\xi^L([X_\eta,X_\zeta])+\eta^L([X_\zeta,X_\xi])\\
&+X_\zeta(\xi^L(X_\eta))+X_\xi(\eta^L(X_\zeta))+X_\eta(\zeta^L(X_\xi)).
\end{align*}
Using the definition of the bracket, we have also 
\begin{align*}
 &[\zeta,\eta]^L(X_\xi)+[\xi,\zeta]^L(X_\eta)+[\eta,\xi]^L(X_\zeta)\\
=&(\ldr{X_\zeta}\eta^L-\ip{X_\eta}\dr\zeta^L)(X_\xi)
+(\ldr{X_\xi}\zeta^L-\ip{X_\zeta}\dr\xi^L)(X_\eta)
+(\ldr{X_\eta}\xi^L-\ip{X_\xi}\dr\eta^L)(X_\zeta)\\
=&2\,
\Bigl(\zeta^L([X_\eta,X_\xi])+\eta^L([X_\xi,X_\zeta])+\xi^L([X_\zeta,X_\eta])\Bigr)\\
&+3\,\Bigl(X_\zeta(\eta^L(X_\xi))+X_\xi(\zeta^L(X_\eta))+X_\eta(\xi^L(X_\zeta))
\Bigr)\\
=&-2\,T_{\mathsf D_G}\Bigl((X_\xi,\xi^L),(X_\eta,\eta^L),(X_\zeta,\zeta^L)\Bigr)
+X_\zeta(\eta^L(X_\xi))+X_\xi(\zeta^L(X_\eta))+X_\eta(\xi^L(X_\zeta)).
\end{align*}
Evaluated at $e$, this leads to
\begin{align*}
T_{\mathsf D_G}(e)((X_\xi(e),\xi),(X_\eta(e),\eta),(X_\zeta(e),\zeta))
=&\frac{1}{2}
\Bigl(\dr_e(\eta^L(X_\xi))(X_\zeta(e))
+\dr_e(\zeta^L(X_\eta))(X_\xi(e))+\dr_e(\xi^L(X_\zeta))(X_\eta(e))\\
&\qquad-[\zeta,\eta](X_\xi(e))-[\xi,\zeta](X_\eta(e))-[\eta,\xi](X_\zeta(e))\Bigr)\\
=&[\xi,\eta](X_\zeta(e))+[\eta,\zeta](X_\xi(e))+[\zeta,\xi](X_\eta(e)).
\end{align*} 
\end{proof}

\begin{lemma}\label{lemmaS}
Assume that the bracket on $\lie p_1\times \lie p_1$ has image in $\lie p_1$. 
Then,
$T_{\mathsf D_G}\Bigl((X_\xi,\xi^L),(X_\eta,\eta^L),(X_\zeta,\zeta^L)\Bigr)$
is independant of the choice of the vector fields
$X_\xi,X_\eta,X_\zeta\in\mx(G)$. The tensor $T_{\mathsf D_G}$
defines in this case a tensor $S_{\mathsf D_G}\in\Gamma(\bigwedge^3\lie \mathsf{P_1}^*)$
by
\[S_{\mathsf D_G}(\xi^L,\eta^L,\zeta^L)
=T_{\mathsf D_G}\Bigl((X_\xi,\xi^L),(X_\eta,\eta^L),(X_\zeta,\zeta^L)\Bigr)
\]
for all $\xi,\eta,\zeta\in\lie p_1$ and $(G,\mathsf D_G)$ is integrable if and only
if $S_{\mathsf D_G}$ vanishes on $G$.
\end{lemma}

\begin{proof}
Consider \eqref{formel}. Since $[\xi,\eta]\in\lie p_1$, we have
$[\xi,\eta]^L(X_\zeta+Z)=[\xi,\eta]^L(X_\zeta)$ for all $Z\in\Gamma(\mathsf{G_0})$.
Thus, we have only to show
that $X_\xi(\eta^L(X_\zeta)$ is independent of the choices of $X_\xi,X_\zeta$. 
Choose $Z$ and $W\in\Gamma(\mathsf{G_0})$ and compute
\begin{align*}
(X_\xi+Z)(\eta^L(X_\zeta+W))
&=X_\xi(\eta^L(X_\zeta))+Z(\eta^L(X_\zeta))+(X_\xi+Z)(\eta^L(W))
=X_\xi(\eta^L(X_\zeta)+Z(\eta^L(X_\zeta))
\end{align*}
since $\eta^L(W)=0$. For any $x\in\lie g_0$, we have
$x^L(\eta^L(X_\zeta))=(\ad_x^*\eta)^L(X_\zeta)+\eta^L(\ldr{x^L}X_\zeta)=0$
since $\ad_x^*\eta=0$ and $\ldr{x^L}X_\zeta\in\Gamma(\mathsf{G_0})$ by Remark
\ref{remark_on_lieder_by_g0}. Since $\Gamma(\mathsf{G_0})$ is spanned as
a $\smoo$-module by $\{x^L\mid x\in\lie g_0\}$, we are done. 

Recall that the pairs $(x^L,0)$ and $(X_\xi,\xi^L)$, for all $x\in\lie g_0$
and $\xi\in\lie p_1$ span the Dirac bundle $\mathsf D_G$. Hence, to prove
integrability of $\mathsf D_G$, we have only to show that
the Courant bracket of two sections of $\mathsf D_G$ of this type is a section
of $\mathsf D_G$.
We have already
$[(x_1^L,0),(x_2^L,0)]\in\Gamma(\mathsf D_G)$ for all $x_1,x_2\in\lie g_0$ 
since $\lie g_0$ is an ideal in
$\lie g$ and 
$[(x^L,0),(X_\xi,\xi^L)]=(\ldr{x^L}X_\xi,(\ad_x^*\xi)^L)=(\ldr{x^L}X_\xi,0)
\in\Gamma(\mathsf D_G)$ by Remark
\ref{remark_on_lieder_by_g0} .
Thus, we have only to show that
$([X_\xi,X_\eta],[\xi,\eta]^L)$ is a section of $\mathsf D_G$ for all
$\xi,\eta\in\lie p_1$. Since $[\xi,\eta]\in\lie p_1$, we have 
$\left\langle (x^L,0),([X_\xi,X_\eta],[\xi,\eta]^L)\right\rangle=
[\xi,\eta](x)=0$ for all $x\in \lie g_0$. The Dirac structure $\mathsf D_G$ is
thus integrable if and only if 
$\left\langle([X_\xi,X_\eta],[\xi,\eta]^L),(X_\zeta,\zeta^L)\right\rangle=0$
for all $\xi,\eta,\zeta\in\lie p_1$, that is, if and only if $S_{\mathsf D_G}=0$.
\end{proof}

\begin{proof}[of Theorem \ref{liealgebra}] 
We have to show that the bracket has image in $\lie p_1$ and satisfies the Jacobi 
identity if and only if the Dirac Lie group $(G,\mathsf D_G)$ is integrable.

Assume first that $(G,\mathsf D_G)$ is integrable. The tensor 
$T_{\mathsf D_G}$ vanishes
identically on $G$ and 
\[\left[(X_\xi,\xi^L),(X_\eta,\eta^L)\right]
=\left([X_\xi,X_\eta],\ldr{X_\xi}\eta^L-\ip{X_\eta}\dr\xi^L\right)
=\left([X_\xi,X_\eta],[\xi,\eta]^L\right)\]
is a section of $\mathsf D_G$ for any $\xi,\eta\in\lie p_1$. Hence the covector 
$[\xi,\eta]=\dr_e(\eta^L(X_\xi))$ is an element of $\lie p_1$ for all 
$\xi,\eta\in\lie p_1$.
We get then using \eqref{formel}
\begin{align*}
&[\xi,[\zeta,\eta]]+[\eta,[\xi,\zeta]]+[\zeta,[\eta,\xi]]
=\,\dr_e\left([\zeta,\eta]^L(X_\xi)\right)+\dr_e\left([\xi,\zeta]^L(X_\eta)\right)
+\dr_e\left([\eta,\xi]^L(X_\zeta)\right)\\
=&-2\dr_e\left(T_{\mathsf D_G}\Bigl((X_\xi,\xi^L),(X_\eta,\eta^L),
(X_\zeta,\zeta^L)\Bigr)\right)
+\dr_e\left(X_\zeta(\eta^L(X_\xi))+X_\xi(\zeta^L(X_\eta))+X_\eta(\xi^L(X_\zeta))\right)
\end{align*}
for all $\xi,\zeta,\eta\in\lie p_1$. 
We have for any $x\in \lie g$:
\begin{align*}
&\dr_e\left(X_\zeta(\eta^L(X_\xi))\right)(x)=x^L\left(X_\zeta(\eta^L(X_\xi))\right)(e)\\
\overset{\eqref{left_inv_der_of_Xmu}}
=&\,\dr_e\left(\eta^L(X_\xi)\right)(-\ad_\zeta^*x)
+\dr_e\left((\ad_x^*\eta)^L(X_\xi)\right)(X_\zeta(e))
+X_\zeta\left(\eta^L\left((-\ad_\xi^*x)^L+X_{\ad_x^*\xi}
\right)\right)(e)\\
=&\,\,[\xi,\eta](-\ad_\zeta^*x)
+[\xi,\ad_x^*\eta](X_\zeta(e))
+[\ad_x^*\xi,\eta](X_\zeta(e))=[\zeta,[\xi,\eta]](x).
\end{align*}
We have used the equality $\mathsf D_G(e)=\lie g_0\times\lie p_1$ and 
$[\xi,\ad_x^*\eta],[\ad_x^*\xi,\eta]\in\lie p_1$ as we have seen above.
This leads to
\begin{align*}
[\xi,[\zeta,\eta]]+[\eta,[\xi,\zeta]]+[\zeta,[\eta,\xi]]
=-\dr_e\left(T_{\mathsf D_G}\Bigl((X_\xi,\xi^L),(X_\eta,\eta^L),
(X_\zeta,\zeta^L)\Bigr)\right)=0.
\end{align*}

\medskip

For the converse implication, we know by Lemma \ref{lemmaS} and the hypothesis that
the Lie bracket has image in $\lie p_1$ that
we have only to show the equality $S_{\mathsf D_G}=0$. We compute 
$\ldr{x^L}(S_{\mathsf D_G})$ for any
$x\in\lie g$. It is given for any $g\in G$ and $\xi,\eta,\zeta\in\lie p_1$ by
\begin{align*}
\left(\ldr{x^L}S_{\mathsf D_G}\right)(g)(\xi^L(g),\eta^L(g),\zeta^L(g))
=&\ldr{x^L}\left(S_{\mathsf D_G}(\xi^L,\eta^L,\zeta^L)\right)(g)
-S_{\mathsf D_G}((\ad_x^*\xi)^L,\eta^L,\zeta^L)(g)\\
&\,-S_{\mathsf D_G}(\xi^L,(\ad_x^*\eta)^L,\zeta^L)(g)
-S_{\mathsf D_G}(\xi^L,\eta^L,(\ad_x^*\zeta)^L)(g).
\end{align*} 

A long but straightforward calculation using the definition of 
$S_{\mathsf D_G}$, \eqref{formel}, \eqref{left_inv_der_of_Xmu}
and
\eqref{ppart} 
yields that
\begin{align*}
\ldr{x^L}\left(S_{\mathsf D_G}(\xi^L,\eta^L,\zeta^L)\right)
=&S_{\mathsf D_G}((\ad_x^*\xi)^L,\eta^L,\zeta^L)
 +S_{\mathsf D_G}(\xi^L,(\ad_x^*\eta)^L,\zeta^L)
 +S_{\mathsf D_G}(\xi^L,\eta^L,(\ad_x^*\zeta)^L)\\
&+ \Bigl([[\zeta,\eta],\xi]+[[\xi,\zeta],\eta]+[[\eta,\xi],\zeta]
 \Bigr)(x).\end{align*}
Since  $[\cdot,\,\cdot]:\lie p_1\times \lie p_1\to \lie p_1$ satisfies the
Jacobi identity by hypothesis,
this shows that $\ldr{x^L}S_{\mathsf D_G}=0$ for all $x\in\lie g$ and
consequently that $S_{\mathsf D_G}$ is right invariant. Thus, we get
\begin{align*}
S_{\mathsf D_G}(g)\left(\xi^R(g),\zeta^R(g),\eta^R(g)\right)
&=S_{\mathsf D_G}(e)(\xi,\zeta,\eta)
\overset{\eqref{formel1}}=[\zeta,\eta](X_\xi(e))+[\xi,\zeta](X_\eta(e))+[\eta,\xi](X_\zeta(e))=0
\end{align*}
since $[\cdot,\cdot]$ has image in $\lie p_1$ and $\mathsf D_G(e)=\lie
g_0\times\lie p_1$.
Hence, we have shown that $S_{\mathsf D_G}$ vanishes
identically on $G$ and the Dirac Lie group $(G,\mathsf D_G)$ is consequently
integrable by Lemma \ref{lemmaS}.
\end{proof}

\begin{remark}\label{bracket_of_Xmu}
\begin{enumerate}
\item We can see from the last proof that
\[(\ldr{x^L}S_{\mathsf D_G})(\xi^L,\zeta^L,\eta^L)
=[[\xi,\eta],\zeta]+[[\eta,\zeta],\xi]+[[\zeta,\xi],\eta]
\]
if the bracket
on $\lie p_1\times \lie p_1$ has image in $\lie p_1$.
This shows that $\ldr{x^L}S_{\mathsf D_G}$ is left-invariant and  we can see using
\eqref{formel1} and $\mathsf D_G(e)=\lie g_0\times\lie p_1$ 
that $S_{\mathsf D_G}(e)=0$. Thus, $S_{\mathsf
  D_G}\in\Gamma(\bigwedge^3\mathsf{P_1}^*)$ 
is multiplicative (see for instance \cite{Lu90}).
\item Note that if $(G,\mathsf D_G)$ is integrable, then $[\xi,\eta]\in\lie p_1$ and we get 
$X_{[\xi,\eta]}\in [X_\xi,X_\eta]+\Gamma(\mathsf{G_0})$
for all $\xi,\eta\in\lie p_1$ and any vector field $X_{[\xi,\eta]}$ such that
$(X_{[\xi,\eta]},[\xi,\eta]^L)\in\Gamma(\mathsf D_G)$.
\item If $(G,\mathsf D_G)$ is a Poisson Lie group, that is, $\mathsf D_G$ is the 
graph of the map $\pi^\sharp:T^*G\to TG$ induced by a multiplicative Poisson 
bivector field $\pi_G$ on $G$, we recover the Lie bracket defined by a Poisson Lie
group on the dual of its Lie algebra (see \cite{Lu90}). 
For all $\xi,\eta\in\lie g^*$ and
$x\in\lie g$:
\begin{align*}
[\xi,\eta](x)&= \dr_e(\eta^L(X_\xi))(x)
=x^R(\eta^L(X_\xi))(e)=x^R(\pi(\xi^L,\eta^L))(e)\\
&=(\ldr{x^R}\pi)(e)(\xi,\eta)
+\pi(e)(\ldr{x^R}\xi^L(e),\eta)+\pi(e)(\xi,\ldr{x^R}\eta^L(e))
=(\ldr{x^R}\pi)(e)(\xi,\eta)=(\dr_e\pi)^*(\xi,\eta)(x).
\end{align*}
\end{enumerate}
\end{remark}

Recall that since $\lie g_0$ is an ideal of $\lie g$, 
the quotient $\lie g/\lie g_0$ has a
canonical Lie algebra structure given by
$[x+\lie g_0,y+\lie g_0]=[x,y]+\lie g_0$ for all $x,y\in\lie g$.
Recall from Theorems \ref{cocycle} and  \ref{liealgebra}  that if $(G,\mathsf D_G)$ is integrable,
then the pair $(\lie g/\lie g_0,\lie p_1)$ is a Lie bialgebra.
The following theorem is a general fact about Lie bialgebras
(see for instance \cite{LuWe90}).
\begin{theorem}\label{bracket}
Assume that the Dirac Lie group
$(G,\mathsf D_G)$ is integrable.
The Lie algebra structures on $\lie g/\lie g_0$ and $\lie p_1$ induce a Lie algebra
structure on $\lie g/\lie g_0\times\lie p_1$,
 with the bracket given by
\[
[(x+\lie g_0,\xi),(y+\lie g_0,\eta)]
=\left(
[x,y]-\ad_\eta^*x+\ad_\xi^*y+\lie g_0,[\xi,\eta]+\ad_x^*\eta-\ad_y^*\xi
\right),
\]
for all $x,y\in\lie g$ and $\xi,\eta\in\lie p_1$.  
\end{theorem}

\subsection{The action of  $G$ on $\lie g/\lie g_0\times\lie p_1$}
\begin{theorem}\label{action_of_G_mod_G0}
Let $(G,\mathsf{D}_G)$ be a Dirac Lie group.
Define $$A:G\times (\lie g/\lie g_0\times\lie p_1)
\to\lie g/\lie g_0\times\lie p_1 $$ by
\[A(g,(x+\lie g_0,\xi))=\left(\Ad_gx+T_gR_{g^{-1}}X_\xi(g)+\lie g_0,\Ad_{g\inv}^*\xi\right)
\]
for all $g\in G$, where $X_\xi\in\mx(G)$ is a  vector field 
such that $(X_\xi,\xi^L)\in\Gamma(\mathsf D_G)$.
The map $A$ is a well-defined action of $G$ on $\lie g/\lie g_0\times\lie p_1$.
\end{theorem}

\begin{proof}
We prove first the fact that the action is well-defined, that is, that it
doesn't depend on the choices of $x$ and $X_\xi$. 
Choose $x'\in\lie g$ such that $x'+\lie g_0=x+\lie g_0$. Then 
 $x'-x=:x_0\in\lie g_0$ and 
\[\Ad_{g}x'=\Ad_g(x+x_0)=\Ad_gx+\Ad_gx_0\in\Ad_gx+\lie g_0\]
 for all $g\in G$, since $\lie g_0$ is $\Ad_g$-invariant for all $g\in G$.

Next, if $X_\xi$ and $X_\xi'\in\mx(G)$ are such that
$(X_\xi,\xi^L)$ and $(X_\xi',\xi^L)\in\Gamma(\mathsf{D}_G)$, the difference
$X_\xi'-X_\xi$ is a section of $\mathsf{G_0}$ and hence we can write $(X_\xi'-X_\xi)(g)=T_eR_gx_0$
with $x_0\in\lie g_0$.
This leads to 
\begin{align*}
T_gR_{g\inv}X_\xi'(g)=T_gR_{g\inv}X_\xi(g)+x_0\in
T_gR_{g\inv}X_\xi(g)+\lie g_0.
\end{align*}
The map $A$ is hence shown to be well-defined. We show next that $A$ is an
action of $G$ on $\lie g/\lie g_0\times\lie p_1$.
We have to show that 
\[A\left(g',A(g,(x+\lie g_0,\xi))\right)=A(g'g,(x+\lie g_0,\xi))\]
for all $g,g'\in G$, $x\in\lie g$ and $\xi\in\lie p_1$. 

We have with the same arguments as above
\begin{align*}
A_{g'g}(x+\lie
g_0,\xi)
=&\left(\Ad_{g'}(\Ad_{g}x)+T_{g'g}R_{g\inv{g'}\inv}X_\xi(g'g)+\lie g_0,
\Ad_{{g'}\inv}^*(\Ad_{g\inv}^*\xi)\right)\\
\overset{\eqref{eq4}}
=&\left(\Ad_{g'}(\Ad_{g}x)+T_{g'g}R_{g\inv{g'}\inv}(T_gL_{g'}X_\xi(g)+T_{g'}R_gX_{\Ad_{g\inv}^*\xi}(g'))
+\lie g_0,
\Ad_{{g'}\inv}^*(\Ad_{g\inv}^*\xi)\right)\\
=&\left(\Ad_{g'}(\Ad_{g}x+T_{g}R_{g\inv}X_\xi(g))+T_{g'}R_{{g'}\inv}X_{\Ad_{g\inv}^*\xi}(g')
+\lie g_0,
\Ad_{{g'}\inv}^*(\Ad_{g\inv}^*\xi)\right)\\
=&\,A_{g'}\left(\Ad_{g}x+T_{g}R_{g\inv}X_\xi(g)+\lie
  g_0,\Ad_{g\inv}^*\xi\right)=\,A_{g'}\left(A_{g}(x+\lie g_0,\xi)\right).
\end{align*}
\end{proof}
\begin{remark}\label{action_of_GmodN}
Assume that the bracket on $\lie p_1\times \lie p_1$ has image in $\lie
  p_1$.
\begin{enumerate}
\item 
We have $N\subseteq G_{(x+\lie
g_0,\xi)}$, where $G_{(x+\lie
g_0,\xi)}$ is the isotropy group of $(x+\lie
g_0,\xi)\in\lie g/\lie g_0\times\lie p_1$.

Indeed, for  $n\in N$, we have  
$
 \Ad_{n}x
\in x+\lie g_0$,
for all $x\in\lie g$
and  hence $\Ad_{n\inv}^*\xi=\xi$ for all $\xi\in\lie p_1$.
The proof of this is easy, see also
 \cite{OrRa04}, Lemma 2.1.13.
Since
$\left(X_{\xi},\xi^L\right)\in\Gamma(\mathsf{D}_G)$ and $n\in
N$, we know by Theorem \ref{action_of_N} that
$\left(R_{n}^*X_{\xi},R_{n}^*\xi^L\right)\in\Gamma(\mathsf{D}_G)$.
Hence, we have $R_{n}^*X_{\xi}(e)\in\lie g_0$ because
$\mathsf{D}_G(e)=\lie g_0\times\lie p_1$, that is, 
$T_{n}R_{n\inv}X_{\xi}(n)\in\lie g_0$.
Using this and $\Ad_nx\in x+\lie g_0$, we get 
$\Ad_{n}x+T_{n}R_{n\inv}X_\xi(n)\in x+\lie g_0$.
\item Thus, we get a well-defined action $\bar A$ of $G/N$ on $\lie g/\lie
  g_0\times\lie p_1$,
 that is given by $\bar A(gN,(x+\lie
  g_0,\xi))=A(g,(x+\lie g_0,\xi))$
for all $g\in G$.
\end{enumerate} 
\end{remark}

In fact $(G,\mathsf D_G)$ is integrable and  $N$ is closed in $G$, the next theorem shows that
$\bar A$ is  the adjoint action of $G/N$ on $\lie g/\lie g_0\times\lie
p_1$ integrating the adjoint action of $\lie g/\lie g_0$ defined by the
bracket on $\lie g/\lie g_0\times\lie p_1$.

 \begin{theorem}
Assume that $(G,\mathsf D_G)$ is an integrable Dirac Lie group.
The adjoint action of $\lie g/\lie g_0\simeq \lie g/\lie g_0\times\{0\}\subseteq\lie
g/\lie g_0\times\lie p_1 $ on $\lie
g/\lie g_0\times\lie p_1 $ ``integrates'' to the action $A$ of $G$ on 
$\lie g/\lie g_0\times\lie p_1 $ in the sense that
\begin{align*}
\left.\frac{d}{dt}\right\an{t=0}A\left(\exp(ty),(x+\lie g_0,\xi)\right)
=\left[(y+\lie g_0,0),(x+\lie g_0,\xi)\right]
\end{align*}
for all $y\in\lie g$ and $(x+\lie g_0,\xi)\in\lie g/\lie g_0\times\lie p_1$.
\end{theorem}

\begin{proof}
Choose  $x,y\in\lie g$ and $\xi\in\lie p_1$ and compute
\begin{align*}
\left.\frac{d}{dt}\right\an{t=0}A\left(\exp(ty),(x+\lie g_0,\xi)\right)
=&\left.\frac{d}{dt}\right\an{t=0}\left(\Ad_{\exp(ty)}x
+T_{\exp(ty)}R_{\exp(-ty)}X_\xi(\exp(ty))+\lie
g_0, \Ad_{\exp(-ty)}^*\xi\right)\\
=&\left([y,x]+\ldr{y^L}X_\xi(e)+\lie g_0,\ad_y^*\xi\right)\overset{\eqref{left_inv_der_of_Xmu}}=
\left([y,x]-\ad_\xi^*y+\lie g_0,\ad_y^*\xi\right)\\
=&\left[(y+\lie g_0,0),(x+\lie g_0,\xi)\right].
\end{align*}
\end{proof}

\section{Dirac homogeneous spaces}\label{section_Dirachomogeneous}
\subsection{Definition and properties}
Let $(G,\mathsf{D}_G)$ be a Dirac Lie group and $H$ a closed connected 
Lie subgroup of $G$. Let 
$G/H=\{gH\mid g\in G\}$ be the homogeneous space defined as the quotient space 
by the right action of $H$ on $G$. Let $q:G\to G/H$ be the quotient map.
For $g\in G$, let $\sigma_g:G/H\to G/H$ be the map
defined by $\sigma_g(g'H)=gg'H$. 

\begin{definition}
Let $(G,\mathsf{D}_G)$ be a Dirac Lie group and $H$ a closed connected 
Lie subgroup of $G$. Let 
$G/H$ be  endowed with a Dirac structure
$\mathsf{D}_{G/H}$.
The pair $(G/H, \mathsf D_{G/H})$ is a 
\emph{Dirac homogeneous space of $(G,\mathsf{D}_G)$} if 
the left action
\[\sigma:G\times G/H\to G/H, \quad \sigma_g(g'H)=gg'H\]
is a forward Dirac map, where $G\times G/H$
is endowed with the product Dirac structure 
$\mathsf D_G\oplus\mathsf D_{G/H}$.
\end{definition}

\begin{remark}
If $G/H$ is a homogeneous space of a Lie group $G$, there 
is an induced \emph{Lie groupoid action} of $TG\operp T^*G\rr \lie g^*$ 
on $J: T(G/H)\operp T^*(G/H)\to \lie g^*$,
$(v_{gH},\alpha_{gH})\mapsto (T_e(q\circ R_g))^*\alpha_{gH}$.
The Dirac manifold
$(G/H,\mathsf D_{G/H})$ is a Dirac homogeneous space of $(G,\mathsf D_G)$ if
and only if 
this groupoid action restricts to a Lie groupoid action of $\mathsf D_G\rr
\lie p_1$ on $J\an{\mathsf D_{G/H}}: \mathsf D_{G/H}
\to\lie p_1$.
This will be shown in \cite{Jotz10b}, where Dirac homogeneous spaces 
of Dirac Lie groupoids will be defined in this manner.
\end{remark}

\begin{remark}\label{remark_on_property_*}
The definition is also easily shown to be equivalent to the following:
for all $gH\in G/H$ and $(v_{gH},\alpha_{gH})\in \mathsf{D}_{G/H}(gH)$, there exist
$(w_g,\beta_g)\in \mathsf{D}_G(g)$ and $(u_{eH},\gamma_{eH})\in \mathsf{D}_{G/H}(eH)$ 
such that
\[\beta_g=(T_gq)^*(\alpha_{gH}), \quad
\gamma_{eH}=(T_{eH}\sigma_g)^*(\alpha_{gH}),
\quad \text{ and }\quad v_{gH}=T_gqw_g+T_{eH}\sigma_gu_{eH}.\]

This yields immediately: for all $h\in H$ and $(v_{eH},\alpha_{eH})\in
\mathsf{D}_{G/H}(eH)=\mathsf{D}_{G/H}(hH)$, there exist
$(w_h,\beta_h)\in \mathsf{D}_G(h)$ and $(u_{eH},\gamma_{eH})\in \mathsf{D}_{G/H}(eH)$ 
such that
$\beta_h=(T_hq)^*(\alpha_{eH})$, $\gamma_{eH}=(T_{eH}\sigma_h)^*(\alpha_{eH})$, 
and  $v_{eH}=T_hqw_h+T_{eH}\sigma_hu_{eH} $. 
\end{remark}
\begin{definition}\label{the_property}
Let $(G,\mathsf{D}_G)$ be a Dirac Lie group and $H$ a closed connected 
Lie subgroup of $G$.
We say that a subspace $S\subseteq \lie g/\lie h\times\left(\lie
  g/\lie h\right)^*$ has property $(*)$ if 
for all $h\in H$ and $(\bar x, \bar\xi)\in S$, there exist
$(w_h,\beta_h)\in \mathsf{D}_G(h)$ and $(\bar y,\bar \eta)\in S$ 
such that
$\beta_h=(T_hq)^*(\bar \xi)$, $\bar \eta=(T_{eH}\sigma_h)^*(\bar \xi)$, 
and  $\bar x=T_hqw_h+T_{eH}\sigma_h\bar y $.
\end{definition}

By Remark \ref{remark_on_property_*}, 
if $(G/H,\mathsf D_{G/H})$ is a Dirac homogeneous space of the Dirac Lie group
$(G,\mathsf D_G)$, then $\mathsf D_{G/H}(e)$ has property $(*)$.
This leads to the following lemma.
\begin{lemma}\label{lem18}
Let $\lie D_{G/H}$ be a Dirac subspace of $\lie g/\lie h\times(\lie g/\lie
h)^*$ with the property $(*)$.
Then the inclusions  $(T_eq)^*\bar{\lie p_1}\subseteq \lie
p_1$ and
$T_eq\lie g_0\subseteq\bar{\lie g_0}$ hold, where $\bar{\lie p_1}\subseteq
\left(\lie g/\lie h\right)^*$ and $\bar{\lie g_0}\subseteq \lie g/\lie h$
are the subspaces defined by $\lie{D}_{G/H}$.
\end{lemma}
\begin{proof}
Choose $\alpha\in\bar{\lie p_1}$, then there exists $v\in\lie g/\lie h$ such
that $(v,\alpha)\in\lie{D}_{G/H}$. By $(*)$, there exist $(w_e,\beta_e)\in
\mathsf{D}_G(e)$ and 
$(u_{eH},\gamma_{eH})\in \lie{D}_{G/H}$ 
such that
$\beta_e=(T_eq)^*\alpha$, $\gamma_{eH}=(T_{eH}\sigma_h)^*\alpha$, 
and  $v=T_eqw_e+u_{eH} $.
The covector $\beta_e=(T_eq)^*\alpha$ is an element of $\lie p_1$. The second
inclusion is a consequence of the first.
\end{proof}

We call in the following $\lie D\subseteq \lie g\times \lie g^*$ 
 the pullback Dirac subspace $\lie D=(T_eq)^*\lie D_{G/H}$ of 
$\lie D_{G/H}\subseteq \lie g/\lie h\times(\lie g/\lie
h)^*$ with Property $(*)$,
that is 
 \[\lie D= \big\{(x,\xi)\mid \exists \bar\xi\in(\lie g/\lie h)^* \text{ such that }
(T_eq)^*\bar \xi=\xi
\text{ and } (T_eqx,\bar\xi)\in \lie D_{G/H}
\big\}.
\]

\begin{lemma}\label{lemma_on_lieD}
Let $\lie p_1'\subseteq \lie g^*$ and $\lie g_0'\subseteq \lie g$ 
be the vector subspaces 
associated to the Dirac subspace $\lie D\subseteq \lie g\times \lie g^*$. Then we have
the inclusions
\[\lie g_0+\lie h\subseteq \lie g_0'\quad\text{ and}\quad
\lie p_1'\subseteq\lie p_1\cap\lie h^\circ.\]
Hence, we have  $\lie g_0\times\{0\}\subseteq\lie D \subseteq \lie g\times\lie p_1$
and the vector space $\bar{\lie D}:=\lie D/(\lie g_0\times\{0\})$ can be
seen as  a subset of $\lie g/\lie g_0\times\lie p_1$.
\end{lemma}

\begin{proof}
We know from Lemma \ref{lem18} 
that $T_eq\lie g_0\subseteq \bar{\lie g_0}$ and 
$(T_eq)^*\bar{\lie p_1}\subseteq\lie p_1$.
The inclusions here follow directly from this and the definition of $\lie D$.
\end{proof}

As in the case of a Dirac Lie group, the distributions 
$\mathsf{G_0}$ and $\mathsf{P_1}$
are regular. This fact is proved in the next proposition. In order to simplify the
notation,
we write $\mathsf{G_0}$ and $\mathsf{P_1}$, respectively, for both the
distributions, respectively codistributions, defined by $\mathsf D_G$ on $G$
and by $\mathsf D_{G/H}$ on $G/H$. It will always be clear from the context
which object is to be considered.
\begin{proposition}\label{shape_of_gnul_and_pone}
Let $(G/H, \mathsf{D}_{G/H})$ be a Dirac homogeneous space of
$(G,\mathsf{D}_G)$. Then the distribution $\mathsf{G_0}$ defined by 
$\mathsf D_{G/H}$ on $G/H$ is a subbundle of
 $T(G/H)$. 
Consequently, the codistribution $\mathsf{P_1}$ also has constant
rank on $G/H$.
More explicitly, the distribution $\mathsf{G_0}$ and the  codistribution
$\mathsf{P_1}$ are given by
\[\mathsf{G_0}(gH)=T_{eH}\sigma_g\mathsf{G_0}(eH)
\quad \text{ and } \quad \mathsf{P_1}(gH)=(T_{gH}\sigma_{g\inv})^*\mathsf{P_1}(eH)
\]
for all $gH\in G/H$.
\end{proposition}

\begin{proof}
We show that $\mathsf
{P_1}(gH)=(T_{gH}\sigma_{g\inv})^*\mathsf {P_1}(eH)$ for all $g\in G$: choose first
$\bar \xi\in\mathsf{P_1}(eH)$, then there exists
$\bar x\in T_{eH}(G/H)$ such that $(\bar
x,\bar\xi)\in\mathsf{D}_{G/H}(eH)$
and hence
$w_{g\inv}\in T_{g\inv}G$ and $u_{gH}\in T_{gH}(G/H)$ such that
$(w_{g\inv},(T_{g\inv}q_H)^*\bar\xi)\in\mathsf D_G(g\inv)$, 
$(u_{gH},(T_{gH}\sigma_{g\inv})^*\bar\xi)\in\mathsf D_{G/H}(gH)$ and 
$T_{g\inv}q_Hw_{g\inv}+T_{eH}\sigma_{g\inv}u_{gH}=\bar x$. This yields immediately 
$(T_{gH}\sigma_{g\inv})^*\mathsf {P_1}(eH)\subseteq \mathsf P_1(gH)$.
The other implication is a direct consequence of Remark \ref{remark_on_property_*}.

Thus, the codistribution $\mathsf{P_1}$ is a subbundle of $T^*(G/H)$ and its
annihilator is equal to $\mathsf{G_0}$, which is consequently given by
$\mathsf{G_0}(gH)=T_{eH}\sigma_g\mathsf{G_0}(eH)$ for all $g\in G$.
\end{proof}

Finally, we see that the notion of Dirac homogeneous spaces generalizes the
Poisson homogeneous spaces. The proof can be done  as in
Example \ref{eqDirac-Poisson}.
\begin{example}
Let $(G,\pi_G)$ be a Poisson Lie group, $H$ a closed subgroup of $G$ and $\pi$ a
Poisson bivector on $G/H$. Let $(G,\mathsf{D}_G)$ and $(G/H,\mathsf{D}_{G/H})$ be the
Dirac Lie group and Dirac manifold induced by $(G,\pi_G)$ and
$(G/H,\pi)$, respectively.
The Dirac manifold $(G/H,\mathsf{D}_{G/H})$ is a Dirac homogeneous space of
$(G,\mathsf{D}_G)$ if and only if $(G/H,\pi)$ is a Poisson homogeneous space
of $(G,\pi_G)$. 
\end{example}

\subsection{The pullback to $G$ of a homogeneous Dirac structure}\label{subsectionD'}
 Consider a Dirac Lie group $(G,\mathsf D_G)$ and let $\lie D$ be a Dirac subspace of
$\lie g\times\lie g^*$ satisfying \[\lie g_0\times\{0\}\subseteq\lie D\subseteq
\lie g\times\lie p_1.\quad \quad (**) \]  We denote by $\lie g_0'\subseteq \lie g$ and $\lie
p_1'\subseteq \lie g^*$ the subspaces defined by $\lie D$. We have 
$\lie g_0\subseteq \lie g_0'$ and $\lie p_1'\subseteq \lie p_1$
and we can define the generalized distribution  $\mathsf D'\subseteq\mathsf P_G$  by
\begin{equation}\label{definition_of_D'}
\mathsf D'(g):=\left\{\left.(x^L(g)+v_g,\xi^L(g))\in \mathsf P_G(g)\right| (x,\xi)\in\lie D 
\text{ and } (v_g,\xi^L(g))\in\mathsf D_G(g)\right\}
\end{equation}
for all $g\in G$. Note that $\mathsf D'$
is smooth since
it is  spanned by the smooth sections  $(X_\xi+x^L,\xi^L)$, with $(x,\xi)\in\lie D$
and $X_\xi\in\mx(G)$  such that $(X_\xi,\xi^L)\in\Gamma(\mathsf D_G)$.

\begin{proposition}\label{defD'}
Let $(G,\mathsf D_G)$ be a Dirac Lie group and $\lie D$  a Dirac subspace of
$\lie g\times\lie g^*$ satisfying $(**)$.
The induced subset $\mathsf{D}'\subseteq TG\operp T^*G$ as in \eqref{definition_of_D'} is
a Dirac structure on $G$. 
\end{proposition}
The construction of the Dirac structure $\mathsf D'$ is inspired by
\cite{DiMe99}.

Note that the codistribution $\mathsf{P_1}'$ induced by 
$\mathsf{D}'$ on $G$ is equal to
$\mathsf{P_1}'={\lie p_1'}^L$ by definition
 and consequently the distribution 
$\mathsf{G_0}'$ induced by 
$\mathsf{D}'$ on $G$ is equal to
$\mathsf{G_0}'$.
We have 
$\mathsf{G_0}\subseteq\mathsf{G_0}'$ and $\mathsf{P_1}'\subseteq \mathsf{P_1}$.

\begin{proof}
Choose $g\in G$ and   
$(x^L(g)+v_g,\xi^L(g))$, $(y^L(g)+w_g,\eta^L(g))\in\mathsf{D}'(g)$, 
i.e., with $(x,\xi),(y,\eta)\in\lie D$ and $(v_g,\xi^L(g)), (w_g,\eta^L(g))\in\mathsf D_G(g)$. We
have \begin{align*}
\langle \left(x^L(g)+v_g,\xi^L(g)\right), \left(y^L(g)+w_g,\eta^L(g)\right)\rangle
=\langle \left(v_g,\xi^L(g)\right), \left(w_g,\eta^L(g)\right)\rangle+\langle (x,\xi),(y,\eta)\rangle=0,
\end{align*} 
where we have used the equalities $\mathsf{D}_G=\mathsf{D}_G^\perp$ and 
$\lie D=\lie D^\perp$. This shows $\mathsf D'(g)\subseteq \mathsf D'(g)^\perp$.

Conversely, let $(u_g,\gamma_g)\in\mathsf P_G(g)$ be an element of 
$\mathsf{D}'(g)^\perp$.
Then we have 
$\gamma_g(x^L(g))=\langle (x^L(g),0),(u_g,\gamma_g)\rangle=0$
for all  $x\in\lie g_0'$, and 
$\gamma_g$ is thus an element of $\mathsf{P_1}'(g)\subseteq\mathsf{P_1}(g)$.
Choose $v_g\in T_gG$ such that
$(v_g,\gamma_g)\in\mathsf D_G(g)$ and  
set $w_g=u_g-v_g$. Then we get for any $(x,\xi)\in\lie D$:
\begin{align*}
\langle (w_g,\gamma_g), (x^L,\xi^L)(g)\rangle
=\gamma_g(x^L(g))+\xi^L(g)(w_g)
=\gamma_g(x^L(g))+\xi^L(g)(w_g)+\xi^L(g)(v_g)+\gamma_g(X_\xi(g))
\end{align*}
where $X_\xi\in\mx(G)$ is such that $(X_\xi,\xi^L)$ is a section of 
$\mathsf{D}_G$ that is defined at $g$. 
We have used the identity $\xi^L(v_g)+\gamma_g(X_\xi(g))=0$
which holds
because $(X_\xi,\xi^L)\in\Gamma(\mathsf D_G)$ and $(v_g,\gamma_g)\in\mathsf{D}_G(g)$.
But this equals
\begin{align*}
\langle (v_g+w_g,\gamma_g),(X_\xi+x^L,\xi^L)(g)\rangle
=\langle (u_g,\gamma_g),(X_\xi+x^L,\xi^L)(g)\rangle=0,
\end{align*}
since $(X_\xi+x^L,\xi^L)$ is by definition a section of $\mathsf{D}'$ and 
$(u_g,\gamma_g)\in\mathsf{D}'(g)^\perp$.
This shows that $(w_g,\gamma_g)\in\lie D^L(g)^\perp=\lie D^L(g)$. Hence, we have shown
$(u_g,\gamma_g)=(w_g+v_g,\gamma_g)\in\mathsf{D}'(g)$.
\end{proof}

\begin{remark}\label{rem} 
If $(Z,\alpha)$ is a section of $\mathsf{D}'$, then 
we have $(Z,\alpha)=(X_\alpha+Y_\alpha,\alpha)$ with $X_\alpha$ and
$Y_\alpha\in\mx(G)$ such that $(X_\alpha,\alpha)\in\Gamma(\mathsf{D}_G)$
and $(Y_\alpha,\alpha)\in\Gamma(\lie D^L)$. Hence, we have
\[(Z(e),\alpha(e))=(Y_\alpha(e),\alpha(e))+(X_\alpha(e),0)\in\lie D+(\lie
g_0\times\{0\})
=\lie D
\]
because $\mathsf{D}_G(e)=\lie g_0\times\lie p_1$ and $\lie
g_0\times\{0\}\subseteq \lie D$. Since $\lie D$ and $\mathsf D'(e)$ are
Lagrangian, this shows  that $\lie D=\mathsf D'(e)$.
\end{remark}

Let now $H$ be a closed subgroup of $G$ with Lie algebra $\lie h$, and denote 
by $q_H:G\to G/H$ the smooth surjective submersion. Let $\lie D_{G/H}\subseteq\lie
g/\lie h \times\left(\lie g/\lie h\right)^*$ be a Dirac subspace, such that
$\lie D\subseteq \lie g\times\lie g^*$ defined
by $\lie D=(T_eq_H)^*\lie D_{G/H}$ satisfies $(**)$.
Recall that property $(*)$ has been defined in Definition \ref{the_property}.
\begin{theorem}\label{th_on_D'}
The following are equivalent for $\lie D_{G/H}$  and $\lie D $ as above
and $\mathsf D'$ as in \eqref{definition_of_D'}.
\begin{enumerate}
\item $\lie{D}_{G/H}$ satisfies property $(*)$
\item $A_h\left(\lie D/(\lie g_0\times\{0\})\right)\subseteq
\lie D/(\lie  g_0\times\{0\})$ 
for all $h\in H$
\item $\mathsf D'$ is invariant under the right action of $H$ on $G$
\item $(G,\mathsf D')$ projects under $q_H$ to a Dirac homogeneous space
$(G/H,\mathsf D_{G/H})$ such that $\mathsf D_{G/H}(eH)=\lie D_{G/H}$.
\end{enumerate}
\end{theorem}

\begin{proof}
Assume first that $\lie D_{G/H}$ satisfies $(*)$ and choose $(x+\lie
g_0,\xi)\in\lie D/(\lie g_0\times\{0\})$. We have then $(x,\xi)\in\lie D$ and
hence there exists $\bar\xi\in\left(\lie g/\lie h\right)^*$ such that 
$(T_eq_H)^*\bar \xi=\xi$ and $(T_eq_H x,\bar\xi)\in\lie D_{G/H}$. By $(*)$ for
$(T_eq_H x,\bar\xi)\in\lie D_{G/H}$
and $h\inv\in H$, there
exists  $(w_{h\inv},\beta_{h\inv})\in\mathsf D_G(h\inv)$ and $(\bar
y,\bar\eta)\in\lie D_{G/H}$
such that $\beta_{h\inv}=(T_{h\inv}q_H)^*\bar\xi$,
$\bar\eta=(T_{eH}\sigma_{h\inv})^*\bar\xi$ and $T_eq_H x=T_{eH}\sigma_{h\inv}\bar y+T_{h\inv}q_Hw_{h\inv}$.
We compute
\begin{align*}
&(T_eL_{h\inv})^*\beta_{h\inv}=(T_{h\inv}q_H\circ T_eL_{h\inv})^*\bar\xi
=(T_eq_H\circ T_{h\inv}R_h\circ T_eL_{h\inv})^*\bar\xi
=\Ad_{h\inv}^*\xi\\
\text{ and also }\quad &
\Ad_{h\inv}^*\xi=(T_eL_{h\inv})^*\beta_{h\inv}=(T_{h\inv}q_H\circ
T_eL_{h\inv})^*\bar\xi
=(T_{eH}\sigma_{h\inv}\circ T_eq_H)^*\bar\xi=(T_eq_H)^*\bar\eta=:\eta.
\end{align*}
This yields also $\eta\in\lie p_1$, and
there exists a vector field $X_\eta\in\mx(G)$ such that
$(X_\eta,\eta^L)\in\Gamma(\mathsf D_G)$ and $X_\eta(h\inv)=w_{h\inv}$.
We have 
\[X_\eta(e)=X_{\Ad_{h\inv}^*\xi}(hh\inv)
\overset{\eqref{eq4}}=T_{h\inv}L_hX_{\Ad_{h\inv}^*\xi}(h\inv)+T_hR_{h\inv}X_\xi(h)+z
\]
with $z\in\lie g_0$.
We get
\begin{align*}
\bar y&=T_{eH}\sigma_hT_eq_H x-T_{eH}\sigma_hT_{h\inv}q_Hw_{h\inv}
=T_hq_HT_eL_hx-T_eq_H T_{h\inv}L_hX_{\eta}(h\inv)\\
&=T_eq_H\left(T_hR_{h\inv}T_eL_hx+T_hR_{h\inv}X_\xi(h)-X_{\Ad_{h\inv}^*\xi}(e)+z\right).
\end{align*}
Since $(\bar y,\bar\eta)$ is 
an element of $\lie D_{G/H}$, we have $(y,\eta)\in\lie D$ for any $y\in\lie g$ 
such that $T_eq_Hy=\bar y$.
Hence, the pair $(\Ad_hx+T_hR_{h\inv}X_\xi(h)-X_{\Ad_{h\inv}^*\xi}(e)+z, \eta)
=(\Ad_hx+T_hR_{h\inv}X_\xi(h)-X_{\Ad_{h\inv}^*\xi}(e)+z, \Ad_{h\inv}^*\xi)$ is
an element of $\lie D$. With $X_{\Ad_{h\inv}^*\xi}(e)+z\in\lie g_0$, this shows 
that $A(h,(x+\lie g_0,\xi))=(\Ad_hx+T_hR_{h\inv}X_\xi(h)+\lie g_0,
\Ad_{h\inv}^*\xi)
\in\lie D/(\lie g_0\times\{0\})$.

\medskip

Assume next that $\lie D/(\lie g_0\times\{0\})$ is $H$-invariant and choose a
spanning section $(X_\xi+x^L,\xi^L)$ of $\mathsf D'$, hence with
$(x,\xi)\in\lie D$ and $X_\xi\in\mx(G)$ a vector field satisfying
$(X_\xi,\xi^L)\in\Gamma(\mathsf D_G)$.
Since $(x+\lie g_0,\xi)$ is an element of $\lie D/(\lie g_0\times\{0\})$ we
get for an arbitrary $h\in H$
\[A_{h}(x+\lie g_0,\xi)
=\left(\Ad_hx+T_hR_{h\inv}X_\xi(h)+\lie g_0,\Ad_{h\inv}^*\xi\right)\in
\lie D/(\lie g_0\times\{0\}),
\]
and hence 
\begin{equation}\label{eq18}
\left(\Ad_hx+T_hR_{h\inv}X_\xi(h),\Ad_{h\inv}^*\xi\right)\in\lie D.
\end{equation}
Then we can compute for $g\in G$:
\begin{align*}
&\left(R_h^*(X_\xi+x^L)(g),R_h^*(\xi^L)(g)\right)
=\left(T_{gh}R_{h\inv}T_eL_{gh}x+T_{gh}R_{h\inv}X_\xi(gh),
\xi\circ T_{gh}L_{h\inv g\inv}\circ T_gR_h\right)\\
&\hspace*{2cm}\overset{\eqref{eq4}}=\left(T_eL_g\Ad_hx+T_{gh}R_{h\inv}(T_gR_hX_{\Ad_{h\inv}^*\xi}(g)
+T_hL_gX_\xi(h)+T_gR_{h}T_eL_gz),
(\Ad_{h\inv}^*\xi)^L(g)
\right)
\end{align*}
for some $z\in\lie g_0$ by Lemma \ref{too_useful}.
Thus, we get
\begin{align*}
\left(R_h^*(X_\xi+x^L)(g),R_h^*(\xi^L)(g)\right)
=&\left((\Ad_hx)^L(g)+X_{\Ad_{h\inv}^*\xi}(g)
+\left(T_hR_{h\inv}X_\xi(h)\right)^L(g)+z^L(g),
(\Ad_{h\inv}^*\xi)^L(g)
\right)\\
=&\left(\left(\Ad_hx+T_hR_{h\inv}X_\xi(h)\right)^L(g)+X_{\Ad_{h\inv}^*\xi}(g),
(\Ad_{h\inv}^*\xi)^L(g)
\right)+(z^L,0)(g).
\end{align*}
By the definition  of $\mathsf{D}'$ and \eqref{eq18}, we get consequently that
$\left(R_h^*(X_\xi+x^L)(g),R_h^*(\xi^L)(g)\right)\in\mathsf{D}'(g)$
(note that $z^L$ is a section of $\mathsf{G_0}\subseteq \mathsf{G_0'}$),
and hence that the right-action of $H$ on $(G,\mathsf{D}')$ is 
canonical.

\medskip

Assume that the right action of $H$ on $(G,\mathsf D')$ is canonical.
The vertical space $\V_H$ of the right action of $H$ on
$G$ is $\V_H=\lie h^L\subseteq \mathsf{G_0}'$ since by definition of $\lie D$
and $\lie g_0'$, we have $\lie h\subseteq \lie g_0'$. Thus, we have
$\mathsf{P_1}'\subseteq {\V_H}^\circ$ and hence, $\mathsf{D}'\cap
\K_H^\perp=\mathsf{D'}$ (recall the notations for \eqref{bucagured}). 
The reduced Dirac structure $\mathsf D_{G/H}$ is then given
by \[\mathsf D_{G/H}:=q_H(\mathsf D')=\left.\left(\frac{\mathsf{D}'+\K_H}{\K_H}\right)\right/H=
\left.\frac{\mathsf{D}'}{\K_H}\right/H.\]
We have to show that this defines a Dirac homogeneous space of $(G,\mathsf
D_G)$.
Note first that if $(\bar x,\bar\xi)\in\mathsf D_{G/H}(eH)$, then there exists 
$(x,\xi)\in\mathsf D'(e)=\lie D$ (see Remark \ref{rem})  such that $T_eq_H x=\bar x$ and
$(T_eq_H)^*\bar\xi=\xi$. But then $(\bar x,\bar \xi)$ is an element of $\lie
D_{G/H}$. The other inclusion can be shown in the same manner and we get
$\mathsf D_{G/H}(eH)=\lie D_{G/H}$.
Choose then $gH\in G/H$ and $(\bar v,\bar\alpha)\in \mathsf{D}_{G/H}(gH)$, that is,
$(\bar v,\bar\alpha)\in T_{gH}(G/H)\times
T_{gH}(G/H)^*$ such that there exists $v\in T_gG$ with $T_gq_H v=\bar
v$ and $(v,(T_gq_H)^*\bar\alpha)\in\mathsf{D}'(g)$. Then we can write $v$ as a sum
$v=w+u$ with $w,u\in T_gG$ such that $(w, (T_gq_H)^*\bar\alpha)\in\mathsf{D}_G(g)$ 
and $(u,(T_gq_H)^*\bar\alpha)\in\lie D^L(g)$, i.e.,
$(T_gL_{g\inv}u,(T_eL_g)^*\circ(T_gq_H)^*\bar\alpha)\in\lie D$. 
Since
$(T_eL_g)^*\circ(T_gq_H)^*\bar\alpha=(T_eq_H)^*(T_{eH}\sigma_g)^*\bar\alpha$,
we get
$(T_eq_H
T_gL_{g\inv}u,(T_{eH}\sigma_g)^*\bar\alpha)\in\lie {D}_{G/H}=\mathsf
D_{G/H}(eH)$. 
Set $\bar u:=T_eq_H
T_gL_{g\inv}u$, then we have $T_{eH}\sigma_g\bar u=T_gq_Hu$  and hence 
\[\bar v=T_gq_H w+ T_gq_H u=T_gq_H w+T_{eH}\sigma_g\bar u.\]

\medskip

The proof of the last implication $4\Rightarrow 1$ is given by Remark \ref{remark_on_property_*}.
\end{proof}

We have immediately the following corollary, which, together with the
preceding theorem, classifies the Dirac structures on $G/H$ that make
$(G/H,\mathsf D_{G/H})$ a Dirac homogeneous space of $(G,\mathsf D_G)$. 

\begin{corollary}
Let $(G,\mathsf D_G)$ be a Dirac Lie group, $H$ a closed Lie subgroup of $G$ and
$(G/H,\mathsf D_{G/H})$ a Dirac homogeneous space of $(G,\mathsf D_G)$. The
Dirac  structure $\mathsf D_{G/H}$ on $G/H$ is then uniquely determined by $\mathsf
D_{G/H}(eH)$ and $(G,\mathsf D_G)$. 
\end{corollary}

\begin{proof}
Since $(G/H,\mathsf D_{G/H})$ is a Dirac homogeneous space of $(G,\mathsf D_G)$,
the subspace $\mathsf D_{G/H}(eH)$ satisfies $(*)$ by Remark
\ref{remark_on_property_*}, and  $\lie D=T_eq_H^*\mathsf D_{G/H}(eH)$ satisfies
$(**)$ by Lemma \ref{lemma_on_lieD}. Define
$\mathsf D'$ as above.
Then, by the preceding theorem, we get that $\mathsf D'$ is right
$H$-invariant and projects under $q_H$ to a Dirac structure $q_H(\mathsf
D')$. It is easy to check that $ q_H(\mathsf
D')=\mathsf D_{G/H}$.
\end{proof}

\begin{remark}
\begin{enumerate}
\item 
Since the vertical space of the right action of $H$ on $(G,\mathsf D')$
denoted here by $\mathcal V_H$ is equal to $\lie h^L$ and hence contained in 
$\mathsf{G_0}'$, the Dirac structure $\mathsf D'$ is the backward Dirac image
of $\mathsf D_{G/H}$ defined on $G$ by $q_H$ 
(see subsection \ref{subsection_dirac_structures}).
\item The quotient $\lie D/(\lie g_0\times\{0\})$ is easily shown to be a
  Lagrangian subspace  of $\lie g/\lie g_0\times\lie p_1$ if and only if
  $\lie D_{G/H}$ is a Lagrangian subspace of $\lie g/\lie h\times\left(\lie
    g/\lie h\right)^*$ satisfying $(**)$.
\end{enumerate}
\end{remark}

\begin{corollary}\label{closed}
The Dirac homogeneous space $(G/H,\mathsf{D}_{G/H})$ is integrable if and only if  
the smooth Dirac manifold $(G,\mathsf{D}')$ defined by
$\mathsf{D}_{G/H}(eH)$ as in \eqref{definition_of_D'} is integrable.
\end{corollary}

\begin{proof}It is known by the theory of Dirac reduction that if an integrable
  Dirac manifold $(M,\mathsf D)$ is acted upon in a free and proper canonical way
  by a Lie group $H$, then the quotient Dirac manifold $(M/H,q_H(\mathsf D))$
  is also integrable. Hence, if $(G,\mathsf D')$ is integrable, then $(G/H,\mathsf
  D_{G/H})$ is also integrable (see for instance \cite{JoRa08}).

For the converse implication, we deduce from the proof of Theorem \ref{th_on_D'}  that
$(\ldr{\xi^L}X,\ldr{\xi^L}\alpha)$ is an element  of $\Gamma(\mathsf{D}')$ for
all sections $(X,\alpha)$ of $\mathsf{D}'$ and Lie algebra elements
$\xi\in\lie h$.
This yields that $\mathsf{D'}=\mathsf{D'}\cap\K_H^\perp$ satisfies 
\[[\Gamma(\K_H),\Gamma(\mathsf{D}')]\subseteq \Gamma(\mathsf{D}'+\K_H).\]
We get from a result  in \cite{JoRaZa11} that $\mathsf{D}'$ is spanned by
\emph{right $H$-descending  sections} $(X,\alpha)\in\Gamma(\mathsf{D}')$, that is, with
$[X,\Gamma(\V_H)]\subseteq \Gamma(\V_H)$ and $\alpha\in\Gamma({\V_H}^\circ)^H$.  Hence, it
suffices to show that if $(X,\alpha)$ and $(Y,\beta)$ are such elements of
$\Gamma(\mathsf{D}')$, then their bracket $[(X,\alpha),(Y,\beta)]$ is a
section of $\mathsf{D'}$.

Since $(X,\alpha)$ and $(Y,\beta)$ are $H$-descending and
$(G/H,\mathsf{D}_{G/H})$ is the Dirac quotient space of $(G,\mathsf{D}')$, 
we find $(\bar X,\bar \alpha)$ and $(\bar Y,\bar
\beta)\in\Gamma(\mathsf{D}_{G/H})$
such that $X\sim_{q_H}\bar X$, $Y\sim_{q_H}\bar Y$, $\alpha=q_H^*\bar\alpha$ and $\beta=q_H^*\bar\beta$.

We have then $[X,Y]\sim_{q_H}[\bar X,\bar Y]$ and
$\ldr{X}\beta-\ip{Y}\dr\alpha=q_H^*(\ldr{\bar X}\bar\beta-\ip{\bar
  Y}\dr\bar\alpha)$. If $(G/H,\mathsf{D}_{G/H})$ is integrable, the pair
$[(\bar X,\bar\alpha),(\bar Y,\bar\beta)]=([\bar X,\bar Y], \ldr{\bar X}\bar\beta-\ip{\bar
  Y}\dr\bar\alpha)$ is a section of $\mathsf{D}_{G/H}$. By construction of the
Dirac quotient of a Dirac manifold by a smooth Dirac action, there exists a
smooth vector field $Z\in\mx(G)$ such that $(Z,q_H^*(\ldr{\bar X}\bar\beta-\ip{\bar
  Y}\dr\bar\alpha))$ is an element of $\Gamma(\mathsf{D}')$ and $Z\sim_{q_H}[\bar
X,\bar Y]$. But then there exists a smooth section $V\in\Gamma(\V_H)=\Gamma(\lie
h^L)$ such that $Z+V=[X,Y]$. Since $\lie h^L\subseteq \mathsf{G_0}'$, this
yields 
\[[(X,\alpha),(Y,\beta)]=([X,Y],
\ldr{X}\beta-\ip{Y}\dr\alpha)=(Z,\ldr{X}\beta-\ip{Y}\dr\alpha)+(V,0)
\in\Gamma(\mathsf{D}')\]
and thus the Dirac manifold $(G,\mathsf{D}')$ is integrable.
\end{proof}

\begin{remark}\label{action_of_J_on_lieD}
\begin{enumerate}
\item If $(G/H,\mathsf D_{G/H})$ is integrable, the Dirac structure $\mathsf D'$ is also 
integrable as we
have seen above and the subbundle $\mathsf{G_0}'={\lie g_0'}^L$ is integrable
in the sense of Frobenius. The vector subspace $\lie g_0'\subseteq \lie g$ is
then a subalgebra and the integral leaf of $\mathsf{G_0}'$ through $e$ is a
Lie  subgroup of $G$, which will be called $J$ in the following. 
As in the proof of Lemma \ref{lemma_on_action_of_G_0}, we can show
  that the right action of $J$ on $G$ is canonical on $(G,\mathsf D')$. 
Theorem \ref{th_on_D'} yields then
$$A_{j}\left(\lie D/(\lie g_0\times\{0\})\right)\subseteq \lie D/(\lie
g_0\times\{0\})$$
for all $j\in J$.
\item  We can also show that if $(G,\mathsf D_G)$ is $N$-invariant, then
  $(G,\mathsf D')$ is $N$-invariant.
By
Theorem \ref{action_of_N},
the bracket on $\lie p_1\times \lie p_1$ defined in Definition
  \ref{defbracket}
has image in $\lie p_1$, and by Remark \ref{action_of_GmodN}, we know then
that the action of $N$ on $\lie g/\lie g_0\times\lie p_1$ is trivial.
Hence, we have
$A_n(\lie D/(\lie g_0\times\{0\}))=\lie D/(\lie g_0\times\{0\})$ for all
$n\in N$, and we can apply Theorem \ref{th_on_D'}.
\end{enumerate}
\end{remark}

\subsection{Integrable Dirac homogeneous spaces}

We consider here an \emph{integrable} Dirac Lie group $(G,\mathsf{D}_G)$, a closed Lie subgroup
$H$ of $G$ (with Lie algebra $\lie h$) and a Dirac homogeneous space 
$(G/H,\mathsf{D}_{G/H})$ of
$(G,\mathsf{D}_G)$.
As above, we consider the backward image $\lie D=(T_eq_H)^*\mathsf{D}_{G/H}(eH)$ of
$\mathsf{D}_{G/H}(eH)$ under $T_eq_H$, i.e.,
\[\lie D=\big\{(x,(T_eq)^*\bar\xi)\mid  x\in\lie g, \bar
\xi\in(\lie g/\lie h)^*\text{ such that } 
(T_eq x,\bar\xi)\in\mathsf{D}_{G/H}(eH)\big\},
\] and the Dirac structure $\mathsf D'$ defined by $\lie D$ on $G$ as in
\eqref{definition_of_D'} and Proposition \ref{defD'}.

\begin{theorem}
The quotient $\lie D/(\lie g_0\times\{0\})$ is a subalgebra of 
$\lie g/\lie g_0\times\lie
p_1$ if and only if $(G,\mathsf D')$ (or equivalently $(G/H,\mathsf D_{G/H})$)
is integrable.
\end{theorem}

\begin{proof}
Choose  $(x,\xi),(y,\eta)$ in $\lie D$, then
the pairs $(X_\xi+x^L,\xi^L)$ and 
 $(X_\eta+y^L,\eta^L)$ are sections of $\mathsf{D}'$.
We have by  Proposition \ref{lemma},
Definition \ref{defbracket},  Proposition \ref{lie_der_of_Xmu}
and Remark \ref{bracket_of_Xmu}:
\begin{align}
\left[(X_\xi+x^L,\xi^L),(X_\eta+y^L,\eta^L)\right]
=&\left([X_\xi,X_\eta]+\ldr{x^L}X_\eta-\ldr{y^L}X_\xi+[x,y]^L,
\ldr{X_\xi}\eta^L-\ip{X_\eta}\dr\xi^L+\ldr{x^L}\eta^L-\ip{y^L}\dr\xi^L
\right)\nonumber\\
\overset{\eqref{left_inv_der_of_Xmu}}= &\left(X_{[\xi,\eta]}+X_{\ad_x^*\eta}-(\ad^*_\eta
  x)^L-X_{\ad_y^*\xi}+(\ad_\xi^*y)^L+[x,y]^L,
([\xi,\eta]+\ad_x^*\eta-\ad^*_y\xi)^L\right)\nonumber\\
&+(Z,0) \quad \text{ for some }Z\in\Gamma(\mathsf{G_0})\nonumber\\
=&\left(X_{[\xi,\eta]+\ad_x^*\eta-\ad_y^*\xi}+([x,y]-\ad^*_\eta
  x+\ad_\xi^*y)^L,([\xi,\eta]+\ad_x^*\eta-\ad^*_y\xi)^L\right),\label{eq1000}
\end{align}
where we have chosen the 
vector field $X_{[\xi,\eta]+\ad_x^*\eta-\ad_y^*\xi}:=X_{[\xi,\eta]}-X_{\ad_y^*\xi}-X_{\ad_y^*\xi}+Z$.

If $(G,\mathsf D')$ is integrable,
we have 
$\left[(X_\xi+x^L,\xi^L),(X_\eta+y^L,\eta^L)\right]\in\Gamma(\mathsf{D}')$,
and hence its value at the neutral element $e$ is an element of $\lie D$
by Remark \ref{rem}.
But since $X_{[\xi,\eta]+\ad_x^*\eta-\ad_y^*\xi}(e)$ is an element of $\lie
g_0$, \eqref{eq1000} yields 
\begin{align*}
\left[(X_\xi+x^L,\xi^L),(X_\eta+y^L,\eta^L)\right](e)
\in\left([x,y]-\ad_\eta^*x+\ad_\xi^*y,[\xi,\eta]+\ad_x^*\eta-\ad_y^*\xi\right)
+(\lie g_0\times\{0\}).
\end{align*}
This leads to
\begin{align*}
[(x+\lie g_0,\xi),(y+\lie g_0,\eta)]
=([x,y]-\ad_\eta^*x+\ad_\xi^*y+\lie g_0,
[\xi,\eta]+\ad_{x}^*\eta-\ad_{y}^*\xi)\in\lie D/(\lie
g_0\times \{0\}).
\end{align*}

For the converse implication, it is sufficient to show that for all
$(x,\xi), (y,\eta)\in\lie D$, we have
\[\left[(X_\xi+x^L,\xi^L),(X_\eta+y^L,\eta^L)\right]\in\Gamma(\mathsf{D}')
\]
since $\mathsf D'$ is spanned by these sections.
By hypothesis, we have
\begin{equation}\label{lieD} [(x+\lie g_0,\xi), (y+\lie g_0,\eta)]
=([x,y]-\ad_\eta^*x+\ad_\xi^*y+\lie g_0, [\xi,\eta]+\ad_x^*\eta-\ad_y^*\xi)\in
\lie D/(\lie g_0\times\{0\})\end{equation}
for all $(x,\xi),(y,\eta)\in \lie D$ and the claim follows using \eqref{eq1000}. 
\end{proof}

We have proved the following theorem which is a generalization of the theorem
in \cite{Drinfeld93}.
\begin{theorem}\label{drinfeld}
Let $(G,\mathsf{D}_G)$ be a Dirac Lie group and 
 $H$ a closed subgroup of $G$ with Lie algebra $\lie h$. The assignment 
\[\mathsf{D}_{G/H}\mapsto \lie D=(T_eq_H)^*\mathsf{D}_{G/H}(eH)
\] 
gives a one-to-one correspondence between $(G,\mathsf{D}_G)$-Dirac homogeneous
structures on $G/H$ and Dirac subspaces  $\lie D\subseteq \lie g\times\lie
g^*$ such that
\begin{enumerate}
\item $(\lie g_0+\lie h)\times\{0\}\subseteq \lie D\subseteq \lie g\times(\lie
p_1\cap\lie h^\circ)$,
\item $\lie D/(\lie g_0\times\{0\})$ is Lagrangian in $\lie
  g/\lie g_0\times\lie p_1$, and 
\item $A_{h}\left(\lie D/(\lie g_0\times\{0\})\right)\subseteq 
\lie D/(\lie g_0\times\{0\})$ for all $h\in H$.
\end{enumerate}
If the Dirac Lie group is integrable, then $(G/H,\mathsf{D}_{G/H})$ is integrable if and
only if
$\lie D/(\lie g_0\times\{0\})$ is a \emph{subalgebra} of $\lie
  g/\lie g_0\times\lie p_1$.
\end{theorem}
\begin{example} \label{example}
\begin{enumerate}
\item
Let $(G,\mathsf D_G)$ be a Dirac Lie group and $(G/H,\mathsf
D_{G/H})$ a Dirac homogeneous space of $(G,\mathsf D_G)$
corresponding by Theorem \ref{drinfeld} to 
the Lagrangian subspace $\lie D\subseteq \lie g\times\lie
g^*$. Then, again by
Theorem \ref{drinfeld} applied to the Lie subgroup $\{e\}$ of $G$  and the Dirac
subspace $\lie D\subseteq \lie g\times\lie
g^*$, we get that $(G,\mathsf D')$ is a Dirac homogeneous space of $(G,\mathsf
D_G)$.
If $(G,\mathsf D_G)$ is integrable, we recover the fact that $(G,\mathsf D')$ 
is integrable if and only if $\lie D/(\lie
g_0\times\{0\})$ is a \emph{subalgebra} of $\lie g/\lie g_0\times\lie p_1$,
that is, if and only if $(G/H,\mathsf D_{G/H})$ is integrable.
\item Choose a Dirac Lie group $(G,\mathsf D_G)$ and assume that
the corresponding bracket on $\lie p_1$ has image in $\lie p_1$ and that
the Lie subgroup $N$ is closed in $G$. 
The Lagrangian subspace
 $\lie g_0\times\lie p_1$ of $\lie g\times\lie g^*$
satisfies $(**)$ and the corresponding Dirac structure $\mathsf D'$
is equal 
to $\mathsf D_G$ by definition. Since $N$ corresponds to the Lie subalgebra
$\lie g_0$ of $\lie g$ and fixes
$\lie g/\lie g_0\times\lie p_1$ pointwise
 by Remark \ref{action_of_J_on_lieD}, we get from 
Theorem \ref{th_on_D'}
that the quotient $(G/N,q_N(\mathsf{D_G}))$ is a Dirac homogeneous
space of the Dirac Lie group $(G,\mathsf D_G)$. We will study this particular 
Dirac homogeneous space in section \ref{special_case}.
\end{enumerate}
\end{example}

\begin{remark}
The previous theorem does not reduce, in the case of Poisson Lie groups,
to the same theorem but with $\lie g_0$ 
set to be $\{0\}$, as in many 
other statements of this work. Indeed, the theorem of Drinfel$'$d (\cite{Drinfeld93}) gives a correspondence
between \emph{Poisson} homogeneous structures on $G/H$ of a \emph{Poisson} Lie group 
$(G,\pois)$ and Lagrangian subalgebras $\lie D\subseteq \lie g\times\lie g^*$
satisfying $A_{h}\lie D\subseteq \lie D$ for all $h\in H$ and \emph{the equality}
$\lie D\cap(\lie g\times\{0\})=\lie h\times \{0\}$ (see \cite{Drinfeld93}),
that is, $\lie g_0'=\lie h$. 

Here, we have $\lie D\cap (\lie g\times\{0\})=\lie g_0'\times\{0\}$ and we get
the following cases.
\begin{enumerate}
\item $\lie g_0\subseteq \lie h=\lie g_0'$: the Dirac homogeneous space 
is a Poisson homogeneous space of the Dirac Lie group $(G,\mathsf{D}_G)$.
Furthermore, if $\lie g_0=\{0\}$, the Dirac Lie group is a Poisson Lie group
and we are in the situation of Drinfel$'$d's theorem. The case $\lie g_0=\lie
h$ (see the second part of Example
\ref{example}) will be studied in section \ref{special_case}.
\item  $\lie h\subsetneq \lie g_0'$: the Dirac homogeneous space has
non-trivial $\mathsf{G_0}$-distribution and is hence not a Poisson homogeneous
space of $(G,\mathsf{D}_G)$. 
Therefore, in the case where $\lie g_0=\{0\}$, we obtain
a Dirac homogeneous space of a 
Poisson Lie group.
\end{enumerate}
\end{remark}

\begin{example}
Consider an $n$-dimensional torus $G:=\mathbb{T}^n$. 
In Corollary \ref{abelianDiracLie}, we have recovered the
fact  that the only
multiplicative Poisson structure on $\mathbb{T}^n$ is the trivial Poisson
structure $\pi=0$, that is $\mathsf D_\pi=\{0\}\operp T^*\mathbb{T}^n$.  
The Lie algebra structure on $\lie g\times\lie g^*$ is
given by $[(x,\xi),(y,\eta)]=([x,y],\ad_x^*\eta-\ad_y^*\xi)=(0,0)$ since the
Lie group $\mathbb{T}^n$ is Abelian. Hence, every Dirac subspace of $\lie
g\times\lie g^*$ is a Lagrangian subalgebra. Indeed, it is easy to verify that
each left invariant Dirac structure on $\mathbb{T}^n$ is an integrable Dirac homogeneous
space of the trivial Poisson Lie group $(\mathbb{T}^n,\pi=0)$.

In general, if $(G,\pi=0)$ is a trivial Poisson Lie group, the Lie algebra
structure on $\lie g\times\lie g^*$ is given by 
$[(x,\xi),(y,\eta)]=([x,y],\ad_x^*\eta-\ad_y^*\xi)$. The  $(G,\pi=0)$-homogeneous
 structures  on $G$ are here  
the left invariant Dirac structures
$\lie D^L$ on $G$. 
Hence, the integrable  homogeneous Dirac structures on $G$ are the left
invariant Dirac structures $\lie D^L$ such that
$\lie D$ is a subalgebra of $\lie g\times\lie g^*$. But
$\lie D$ is a subalgebra of $\lie g\times\lie g^*$ if and only if  
\[\left\langle
\right([x,y],\ad_x^*\eta-\ad_y^*\xi),(z,\zeta)\rangle=\xi([y,z])+\eta([z,x])+\zeta([x,y])=0
\]
for all $(x,\xi)$, $(y,\eta)$ and $(z,\zeta)\in \lie D$.
 We recover 
Proposition \ref{invariant_Dirac_closed}
about integrability of a  left-invariant Dirac structure on $G$, see also \cite{Milburn07}. 
\end{example}

\section{The Poisson Lie group induced as a Dirac homogeneous space of a
  Dirac Lie group if  $N$ is closed in  $G$}\label{special_case}
We will see in this
section that if $(G,\mathsf D_G)$ is an 
integrable Dirac Lie group, such that the leaf $N$ of
the involutive subbundle $\mathsf{G_0}$ through the neutral element $e$ is a
\emph{closed} normal subgroup of $G$, then the Lie bialgebra
$(\lie g/\lie g_0,\lie p_1)\simeq \left(\lie g/\lie g_0, \left(\lie g/\lie
  g_0\right)^*\right)$ arises from a natural multiplicative Poisson structure $\pi$
on the quotient $G/N$, that makes $(G/N,\pi)$ a Poisson homogeneous
space of $(G,\mathsf D_G)$.

\begin{theorem}\label{canonical_Poisson_Lie_group}
Let $(G,\mathsf{D}_G)$ be a Dirac Lie group such that
the bracket on $\lie p_1\times\lie p_1$ has image in $\lie p_1$
and $N$ is closed in $G$. 
The
reduced Dirac structure $\mathsf{D}_{G/N}=q_N(\mathsf D_G)$ on $G/N$
(that is a homogeneous Dirac structure of $(G,\mathsf D_G)$, see Example \ref{example}) is the graph of a
skew-symmetric
multiplicative bivector field $\pi$ on $G/N$
(as in Example \ref{exPoisson}).
If $(G,\mathsf D_G)$ is integrable, the quotient 
$(G/N,\mathsf{D}_{G/N})=:(G/N, \pi)$ is a Poisson Lie group, and the induced Lie bialgebra
$(\lie g/\lie g_0,\lie p_1)
\simeq \left(\lie g/\lie g_0,(\lie g/\lie g_0)^*\right)$  as in Remark \ref{bracket_of_Xmu}
is the Lie bialgebra 
defined by $(G,\mathsf D_G)$ as in Theorems \ref{cocycle} and \ref{liealgebra}.
\end{theorem}

Since each normal subgroup of a simply connected Lie group $G$ is 
closed (see \cite{HiNe91}), we have the following immediate corollary.
\begin{corollary}
Let  $(G,\mathsf{D}_G)$ be an integrable, simply connected Dirac Lie group. Then
$\mathsf{D}_G$ is the pullback Dirac structure defined on $G$ by $q_N:G\to G/N$ and
a multiplicative Poisson bracket on $G/N$.
\end{corollary}

\begin{proof}[of Theorem \ref{canonical_Poisson_Lie_group}]
 Since $\lie
g_0$ is an ideal in $\lie g$, the Lie subgroup $N$ is normal in
$G$. If it is closed in $G$, the left or right action of $N$ on $G$ is free and proper
and the reduced space $G/N$ is a Lie group. Let $q_{N}:G\to G/N$ be the
projection.

The vertical distribution $\V_N$ of the left (right) action of $N$ on $G$ is
the span
of the right-invariant vector fields $x^R$, for all $x\in\lie g_0$, that
is, $\V_N=\mathsf{G_0}$. This yields $\K_N=\V_N\operp\{0\}=\mathsf{G_0}\operp\{0\}$,
and hence $\K_N^\perp=TG\operp\mathsf{P_1}$. The intersection
$\mathsf{D}_G\cap\K_N^\perp$ is consequently equal to $\mathsf{D}_G$ 
and has constant rank on $G$.  
Recall that 
the set of smooth local sections of 
$\mathsf{D}_{G/N}$ is given by
\begin{equation*}
\left\{(\bar X,\bar\alpha)\in\mx(G/N)\times\Omega^1(G/N)\left|
\begin{array}{c} 
\exists
X\in\mx(G)\text{ such that } X\sim_{q_N}\bar X\\
\text{ and }
(X,q_N^*\bar\alpha)\in\Gamma(\mathsf{D}_G)
\end{array}\right\}.\right.
\end{equation*}

Since $N$ lets $(G,\mathsf D_G)$ invariant by Theorem \ref{action_of_N}
and is connected by definition, we have 
$[\Gamma(\K_N),\Gamma(\mathsf D_G)]\subseteq \Gamma(\mathsf D_G)$ and we get using 
a result in \cite{JoRaZa11} that $\mathsf D_G$ is spanned by its 
$N$-descending  sections, that is, the pairs $(X,\alpha)\in\Gamma(\mathsf D_G)$ with
$[X,\Gamma(\V_N)]\subseteq \Gamma(\V_N)$ and $\alpha\in\Gamma(\V_N^\circ)^{N}$ (see
\cite{JoRaSn11} or \cite{JoRaZa11}). 
 The vector subbundle 
 $\mathsf{P_1}=\V_N^\circ$ of $T^*G$ is
spanned by its descending sections since $\V_N$ is a smooth integrable subbundle 
of $TG$ (see \cite{JoRaZa11}), and
the push-forwards of the descending sections of $\V_N^\circ$ are exactly the
sections of  the cotangent
space $T^*(G/N)$ of $G/N$. 
Since $\mathsf{D}_G$ is spanned by its descending
sections $(X,\alpha)\in\Gamma(\mathsf D_G)$,
$\mathsf{P_1}$ is in particular spanned by descending sections
belonging to descending pairs in $\Gamma(\mathsf{D}_G)$. This shows that the
cotangent distribution $\bar{\mathsf{P_1}}$ defined by $\mathsf{D}_{G/N}$ on
$G/N$ is equal to $T^*(G/N)$, and 
$(T_gq_N)^*(T_{gN}^*G/N)=\mathsf{P_1}(g)$ for all $g\in G$.
This yields that $\mathsf{D}_{G/N}$ is the graph of a 
skew-symmetric
bivector field $\pi$
on
$G/N$. Thus, if we show that $(G/N,\mathsf{D}_{G/N})$ is a Dirac Lie group, we
will have simultaneously proved that $( G/N,\pi)$ is multiplicative
by Example \ref{eqDirac-Poisson}.

We thus show that $\mathsf{D}_{G/N}$ is multiplicative. Choose a product
$gNg'N=gg'N\in G/N$ and \linebreak
$(\bar v_{gg'N},\bar\alpha_{gg'N})\in \mathsf{D}_{G/N}(gg'N)$. Then there exists
a pair  $(v_{gg'},\alpha_{gg'})\in\mathsf{D}_G(gg')$ such that
$$T_{gg'}q_Nv_{gg'}=\bar v_{gg'N}\quad \text{ and }\quad 
(T_{gg'}q_N)^*\bar\alpha_{gg'N}=\alpha_{gg'}.$$
Since $\mathsf{D}_G$ is multiplicative, we can find
$w_g\in T_gG$ and $u_{g'}\in T_{g'}G$ such that
\[T_gR_{g'}w_g+T_{g'}L_gu_{g'}=v_{gg'}, \quad 
\left(w_g,(T_{g}R_{g'})^*\alpha(gg')\right)\in\mathsf{D}_G(g),
\quad \text{ and }\left(u_{g'},(T_{g'}L_{g})^*\alpha_{gg'}\right)\in\mathsf{D}_G(g').
\]
We have  $\gamma_{g'}:=(T_{g'}L_{g})^*\alpha_{gg'}\in\mathsf{P_1}(g')$ and 
$\beta_g:=(T_{g}R_{g'})^*\alpha_{gg'}\in\mathsf{P_1}(g)$
and hence, by the considerations above, 
there exist $\bar\beta_{gN}\in\bar{\mathsf{P_1}}(gN)$ and 
$\bar\gamma_{g'N}\in\bar{\mathsf{P_1}}(g'N)$ satisfying
   $(T_gq_N)^*\bar\beta_{gN}=\beta_g$ and $(T_{g'}q_N)^*\bar\gamma_{g'N}=\gamma_{g'}$.
By construction of $\mathsf{D}_{G/N}$, we have then $(T_gq_Nw_g,\bar\beta_{gN})\in
\mathsf{D}_{G/N}(gN)$ and 
$(T_{g'}q_Nu_{g'},\bar\gamma_{g'N})\in\mathsf{D}_{G/N}(g'N)$.

We compute
\begin{align*}
T_{g'N}L_{gN}T_{g'}q_Nu_{g'}+T_{gN}R_{g'N}T_{g}q_Nw_g
=T_{g'g}q_N\left(T_{g'}L_gu_{g'}+T_gR_{g'}w_g \right)
=T_{gg'}q_Nv_{gg'}=\bar v_{gg'N},
\end{align*}
\begin{align*}
(T_{g'}q_N)^*\left((T_{g'N}L_{gN})^*\bar\alpha_{gg'N}\right)
=(T_{g'}L_g)^*\left((T_{gg'}q_N)^*\bar\alpha_{gg'N}\right)
=(T_{g'}L_g)^*\,\alpha_{gg'}=\gamma_{g'}=(T_{g'}q_N)^*\bar\gamma_{g'N},
\end{align*}
and in the same manner
$(T_{g}q_N)^*\left((T_{gN}R_{g'N})^*\bar\alpha_{gg'N}\right)
=\beta_{g}=(T_{g}q_N)^*\bar\beta_{gN}$.
This leads to 
\[(T_{g'N}L_{gN})^*(\bar\alpha_{gg'N})=\bar\gamma_{g'N}
\quad 
\text{ and }
\quad
(T_{gN}R_{g'N})^*(\bar\alpha_{gg'N})=\bar\beta_{gN}
\]
since $q_N$ is a smooth surjective submersion. Hence, we have shown that
$(G/N,\mathsf{D}_{G/N})$ is a Dirac Lie group which we will sometimes 
write $(G/N,\pi)$ in the following since $\mathsf{D}_{G/N}$ is the graph of a multiplicative
skew-symmetric bivector field $\pi$ on $G/N$.

The last statement is obvious with the considerations above and Proposition
\ref{other_formula_for_the_bracket}.
\end{proof}

\bigskip

Furthermore, we can show that each Dirac homogeneous structure  on $G/H$, $H$ a
closed subgroup of $G$, can be assigned to a unique Dirac homogeneous space of 
the Poisson Lie
group $(G/N,\pi)$ if the product $N\cdot H$ remains closed in $G$. 

Let $(G/H,\mathsf D_{G/H})$ be a Dirac homogeneous space. We assume
that the Lie subgroup $N\cdot H$ (with Lie algebra $\lie g_0+\lie h$) is closed in $G$.
The Lie group $N$ acts by smooth left actions given by $n\cdot gH= ngH$
for all $n\in N$ and $g\in G$
on the homogeneous space $G/H$. This is well-defined since
if $g\inv g'\in H$, we have $g\inv n\inv ng'\in H$ and hence $ngH=ng'H$. 
It is easy to check that the quotient
of $G/H$ by the left action of $N$ is equal to the quotient of $G$ by the right action
of $N\cdot H$. 
Indeed, the class of $gH$ in $(G/H)/N$ is the set $\{ngH\mid n\in N\}=NgH$. But 
since $N$ is normal in $G$, this class is equal to $gNH$, which is the class of
the element $g\in G$ in the quotient by the right action of $N\cdot H$ on $G$.
Since $G/(N\cdot H)$ has the structure of a smooth regular quotient manifold and 
the maps $q_H$ and $q_{N\cdot H}$ are smooth surjective submersions,
the projection $q_{N,H}:G/H\to (G/H)/N$ is also a smooth surjective
submersion.

In the second diagram, we have $(G/N)/(NH/N)\simeq G/(N\cdot H)\simeq (G/H)/N$.
\begin{displaymath}\begin{xy}
\xymatrix{
N\times G\ar[d]_{\operatorname{Id}_N\times q_H}\ar[r]^{m\an{N\times G}}&G\ar[d]^{q_H}\\ 
N\times G/H\ar[r]&G/H
}
\end{xy}
\qquad\qquad
\begin{xy}
\xymatrix{ 
G\ar[d]_{q_{N}}\ar[rr]^{q_H}\ar[rrd]_{q_{NH}}&&G/H\ar[d]^{q_{N,H}}\\
G/N\ar[rr]_{q_{N,NH}}&& G/(N\cdot H)
}
\end{xy}
\end{displaymath}

We have the following theorem. We assume here for simplicity
that the Dirac Lie group $(G,\mathsf
D_G)$ is integrable, but analogous results can be shown for a Dirac Lie group that
is invariant under the action of the induced Lie subgroup $N$.
\begin{theorem}
Let $(G/H,\mathsf D_{G/H})$ be an integrable Dirac homogeneous space of the
integrable Dirac Lie
group $(G,\mathsf D_G)$ such that $N$ and $NH$ are closed in $G$. 

The Lie group $N$ acts smoothly on the left on
$(G/H,\mathsf D_{G/H})$ by Dirac actions, and the Lie group
$N\cdot H$ acts smoothly on the right on the Dirac manifold $(G,\mathsf D')$ 
by Dirac actions. 

The quotient Dirac structures  on
$G/(N\cdot H)\simeq (G/H)/N$  are equal and will be called $\mathsf D_{G/(NH)}$.
The pair $(G/(N\cdot H), \mathsf D_{G/(NH)})$
is a Dirac homogeneous space of the Poisson Lie group $(G/N,\pi)$ and of the
Dirac Lie group $(G,\mathsf D_G)$.

Conversely, if $(G/(NH),\mathsf D_{G/(NH)})$ is a Dirac homogeneous space of
the Poisson Lie group $(G/N,\pi)$, then the pullbacks 
$(G/H,q_{N,H}^*(\mathsf D_{G/(NH)}))$
and $(G,q_{NH}^*(\mathsf D_{G/(NH)}))$ are Dirac homogeneous spaces
of the Dirac Lie group $(G,\mathsf D_G)$.
\end{theorem}
\begin{proof}
Consider again the Dirac subspace $\lie D=(T_eq_H)^*\mathsf D_{G/H}(eH)\subseteq 
\lie g\times\lie p_1$. We write $\bar{\lie D}$ for the quotient 
$\lie D/(\lie g_0\times\{0\})$.  Since  $(G,\mathsf D_G)$ and $(G/H,\mathsf D_{G/H})$ are integrable,
we get from  Theorem \ref{drinfeld} that  $\bar{\lie D}$ is
a subalgebra of $\lie g/\lie g_0\times\lie p_1$ and from Remark \ref{action_of_J_on_lieD} that
$\bar{\lie D}$
is $N$-invariant. We have then 
$A_{nh}\bar{\lie D}\subseteq \bar{\lie D}$ for all $nh\in  N\cdot H$ and, by 
Theorem \ref{drinfeld}, the 
group $N\cdot H$ acts on $(G,\mathsf D')$ by Dirac actions and the quotient
$(G/(N\cdot H),q_{NH}(\mathsf D'))=:(G/(N\cdot H),\mathsf D_{G/NH})$ is an integrable 
Dirac homogeneous space of the Dirac Lie group $(G,\mathsf D_G)$.

Next we show that the left action $\Phi$
of $N$ on $(G/H,\mathsf D_{G/H})$ is canonical.
Let $(\bar X,\bar \alpha)$ be a section of $\mathsf D_{G/H}$. Then there exists
$(X,\alpha)\in\Gamma(\mathsf D')$ such that  $X\sim_{q_H}\bar X$ and 
$\alpha=q_H^*\bar\alpha$. We have $q_H\circ L_n=\Phi_n\circ q_H$
for all $n\in N$ and hence 
$L_n^*X\sim_{q_H}\Phi_n^*\bar X$ and 
$L_n^*\alpha=q_H^*\Phi_n^*\bar\alpha$. Since the action of $N$ on $(G,\mathsf
D')$ is canonical, we have $(L_n^*X,L_n^*\alpha)\in\Gamma(\mathsf D')$ and the pair
$(\Phi_n^*\bar X,\Phi_n^*\bar\alpha)$ is consequently a section of $\mathsf
D_{G/H}$.

Let $\V$ be the vertical space of the action $\Phi$ of $N$ on $G/H$ and
$\K=\V\operp\{0\}$. The subbundle $\V$ of $T(G/H)$  is
spanned by the projections to $G/H$ of the right-invariant vector fields $x^R$
on $G$, for all $x\in \lie g_0$, and $\V^\circ$ is spanned by the
push-forwards of the one-forms $\xi^R$, for all $\xi\in\lie p_1\cap\lie
h^\circ$. But since $\mathsf D'\cap\K_H^\perp=\mathsf D'\subseteq 
TG\operp (\lie p_1\cap\lie
h^\circ)^R$, and $\mathsf D_{G/H}=q_H(\mathsf D')$,
 we get easily
$\mathsf D_{G/H}\cap\K^\perp=\mathsf D_{G/H}$, which has consequently constant
dimensional fibers on $G/H$.
 Thus, by the regular reduction theorem for Dirac
manifolds, the quotient $((G/H)/N, q_{N,H}(\mathsf D_{G/H}))$ is a smooth Dirac manifold.

We have then to show that the quotient Dirac structure 
$q_{N,H}(\mathsf D_{G/H})$ is equal to $\mathsf D_{G/(NH)}$. 
If $(\tilde X,\tilde \alpha)$ is a section of $q_{N,H}(\mathsf D_{G/H})$, then there
exists $(\bar X,\bar \alpha)$ in $\Gamma(\mathsf D_{G/H})$ such that
$\bar X\sim_{q_{N,H}}\tilde X$ and $q_{N,H}^*\tilde\alpha=
\bar\alpha$. But then there exists $(X,\alpha)\in\Gamma(\mathsf D')$ such that
$X\sim_{q_H}\bar X$ and $\alpha=q_H^*\bar\alpha$. Then we have 
$\alpha=q_H^*q_{N,H}^*\tilde\alpha=q_{NH}^*\tilde\alpha$,
$X\sim_{q_{NH}}\tilde X$ and $(\tilde X,\tilde \alpha)$ is a section 
of $\mathsf D_{G/NH}$. This shows $q_{N,H}(\mathsf D_{G/H})\subseteq \mathsf D_{G/NH}$
and hence equality since both Dirac structures have the same rank.

Finally, we show that $(G/(N\cdot H),\mathsf D_{G/(NH)})$ is a Dirac
homogeneous space of the Poisson Lie group $(G/N,\pi)$. 
The Lie bialgebra of the Poisson Lie group $(G/N,\pi)$ is 
$(\lie g/\lie g_0, \lie p_1)$ with the bracket as in Theorem \ref{bracket}. 
We have 
$(T_{eN}q_{N,NH})^*\mathsf{D}_{G/(NH)}(eNH)= 
(T_eq_N)\left((T_eq_{NH})^*\mathsf{D}_{G/(NH)}(eNH)\right)=
\lie D/(\lie g_0\times\{0\})
\subseteq \lie g/\lie g_0\times\lie p_1$.
By Remark \ref{action_of_GmodN}, the action of $G$ on $\lie D/(\lie
g_0\times\{0\})$ induces an action of $G/N$ on $\lie D/(\lie
g_0\times\{0\})$; this is exactly the action of $G/N$ defined by the 
Poisson Lie group $(G/N,\pi)$ on its Lie bialgebra.
Since $\lie D/(\lie
g_0\times\{0\})$ is $NH$-invariant, it is $NH/N$-invariant under $\bar A$.
Since $\lie D/(\lie
g_0\times\{0\})$ is a Lagrangian subalgebra of $\lie g/\lie g_0\times\lie p_1$
and $(\lie g_0+\lie h)/\lie g_0\times\{0\}\subseteq  \lie D/(\lie
g_0\times\{0\})\subseteq \lie g/\lie g_0\times\lie h^\circ\cap\lie p_1 $, we are done
by Theorem \ref{drinfeld}.

\medskip

For the converse statement, we use Remark \ref{action_of_GmodN} about the
action
$\bar A$ of $G/N$ on $\lie g/\lie g_0\times\lie p_1$ and apply the first part
of Example \ref{example}
to the Dirac Lie group $(G/N,\pi)$ and the closed subgroup $NH/N$ of $G/N$ and
to the Dirac Lie group $(G,\mathsf D_G)$ and the closed subgroup $NH$ of $G$.
\end{proof}

Now choose an integrable Dirac homogeneous space $(G/H,\mathsf{D}_{G/H})$ of
$(G,\mathsf{D}_G)$ and let $(G,\mathsf D')$ be the Dirac structure on $G$ 
defined as in the preceding section. Since $\mathsf D'$ is integrable
and $\mathsf{G_0'}$ is left invariant, it is an involutive subbundle of $TG$
which is consequently integrable in the sense of Frobenius. Then the integral leaf
$J$ of $\mathsf{G_0}'$ through the neutral element $e$, which was defined
in Remark \ref{action_of_J_on_lieD},  is a Lie subgroup
of $G$. 
\begin{lemma}\label{Lemma54}
If the Lie subgroup $J$ is closed in $G$, it acts properly on the
right on $(G,\mathsf{D'})$ 
by Dirac actions. The intersection
$\mathsf D'\cap\K_J^\perp$, with $\K_J=\V_J\operp\{0\}={\lie g_0'}^L\operp\{0\}$, 
is equal to $\mathsf
D'$ by definition of $J$ and 
we can build the quotient $(G/J,q_J(\mathsf D'))$, where $q_J:G\to G/J$
is the projection.

Furthermore, since $N\subseteq J$ is a normal subgroup, we can build the
quotient Lie group 
$J/N$ if $N$ is closed in $G$. It acts properly on the right on  
$G/N$ and we can see that $G/J\simeq (G/N)/(J/N)$
as  a homogeneous space of $G/N$. 
\end{lemma}
\begin{proof}
We have seen in Remark \ref{action_of_J_on_lieD} that if $(G,\mathsf D')$ 
is integrable, we have 
$A_j\bar{\lie D}=\bar{\lie D}$ for all $j\in J$. 
By Theorem \ref{th_on_D'} (note that all
the hypotheses are satisfied  since 
$\lie g_0'\times\{0\}\subseteq\lie D\subseteq \lie
g\times\lie p_1'$, the Dirac subspace $\lie D$ is equal to the pullback
$\lie D=T_eq_J^*(T_eq_J\lie D)$
), we get that the Dirac manifold $(G,\mathsf D')$
is right $J$-invariant.
\end{proof}
In the following diagram, the dashed arrows join Dirac Lie groups to their
Dirac homogeneous spaces.

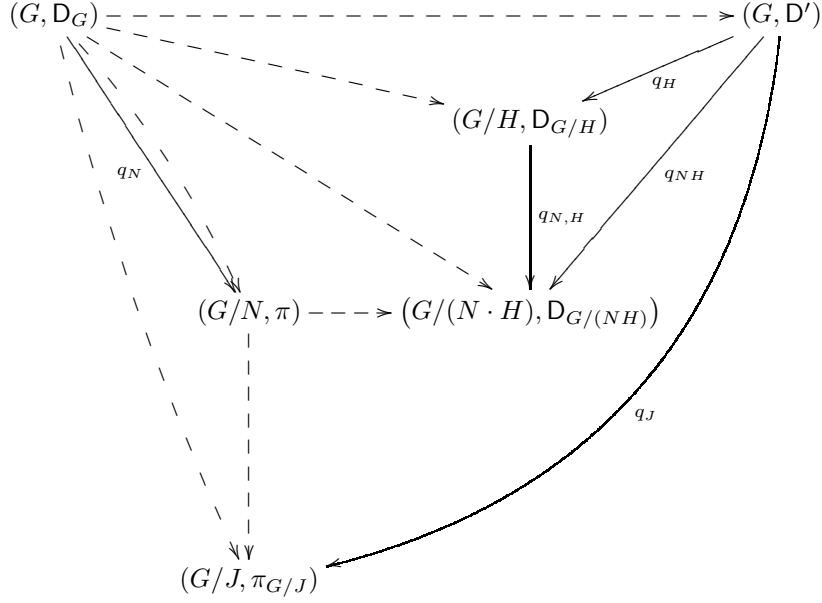
\begin{displaymath}\begin{xy}
\xymatrix{ 
 (G,\mathsf D_G)\ar@{-->}[rrr]
\ar@{-->}[rdr]\ar@{-->}@/_/[ddddddr]\ar@{-->}[dddrr]
\ar[dddr]_{q_N} \ar@{-->}@/^/[dddr]&&&
\ar[dl]^{q_H}(G,\mathsf D')\ar[dddl]^{q_{NH}}\ar@/^2cm/[ddddddll]^{q_J}\\
&&(G/H,\mathsf D_{G/H})\ar[dd]^{q_{N,H}}
&\\
&&&\\
&(G/N,\pi)\ar@{-->}[r]\ar@{-->}[ddd]&\left(G/(N\cdot H),\mathsf
  D_{G/(NH)}\right)&\\
&&&\\
&&&\\
&(G/J,\pi_{G/J})&
}
\end{xy}
\end{displaymath}

\begin{theorem}
Under the hypotheses of the preceding lemma,
the pair $(G/J,q_J(\mathsf{D}'))=:(G/J,\pi_{G/J})$ 
is a Poisson homogeneous space of the Dirac 
Lie group $(G,\mathsf D_G)$ and of the Poisson Lie group $(G/N,\pi)$.
\end{theorem}
\begin{proof}
By Theorem \ref{drinfeld} and Lemma \ref{Lemma54}, $(G/J,\pi_{G/J})$ is a Dirac homogeneous space of
the integrable Dirac 
Lie group $(G,\mathsf D_G)$. The quotient
$\lie D/(\lie g_0\times\{0\})$ is a Lagrangian subalgebra of $\lie g/\lie
g_0\times\lie p_1$ with
$(\lie g_0+\lie g_0')\times\{0\}=\lie g_0'\times\{0\}\subseteq \lie D
\subseteq \lie g\times\lie p_1'=\lie g\times((\lie g_0')^\circ\cap\lie p_1)$. Furthermore,
$\lie D/(\lie g_0\times\{0\})$ is $A_j$-invariant and hence also $\bar
A_{jN}$-invariant by 
Remark \ref{action_of_GmodN} for all $j\in J$ (see also the proof of the
preceding theorem). Hence, $(G/J,\pi_{G/J})$ is also 
 a Dirac homogeneous space of the Poisson Lie group $(G/N,\pi)$. 

Note that since
$\mathsf{G_0}'=\V_J$, the quotient Dirac structure $\mathsf D_{G/J}:=q_J(\mathsf{D}')$ has
vanishing characteristic distribution and is hence a Poisson manifold
(see the proof of Theorem \ref{canonical_Poisson_Lie_group})
\end{proof}

Finally, we give examples where it is not possible to build the diverse
quotients as above. Let $\widetilde{\operatorname{SL_2(\R)}}$ be the universal
covering of the Lie group $\operatorname{SL_2(\R)}$.

\begin{example}
\begin{enumerate}
\item Consider the Lie group
\[G=\left(\mathbb{T}^2\times \widetilde{\operatorname{SL_2(\R)}}\right)/\Gamma,
\]
where $\Gamma$ is the group homomorphism 
$Z(\widetilde{\operatorname{SL_2(\R)}})\simeq\mathbb{Z}\to\mathbb{T}^2$
given by $\Gamma(z)=
(e^{i\sqrt{2}z},e^{iz})$ for all $z\in Z(\widetilde{\operatorname{SL_2(\R)}})$
(or more generally a group homomorphism with dense image in $\mathbb{T}^2$).
The graph of $\Gamma$ is a discrete normal subgroup of 
 $\operatorname{\mathbb T^2}\times \widetilde{\operatorname{SL_2(\R)}}$ and hence
 the quotient $G $ is a Lie group.
 The Lie algebra of $G $ is equal to the direct sum of Lie algebras $\lie g 
 =\R^2\oplus\lie{sl}_2(\R)$ and has hence $\lie g_0:=\lie{sl}_2(\R)$ as an ideal.
 The corresponding Lie subgroup $N$ of $G $ corresponds to the images 
 of the elements of $\widetilde{\operatorname{SL_2(\R)}}$ in $G $ and is hence by construction
 not closed in $G $.
Let $G$ be endowed with the trivial Dirac structure such that $\lie
g_0=\lie{sl}_2(\R)$;  the quotient Poisson Lie group $(G/N,\pi)$
does not exist here.
\item Consider $G=\mathbb{T}^4\times \R$ with coordinates
$s_1,s_2,s_3,s_4,t$ and the integrable Dirac structure 
given by $\lie g_0=\erz\{x_1
\}$,  $\lie p_1=\erz\{ \xi_2,\xi_3,\xi_4,\xi_5\}$ and $\mathsf D_G$ the
(necessarily) trivial multiplicative Dirac structure 
\[\mathsf D_G=\lie g_0^L\operp\lie p_1^L=\erz\{
(x_1^L,0), (0,\xi_2^L), (0,\xi_3^L), (0,\xi_4^L), (0,\xi_5^L)
\},\]
where
\begin{align*}
\begin{array}{ll}
\xi_2=\sqrt{2}\dr s_1(0)-\dr s_2(0)  &\quad \quad 
x_1=\partial_{s_1}(0)+\sqrt{2}\partial_{s_2}(0)+\partial_t(0)\\
\xi_3=\dr s_1(0)-\dr t(0)&\quad\quad  x_2=-\partial_{s_2}(0)\\
\xi_4=\dr s_3(0)&\quad \quad x_3=-\partial_t(0)\\
\xi_5=\dr s_4(0)&\quad \quad x_4=\partial_{s_3}(0)\\
\xi_1=\dr s_1(0)&\quad \quad x_5=\partial_{s_4}(0).
\end{array}
\end{align*}
The group $N$ is then equal to $N=\{ (e^{ti},e^{\sqrt{2}ti},1,1,t)\mid t\in\R 
\}$ and is closed in $G$ as the graph of a smooth map $\R\to \mathbb{T}^4$. 
The quotient $(G/N,\pi)$ is a torus $\mathbb T^4$
with trivial Poisson Lie group structure. Consider the subgroup 
$H=\{ (e^{2ti},e^{2\sqrt{2}ti},1,1,t)\mid t\in\R \}$
of $G$. Then $H$ is closed in $G$ and each Dirac subspace $\lie D\subseteq
\R^5\times{\R^5}^*$ with $(\lie h+\lie g_0)\times\{0\}\subseteq \lie g\times(\lie
h^\circ\cap\lie p_1)$ induces a homogeneous Dirac  structure on
$G/H\simeq \mathbb{T}^4$. 

The subgroup $N\cdot H$ of $G$ is dense in
$\mathbb{T}^2\times\{1\}^2\times \R\subseteq G$ and is hence not closed in
$G$.

Consider now the Dirac subspace 
\[\lie D:=\erz\left\{\begin{array}{c}
(\partial_{s_1}(0),0),\quad 
(\partial_{s_2}(0),0),\quad 
(\sqrt{3}\partial_{s_3}(0)+\partial_{s_4}(0),0),\\
(\partial_{t}(0),0),\quad 
(X,\dr s_3(0)-\sqrt{3}\dr s_4(0))
\end{array}
\right\}\text{ of }\lie g\times\lie g^*
\]
with $X\in\lie g$ an arbitrary vector satisfying $(\dr s_3(0)-\sqrt{3}\dr
s_4(0))(X)=0$.
The Dirac structure $\mathsf D'=\lie D^L$ defines a $(G,\mathsf D_G)$-homogeneous Dirac 
structure of  since $\lie g_0+\lie h\subseteq \lie g_0'$, but
the leaf $J$ of $\mathsf{G_0}'$ through the neutral element $0$ is 
equal
to
$J=\{ (e^{\theta i},e^{\phi i},e^{\sqrt{3}ti},e^{ti},s)\mid
\theta,\phi,t,s\in\R \}$
and thus 
dense in $G$.
\end{enumerate}
\end{example}

\def\cprime{$'$} \def\polhk#1{\setbox0=\hbox{#1}{\ooalign{\hidewidth
  \lower1.5ex\hbox{`}\hidewidth\crcr\unhbox0}}} \def\cprime{$'$}
  \def\cprime{$'$} \def\cprime{$'$} \def\cprime{$'$} \def\cprime{$'$}
  \def\cprime{$'$} \def\cprime{$'$}
  \def\polhk#1{\setbox0=\hbox{#1}{\ooalign{\hidewidth
  \lower1.5ex\hbox{`}\hidewidth\crcr\unhbox0}}}
  \def\polhk#1{\setbox0=\hbox{#1}{\ooalign{\hidewidth
  \lower1.5ex\hbox{`}\hidewidth\crcr\unhbox0}}}
  \def\polhk#1{\setbox0=\hbox{#1}{\ooalign{\hidewidth
  \lower1.5ex\hbox{`}\hidewidth\crcr\unhbox0}}}
  \def\polhk#1{\setbox0=\hbox{#1}{\ooalign{\hidewidth
  \lower1.5ex\hbox{`}\hidewidth\crcr\unhbox0}}} \def\cprime{$'$}
  \def\polhk#1{\setbox0=\hbox{#1}{\ooalign{\hidewidth
  \lower1.5ex\hbox{`}\hidewidth\crcr\unhbox0}}}
  \def\polhk#1{\setbox0=\hbox{#1}{\ooalign{\hidewidth
  \lower1.5ex\hbox{`}\hidewidth\crcr\unhbox0}}}
  \def\polhk#1{\setbox0=\hbox{#1}{\ooalign{\hidewidth
  \lower1.5ex\hbox{`}\hidewidth\crcr\unhbox0}}}
  \def\polhk#1{\setbox0=\hbox{#1}{\ooalign{\hidewidth
  \lower1.5ex\hbox{`}\hidewidth\crcr\unhbox0}}}

\end{document}